\documentclass[10pt,a4paper]{article}
\usepackage[latin2]{inputenc}
\usepackage[english]{babel}
\usepackage{a4wide}

\makeatletter
\let\@fnsymbol\@arabic
\makeatother

\usepackage[english]{babel}
\usepackage[nottoc]{tocbibind}
\usepackage{amsmath,amsfonts,amstext,amsbsy,amsopn,amssymb,amsthm,amscd}
\usepackage{amsxtra,latexsym}
\usepackage{exscale}
\usepackage{paralist}
\usepackage[article]{math_env} 
\numberwithin{equation}{section}
\usepackage{dmath_abbr_wias}
\usepackage{graphicx,mathrsfs}
\usepackage{psfrag}
\usepackage{color}
\definecolor{dred}{rgb}{.8,0,0}
\definecolor{ddmagenta}{rgb}{0.7,0,0.9}
\definecolor{ddcyan}{rgb}{0,0.2,1.0}
\definecolor{Green}{rgb}{0.,0.5,0.}

\newcommand{\DDDS}{\color{black} }
\newcommand{\DDDE}{\color{black}}

\newcommand{\DelS}{\color{yellow} }
\newcommand{\DelE}{\color{black}}

\newcommand{\eps}{\varepsilon}
\newcommand{\umin}{u_{\mathrm{min}}}
\newcommand{\foraa}{\text{for a.a.}}
\newcommand{\weakto}{\rightharpoonup}

\def\weaksto{\stackrel{*}{\rightharpoonup}}

\newcommand{\piecewiseConstant}[2]{\overline{#1}_{\kern-1pt#2}}

\newcommand{\underpiecewiseConstant}[2]{\underline{#1}_{\kern-1pt#2}}

\newcommand{\pairing}[3]{\sideset{_{}}{_{ #1}}  {\mathop{\langle #2 , #3  \rangle}}}
\newcommand{\inner}[3]{\sideset{_{}}{_{ #1}}  {\mathop{( #2 , #3  )}}}

\newcommand{\BV}{\mathrm{BV}}

\newcommand{\mmu}{{\mbox{\boldmath$\mu$}}}

\newtheorem{notation}[theorem]{Notation}
\newcommand{\piecewiseLinear}[2]{#1_{\kern-1pt#2}}

\newcommand{\dd}{\,\mathrm{d}}

\newcommand{\mixed}[1]{\mathrm{M}_{#1}}
\newcommand{\DDD}[3]{\begin{array}[t]{c}#1\vspace*{-1em}\\_{#2}\vspace*{-.5em}\\_{#3}\end{array}}
\newcommand{\ddd}[3]{\DDD{\begin{array}[t]{c}\underbrace{#1}\vspace*{.6em}\end{array}}{\text{\footnotesize #2}}{\text{\footnotesize #3}}}
\newcommand{\epsi}{\epsilon}
\newcommand{\twodis}{\mathrm{d}_2}

\def\trait #1 #2 #3 {\vrule width #1pt height #2pt depth #3pt}
\def\fin{
    \trait .3 5 0
    \trait 5 .3 0
    \kern-5pt
    \trait 5 5 -4.7
    \trait 0.3 5 0
\medskip}

\newcommand{\il}{q}

\newenvironment{rcomm}{\color{black} }{\color{black}}
\newcommand{\beric}{\begin{rcomm}}
\newcommand{\eric}{\end{rcomm}}

\newenvironment{nrcomm}{\color{black} }{\color{black}}
\newcommand{\berin}{\begin{nrcomm}}
\newcommand{\erin}{\end{nrcomm}}

\newenvironment{newnewchanges}{\color{black}
}{\color{black}}
\newcommand{\bnnc}{\begin{newnewchanges}}
\newcommand{\ennc}{\end{newnewchanges}}
\begin{document}
\title{A quasilinear differential inclusion for  viscous and rate-independent damage systems in non-smooth domains}
\author{%
Dorothee Knees\thanks{
Weierstrass Institute for
Applied Analysis and Stochastics,  Mohrenstr.~39,
10117 Berlin, Germany,
 E-Mail: dorothee.knees@wias-berlin.de
}
\and
Riccarda Rossi\thanks{DICATAM -- Sezione di  Matematica,
University of Brescia,
Via Valotti 9,
25133 Brescia, Italy, E-Mail: riccarda.rossi@ing.unibs.it}
\and
Chiara Zanini\thanks{Department of Mathematical Sciences, Politecnico di
  Torino, Corso Duca degli Abruzzi 24, 10129 Torino, Italy, E-Mail:
  chiara.zanini@polito.it}
}

\date{January 26, 2014 }

\maketitle

\begin{abstract}
%

 This paper focuses on
rate-independent damage in elastic bodies. Since the driving energy is nonconvex,
solutions may have jumps as a function of time, and in this situation
it is known that the classical concept
of energetic solutions for rate-independent systems may fail to accurately describe the
behavior of the system at jumps.

 Therefore we resort to the (by now well-established) \emph{vanishing viscosity}
approach to rate-\-in\-dependent modeling, and  approximate the model by its viscous regularization.
In fact, the analysis of the latter PDE system presents remarkable difficulties, due to its highly nonlinear character.
We tackle it by
combining a \emph{variational} approach to a class of abstract doubly nonlinear evolution equations, with
 careful regularity estimates tailored to this specific system,  relying on a $q$-Laplacian type gradient regularization of the damage variable.
 Hence for the viscous problem we conclude the existence of weak
 solutions, satisfying a 
 suitable energy-dissipation inequality that is the starting point for the vanishing viscosity analysis.
 The latter leads to the notion of \emph{(weak) parameterized
   solution} to our rate-independent system, 
 which encompasses the influence of viscosity in the description of the jump regime.
\end{abstract}



%



%
\section{Introduction}
\label{s:introduction}
We analyze the following
PDE system  for damage evolution
\begin{subequations}
\label{pde-syst-intro}
\begin{align}
& \label{pde-syst-intro1} - \mathrm{div} (g(z)\mathbb{C} \eps (u + u_D )) = \ell \quad \text{in }\Omega \times (0,T),
\\
& \label{pde-syst-intro2}
 \begin{aligned}
 \partial \mathrm{R}_1 (z_t)  -\mathrm{div} \left( (1+\abs{  \nabla z}^2)^\frac{q-2}{2} \nabla  z \right) + f'(z)   + \frac12 g'(z)
   \bbC \varepsilon(u + u_D ) : \varepsilon(u+ u_D ) \ni 0 \quad \text{in }\Omega \times (0,T).
     \end{aligned}
\end{align}
\end{subequations}
Here, $\Omega \subset \R^d$, with $d\geq 3$, is a bounded Lipschitz domain,
  occupied by a body subject to damage,
  $u: \Omega \times [0,T] \to \R^d$ the displacement vector,
  $\eps( u)$ denoting the symmetrized strain tensor,
  and  $z  :\Omega \times [0,T] \to [0,1] $ the damage parameter. Within the approach of Generalized Standard Materials (see also \cite{FN96} and \cite{DeSimone1} in the case of stress softening), we  model the degradation of the elastic behavior of the body through the internal variable $z$, which assesses the soundness of the material. Thus, for $z(x,t)=1$ ($z(x,t)=0$, respectively)  the material is in the unbroken state (in the fully damaged state), ``locally" around $x \in\Omega$ and at the process time $t \in [0,T]$; the intermediate case  $0<z(x,t)<1$
describes partial damage.
We consider a gradient regularization for $z$, which leads to the $q$-Laplacian  operator in
\eqref{pde-syst-intro}, with $q \geq 2$.  Rate-independence and unidirectionality of damage evolution stem
 from the $1$-positively homogeneous dissipation potential
\[
\mathrm{R}_1 : \R \to [0,\infty ], \qquad \mathrm{R}_1(\eta) = \begin{cases}
\kappa |\eta| & \text{if } \eta \leq 0,
\\
\infty& \text{otherwise},
\end{cases}
\]
with $\kappa>0$   a given fracture toughness. $\mathrm{R}_1$
 \bnnc enforces \ennc the constraint that $z_t (x,t) \leq 0$
on $\Omega \times (0,T)$; the operator $\partial\mathrm{R}_1: \R
\rightrightarrows \R  $ is its subdifferential in the sense of convex
analysis. Furthermore,
 $f: \R \to \R$ and
 $g : \R \to (0,\infty)$ are given constitutive functions,
$\bbC= \bbC(x)$ is the
 (positive definite, symmetric) $x$-dependent elasticity tensor, $u_D $ a Dirichlet datum (from now on, within this section we will take $u_D=0$ for simplicity), and $\ell$ is the external loading.
 System \eqref{pde-syst-intro} is supplemented
 with zero Neumann conditions for $z$ on the whole of $\partial\Omega$
 and with mixed boundary conditions
 for $u$ on $\partial\Omega=\Gamma_D\cup\Gamma_N$,
 where $\Gamma_D $  is  a closed subset of  $\partial
 \Omega$ on which Dirichlet boundary conditions are prescribed for $u$. 
 For shortness,
 in this introduction we assume homogeneous Dirichlet
 conditions for $u$. 

 Hereafter, we shall suppose that $g(z) \geq c>0$ for all $z \in \R$: joint with the positive-definiteness
 of the tensor $\bbC$, this excludes elliptic degeneracy of equation \eqref{pde-syst-intro1} even in the case
  of maximal damage, i.e.\ for
  $z(x,t)=0$. Namely, here we rule out \emph{complete damage}.


 Observe that \eqref{pde-syst-intro1}
 is the Euler-Lagrange equation for the minimization,
 with respect to the variable $u$,
 of the stored energy functional $ \calE:
[0,T]\times \calU \times \calZ \to \R$
\begin{equation}
\label{ef-intro}
  \calE(t,u,z): = \frac1q \int_\Omega (1+|\nabla z|^2)^{\frac q2} \dx + \int_{\Omega} f(z) \dx +
  \frac12\int_{\Omega} g(z) \bbC \varepsilon(u) : \varepsilon(u)\dx
  - \pairing{\calU}{\ell(t)}{u},
\end{equation}
 with the state spaces $\calU
= \Set{v\in W^{1,2}(\Omega,\R^d)}{v|_{\Gamma_D}=0}$
 for $u$, and
$\calZ= W^{1,q}(\Omega)$ for $z$.
In fact, in what follows we are going to treat \eqref{pde-syst-intro} as an abstract evolution equation set in the
dual space $\calZ^*$, viz.\
\begin{equation}
\label{dndia}
\begin{aligned}
    & \partial \calR_1(z'(t)) + A_\il z(t) + f'(z(t)) + \frac12 g'(z(t))
   \bbC \varepsilon(u(t)) : \varepsilon(u(t)) \ni 0 \quad    \text{in } \calZ^* \quad \foraa\, t \in (0,T),
   \\
   &
   u(t) \in \Argmin\Set{ \calE(t,v,z(t))}{v \in \calU}\quad \foraa\, t \in (0,T),
\end{aligned}
\end{equation}
with
$
\calR_1 :  L^1(\Omega) \to [0,\infty]$   defined by
\begin{equation}
\label{def_R1}
\calR_1(\eta)=
\int_\Omega \mathrm{R}_1 (\eta(x)) \,\dd x =
\begin{cases}
\displaystyle\int_\Omega \kappa \abs{\eta(x)}\dx &\text{if }\eta\leq 0 \text{ a.e.\ in $\Omega$,}\\
\infty &\text{else,}
\end{cases}
\end{equation}
$\partial\calR_1 : \calZ \rightrightarrows \calZ^*$ its (convex analysis) subdifferential,
and
$A_\il$ denoting the $\il$-Laplacian  operator
$
A_\il z= -
\mathrm{div} (1+\abs{  \nabla z}^2)^\frac{q-2}{2} \nabla z
$
with zero Neumann boundary conditions.
Introducing the \emph{reduced} energy
\begin{equation}
\label{e:intro-reduced}
\calI: [0,T]\times \calZ \to \R,
 \quad
\calI(t,z) = \inf_{v \in \calU} \calE (t,v,z),
\end{equation}
we can further reformulate \eqref{dndia} as
\begin{equation}
\label{rate-indep-dne-intro}
 \partial \calR_{1}(z'(t)) +
\rmD_z\calI(t,z(t))\ni 0 \quad \text{in $\calZ^*$ } \ \foraa\, t \in
(0,T),
\end{equation}
where $\rmD_z \calI$ is the  G\^ateaux derivative of $\calI$ w.r.t.\ $z$.

Since $\calR_1$ has only linear growth and the reduced energy $\calI(t,\cdot)$ has no uniform convexity properties,
 solutions to \eqref{rate-indep-dne-intro} are, in general, only $\mathrm{BV}$-functions of time.
 This calls for weak, derivative-free solvability concepts for \eqref{rate-indep-dne-intro}: first and foremost, the notion of \emph{energetic solution} by \textsc{Mielke \& Theil} \cite{MieThe99MMRI, MieThe04RIHM, Mielke05}.
  For \emph{incomplete} damage, the existence of energetic solutions
  to a version of
 \eqref{dndia} was established for $q>d$  in \cite{MR06}, and extended
 to  $q>1$
   in \cite{TM10}; in \cite{Tho12} the case of a
   $\mathrm{BV}$-regularization (i.e.\ $q=1$) was analyzed.

Over the last years,
it has been realized that
 the description of rate-independent evolution resulting from the
 global stability condition of the energetic solution concept does not seem to be mechanically feasible in the case of a nonconvex driving energy.
 Indeed, in order to satisfy the global stability,
 energetic solutions may change instantaneously in a very
drastic way, jumping into very far-apart energetic configurations
(see, for instance, \cite[Ex.\,6.1]{Miel03EFME},
\cite[Ex.\,6.3]{KnMiZa07?ILMC},
\cite[Ex.\,1]{mrs2009dcds},
as well as the characterization of energetic solutions to one-dimensional rate-independent systems provided in \cite{RossiSavare12}).
This observation has motivated the introduction of alternative weak solution notions.

 A well-established approach for deriving a  concept which accurately describes the behavior of the solution
 at jumps is taking the  \emph{vanishing viscosity} limit in the \emph{viscous} approximation of
 a given rate-independent system. Starting from the seminal paper \cite{efendiev-mielke}, this technique
 has by now been thoroughly developed both for abstract rate-independent systems \cite{mrs2009dcds, MRS10a, MZ10},
 and in the applications to fracture \cite{ToaZan06?AVAQ,Cagnetti,KnMiZa07?ILMC,KnMiZa08?CPPM,LazzaroniToader}, and to
 plasticity \cite{DDMM07?VVAQ,BabFraMor12,DalDesSol11,DalDesSol12,FrSt2013}.

 Following on the analysis initiated in  \cite{krz}, in this paper we  develop this approach for  the damage
 system \eqref{dndia}, and accordingly consider its viscous regularization
\begin{equation}
\label{dndia-eps}
    \partial \calR_1(z'(t)) +\epsi z'(t) + A_\il z(t) + f'(z(t)) + \frac12 g'(z(t))
   \bbC \varepsilon(u(t)) : \varepsilon(u(t)) \ni 0  \text{ in } \calZ^* \ \foraa\, t \in (0,T),
   \end{equation}
   with
   $u(t) \in \Argmin\Set{ \calE(t,v,z(t))}{v \in \calU}$ for almost all $t \in (0,T)$.
   Observe that
   \eqref{dndia-eps} rewrites as
   \begin{equation}
\label{viscous-dne-intro}
 \partial \calR_{\epsi}(z'(t)) +
\rmD_z\calI(t,z(t))\ni 0 \quad \text{in $\calZ^*$ } \ \foraa\, t \in
(0,T),
\end{equation}
with $\calR_\epsi (\eta):= \calR_1 (\eta) + \frac\epsi 2 \| \eta\|_{L^2(\Omega)}^2$.
In fact, the analysis of \eqref{dndia-eps} is itself fraught with analytical difficulties.
In what follows, we briefly hint at them, and then illustrate our approach and our 
 existence result, Theorem \ref{thm:ex-viscous}, for
the Cauchy problem associated with \eqref{dndia-eps}.
We then describe the vanishing viscosity analysis of
\eqref{dndia-eps}. 
\paragraph{The viscous problem: mathematical difficulties and existing  results.}
The most evident difficulty attached to the analysis of \eqref{dndia-eps} is
the presence of
the  quadratic term $ g'(z(t))
   \bbC \varepsilon(u(t)) : \varepsilon(u(t))$ on its right-hand side.
   The basic energy estimate for \eqref{dndia-eps}  provides a (uniform w.r.t.\ time) $W^{1,2} (\Omega;\R^d)$-bound for $u$ which,
even assuming  $|g'(z)|\leq C$, only gives an $L^1(\Omega)$-estimate for   $ g'(z)\bbC
    \varepsilon(u) : \varepsilon(u)$. Therefore,
    it is necessary to enhance the spatial regularity of $u$, which
    requires performing enhanced regularity estimates on \eqref{dndia-eps}.

    The latter issue poses further difficulties due to the \emph{doubly nonlinear} 
    character of \eqref{dndia-eps}, because of the simultaneous presence of
    the nonlinear $\il$-Laplacian operator $A_\il$, and of the (maximal monotone) 
    multivalued operator
    $\partial \calR_1 : \calZ \rightrightarrows \calZ^*$.

    Last but not least, since the domain of  $\calR_1$  is not the whole space $\calZ$,
    $\partial \calR_1$ is indeed an \emph{unbounded} operator. This
    rules out the possibility of deriving bounds for $\rmD_z \calI$ by comparison arguments in \eqref{viscous-dne-intro}. Since the term $f'(z)$ contributing to $\rmD_z \calI$ may be considered of lower
    order under suitable assumptions on $f$, the problem boils down to deriving further estimates for $A_\il z$ and, again, for the quadratic term $ g'(z)\bbC 
    \varepsilon(u) : \varepsilon(u)$.

    All of these difficulties are reflected in the results
    available in the literature on damage problems, starting from the first, pioneering
    paper on the viscous system \eqref{dndia-eps}, viz.\
    \cite{bonetti-schimperna}. Therein, the Laplace operator
    (i.e., $q=2$) was considered and the choice $f' = \partial I_{[0,1]} $, with
    $I_{[0,1]}$ the indicator function of $[0,1]$, was allowed for.
    However, the gradient  regularizing term  $A_2 z_t$  was
    also  added, enabling the authors to handle the doubly nonlinear
    character of \eqref{dndia-eps} and to derive enhanced estimates on $z$
    by resorting to elliptic regularity results. Observe that such
    results rely on suitable smoothness
    assumptions on the domain $\Omega$. The latter  are  also at the
    core of the analysis developed in the subsequent paper
    \cite{bonetti-schimperna-segatti},
    where the (doubly nonlinear) evolution equation for the damage
    parameter $z$ (with $q=2$) is coupled with a parabolic equation for $u$, 
    in the context of linear viscoelasticity.
    Therein, the usage of the regularizing term
     $A_2 z_t$  is avoided. In fact, the authors exploit  the available estimates 
     on the viscous term
    $\varepsilon (u_t)$, and elliptic regularity arguments on $u$,  in
    order to test \eqref{dndia-eps} by $\partial_t ( A_2 z + f'(z))$. 
    This allows them to estimate the term $A_2 z$ and
to gain enhanced spatial regularity for $z$, again by elliptic regularity.    
Refined estimates combined with regularity assumptions on the domain $\Omega$ are crucial also in \cite{bonetti-bonfanti},
    extending the analysis to a temperature-dependent model.

    In the recent \cite{HK10}, 
    different techniques have been adopted to analyze models coupling damage with phase separation processes in elastic
    bodies  (see also \cite{HK11}). Also in \cite{HK10}, a $q$-Laplacian regularization with $q>d$ ($d$ being the space dimension) is
    used in order to ensure $\rmC^0 (\overline\Omega)$-regularity for $z$.
 Because of the complexity of the overall system for damage and  phase separation, and because of the triply nonlinear character of the equation for the damage parameter (featuring the $q$-Laplacian  and the multivalued operators $\partial \calR_1$ and $\partial I_{[0,1]}$), the authors are able to prove existence  only
    for a weak solution notion.

\paragraph{The viscous problem: our results.}
Our aim  in this paper is to analyze \eqref{dndia} and its viscous approximation \eqref{dndia-eps}
 under \emph{minimal regularity} assumptions on $\Omega$. This is
 particularly meaningful in view of the applications to 
 engineering problems, where the spatial domain occupied by the
 elastic body is usually far from being of class
 $\rmC^2$. 
 Therefore, we have to apply refined elliptic regularity results to
 enhance the spatial regularity of $u$.

 Let us motivate the choice of $q>d$ for the $q$-Laplacian
operator $A_q$. Since the damage variable $z$ enters into the
coefficients of the  operator of linear elasticity
 $-\Div (g(z)\bbC
\varepsilon(u)$), there is an intimate relation between the regularity
of $z$ and the regularity of the displacements $u$.
In our analysis we rely on the fact that
 $u\in W^{1,p}(\Omega)$ with $p>d$. Such a
regularity property can be achieved for the solutions of linear elliptic
systems
on nonsmooth domains with mixed boundary conditions (under certain
geometric conditions),
 assuming that the coefficients are at least uniformly continuous on
$\overline\Omega$. This is in particular guaranteed, if $z\in
W^{1,q}(\Omega)$ with $q>d$, see Section \ref{sec:red_energy} for details.
However, if $q=2$, i.e.\ $A_q$ coincides with the standard Laplacian,
then the coefficient $g(z)\bbC$ belongs to $L^\infty(\Omega)\cap
H^1(\Omega)$. In contrast to the case of scalar elliptic equations,
for linear elliptic systems this regularity of the coefficients in
general  does not imply that solutions are continuous. This situation
is highlighted in the three-dimensional example due to Ne\v{c}as and
\v{S}tipl, \cite{NeSti76},  with coefficients from  $L^\infty(\Omega)\cap
H^1(\Omega)$ leading to weak solutions $u$ that do not belong to
$\rmC^0(\overline\Omega)$ and hence also not to $W^{1,p}(\Omega)$ with
$p>d$. For this reason in the present paper we focus on the assumption
that $q>d$.  \DDDE

   In the same spirit, in \cite{krz} we  chose the   fractional
    $s$-Laplacian
   operator $A_s$, on the Sobolev-Slobodeckij space $W^{s,2}(\Omega)$,
   with $s \geq \frac d2$, in place  of the
   $q$-Laplacian. 
    Note that for the case $d=2$ the analysis performed in \cite{krz}
    deals with the Laplacian for $z$, so the choice of a ``pure''
    $s$-Laplacian operator was made for space dimension $d\geq 3$.
    The $q$-Laplacian is, 
     however, a more
   physically justifiable regularization than the nonlocal operator $A_s$,
   which fact has motivated the present study.

  Relying on  the spatial continuity of $z$, we  obtain the regularity result which lies at
  the core of our analysis, viz.\ Lemma \ref{lem_reg_babadjan} asserting that, under suitable conditions on the data $u_D$ and $\ell$,
  \begin{equation}
  \label{regularity-for-u-intro}
  \exists\, p_*> d \text{ such that } \| u \|_{W^{1,p_*} (\Omega;\R^d)} \leq C.
  \end{equation}
  Its proof is based  on regularity results for
  elliptic systems with constant (or smooth) coefficients,
   combined with an iteration argument drawn from
   \cite{BabMil12}. Let us stress that the regularity results
   which we invoke allow
  for elliptic operators with
  changing boundary conditions and, more importantly, for  nonconvex, \emph{nonsmooth} polyhedral domains, see Example \ref{example_nonsmooth}
  later on. 

 Estimate \eqref{regularity-for-u-intro}  enables
  us to improve the regularity of $z$ which  results from the sole basic energy estimate. In particular,
  (formally) differentiating \eqref{viscous-dne-intro} and testing it
  by $z'$, we enhance the spatial regularity of $z'$ by  deducing the \emph{mixed estimate}
  \begin{equation}
\label{mixed-estimate-intro}
\int_0^T \int_\Omega (1+|\nabla z(t)|^2)^{\frac{\il-2}{2}} |\nabla z' (t)|^2 \,\dd x \,\dd r <\infty.
\end{equation}

All these calculations are made rigorous on the time-discretization scheme with which we approximate \eqref{viscous-dne-intro}:
discrete solutions $(z_{k}^\tau)_{k=0}^N$, with $\tau>0$ a constant time-step, are
constructed via the time-incremental minimization scheme
\begin{equation}
\label{time-discrete-min-intro}
  z_{k+1}^\tau \in
\Argmin\Set{\calI(t_{k+1}^\tau,z) +
\tau\calR_\epsilon\left(\frac{z
    -z_k^\tau}{\tau}\right)}{z\in\calZ}.
    \end{equation}
    The related piecewise constant and piecewise linear interpolants $\overline{z}_\tau$ and $\hat{z}_\tau$ accordingly fulfill the Euler-Lagrange equation
    \begin{equation}
\label{time-el-intro}
       \partial\calR_1\left(\hat{z}_\tau'(t)\right)+  \epsilon
  \hat{z}_\tau'(t)
  + A_\il \overline{z}_\tau(t) +
  f'(\overline{z}_\tau(t) ) + \frac12 g'(\overline{z}_\tau(t) )
   \bbC \varepsilon( \overline{u}_\tau(t)) : \varepsilon(\overline{u}_\tau(t)) \ni
  0 \quad \foraa\, t \in (0,T).
  \end{equation}
  With a maximum principle argument, we prove (cf.\ Prop.\ \ref{rem:nice} later on) that,
   if the initial datum $z_0 $
   fulfills $z_0 \in [0,1] $ a.e.\ in $\Omega$, then
   $0 \leq \overline{z}_\tau(t) \leq 1$ for all $t \in [0,T]$.

  A crucial ingredient for passing to the limit as $\tau \to 0$ in \eqref{time-el-intro} is
  to obtain suitable estimates  for the term  $A_\il \overline{z}_\tau$. Indeed, its  weak convergence
  cannot be solely deduced from estimates for $(\overline{z}_\tau)_\tau$ in the space $\calZ=W^{1,q}(\Omega)$,
  due to the nonlinear character of $A_\il$. This is \bnnc in its own right \ennc a challenging feature of the  problem investigated here: the linear operator
$A_s$  considerably simplified the existence proof for \eqref{visc-eps-dne}, in \cite{krz}.
In fact, after Lemma \ref{lem_reg_babadjan}, the second milestone of our analysis is  Theorem \ref{app_reg_thm_z}:
based on a careful difference quotient argument,
 for the discrete solutions to \eqref{time-discrete-min-intro} it ensures
\begin{equation}
\label{enh-reg-z-intro}
\exists\, C>0 \ \forall\, \tau>0 \ \forall\, t \in (0,T] \, : \qquad
\| \overline{z}_\tau (t) \|_{W^{1+\beta,q}(\Omega)}\leq C \qquad
 \text{for all } 0<\beta<  \frac{1}{q} \left(1-\frac dq\right).
\end{equation}

Estimate \eqref{enh-reg-z-intro} yields $W^{1,q}(\Omega)$-compactness  for $( \overline{z}_\tau)_\tau$ and thus allows us to take the limit of the term $A_\il \overline{z}_\tau$. Indeed, for the limit passage as $\tau \to 0$
we adopt a \emph{variational} approach: instead of passing to the limit directly in \eqref{time-el-intro}, we
 take the limit of the associated \emph{discrete energy inequality},
 cf.\ \eqref{discr-enid} ahead. 
With suitable compactness and lower semicontinuity arguments, we
deduce that there exists a limit curve $z \in 
L^{2q}(0,T; W^{1+\beta,q}(\Omega)) \cap W^{1,2}(0,T; W^{1,2}(\Omega) ) $,
 with $ z \in [0,1]$ a.e.\ in $\Omega \times (0,T)$,
 for which the mixed estimate \eqref{mixed-estimate-intro} holds and
fulfilling the energy inequality associated with \eqref{viscous-dne-intro}, viz.\
\begin{equation}
\label{energy-inequality-intro}
\int_s^t \calR_\epsi (z'(r)) \,\dd r + \int_s^t \calR_\epsi^* (-\rmD_z \calI (r,z(r))) \,\dd r + \calI (t,z(t)) \leq  \calI (s,z(s))+ \int_s^t \partial_t \calI (r,z(r)) \,\dd r.
\end{equation}
for all $0\leq s \leq t \leq T$, with $\calR_\epsi^*$ the Fenchel-Moreau conjugate of $\calR_\epsi$.
We also prove in Theorem \ref{thm:chain-rule} that, along the limit curve $z$ a chain rule formula is valid, viz.\ for almost all $t \in (0,T)$
\begin{equation}
\label{ch-rule-intro}
\begin{aligned}
\frac{\dd }{\dd t} \calI (t,z(t)) - \partial_t \calI (t,z(t))  & =  \int_\Omega
(1+ |\nabla z(t)|^2)^{\frac{q-2}{2}} \nabla z(t) \cdot \nabla z'(t) \,\dd x
\\
& \quad
 +\int_\Omega (f'(z(t)) + \frac12 g'(z(t)) \bbC \eps(u(t)){\colon}  \eps(u(t)))  z'(t) \,\dd x\,.
 \end{aligned}
\end{equation}
A key ingredient for \eqref{ch-rule-intro} is indeed estimate \eqref{mixed-estimate-intro},
 which guarantees that the first integral on the right-hand side of \eqref{ch-rule-intro} is well defined. 
Relying on \eqref{ch-rule-intro}, in Proposition \ref{prop:equivalence} we show that
    the energy inequality \eqref{ch-rule-intro} is in fact equivalent to
    \begin{equation}
\label{weak-def-sol-intro}
\begin{aligned}
\calR_\epsi(w) - \calR_\epsi (z'(t)) \geq  &  \pairing{\calZ}{-A_{\il} z(t) }{w}  +  \int_\Omega (1+ |\nabla z(t)|^2)^{\frac{q-2}{2}} \nabla z(t) \cdot \nabla z'(t) \,\dd x
\\ & \quad
- \int_\Omega (f'(z(t)) + \frac12 g'(z(t)) \bbC \eps(u(t)){\colon}  \eps(u(t)))(w-z'(t))
 \quad \text{for all } w \in \calZ
\end{aligned}
\end{equation}
for almost all $t \in (0,T)$.
This variational inequality  defines our  notion of weak solution for the viscous doubly nonlinear equation \eqref{dndia-eps}, cf.\ Definition \ref{def:wsol}: our  existence result for weak solutions,  Theorem \ref{thm:ex-viscous}, follows from the aforementioned arguments. 

Observe that, 
as soon as we can interpret
 the terms on
the r.h.s.\ of \eqref{weak-def-sol-intro} as
the duality product
$\pairing{\calZ}{-A_{\il} z - f'(z) - \frac12 g'(z) \mathbb{C}
  \eps(u){\colon} \eps(u)}{w-z'} $, then
\eqref{weak-def-sol-intro} 
is in fact equivalent to the
subdifferential inclusion \eqref{dndia-eps}. 
In Sec.\ \ref{ss:3.1}
  the relation of our weak solution concept for \eqref{dndia-eps}
 to the usual subdifferential formulation \eqref{viscous-dne-intro}
  is discussed at length, 
   also in connection with
   the chain rule \eqref{ch-rule-intro}, and with  the failure of
   the energy inequality \eqref{energy-inequality-intro} to hold as an
   equality.

\paragraph{The vanishing viscosity analysis.} As in \cite{krz}, for passing to the limit in \eqref{viscous-dne-intro}
as $\epsi \to 0$
 we adopt the \emph{reparameterization technique} from  \cite{efendiev-mielke}, which leads to a notion of
 solution for the rate-independent system \eqref{rate-indep-dne-intro}, encompassing a finer description of the energetic behavior of the system jumps.
  The underlying philosophy is that,  at jumps  the  vanishing viscosity solutions to \eqref{rate-indep-dne-intro}
    follow a path which
is reminiscent  of the viscous approximation.
To reveal this, one has to go over to an
extended state space and study  the limiting behavior of the sequence $(
\tilde{t}_\epsilon, \tilde{z}_\epsilon)_\epsilon $ as $\epsilon \downarrow 0$,
for a  suitable reparameterization
 $\tilde{z}_\epsilon= z_\epsilon \circ \tilde{t}_\epsilon$ of  a family  $(z_\epsilon)_\epsi$  of \emph{weak}
solutions (in the sense of \eqref{weak-def-sol-intro})  to \eqref{viscous-dne-intro}.
More precisely,
 we  consider the
$L^2(\Omega)$-arclength parameterization
$
   s_\epsilon(t)=t+\int_0^t \norm{z'_\epsilon(r)}_{L^2(\Omega)}\dr
$
of the graph of $z_\epsi$. The key $\mathrm{BV}$-estimate
 \begin{equation}
 \label{bv-est-intro}
\sup_{\epsilon>0}\int_0^T\norm{z_\epsilon'(t)}_{L^2(\Omega)}\dt \leq C
\end{equation}
for viscous solutions to \eqref{viscous-dne-intro} guarantees that, up to a subsequence,
$S_\epsilon:= s_\epsilon(T)$
converges as $\epsi \to 0$ to some $S\geq T$.
We then
 define
$\tilde{t}_\epsilon:[0,S_\epsilon]\to [0,T]$
via $ \tilde{t}_\epsilon(s):= s_\epsilon^{-1}(s)$, and accordingly set
$
   \tilde{z}_\epsilon(s):=z_\epsilon(\tilde{t}_\epsilon(s)). $

In Theorem \ref{main-thm-vanvisc} we  prove that, up to a subsequence, the curves $(
\tilde{t}_\epsilon, \tilde{z}_\epsilon) $ converge to a pair $(\tilde t, \tilde z) :  [0,S] \to [0,T] \times \calZ
$ which fulfills the parameterized energy inequality
\begin{equation}
\label{parametrized-energy-ineq-intro}
\begin{aligned}
 \int_{s_1}^{s_2} \widetilde{\mathcal{M}}_0 (\tilde{t}'(r),\tilde{z}'(r),  -\rmD_z \calI
(\tilde{t}(r),\tilde{z}(r)))
   \,\mathrm{d}r
  + \calI(\tilde{t}(s_2),\tilde{z}(s_2))  \leq
   \calI(\tilde{t}(s_1),\tilde{z}(s_1))
   +\int_{s_1}^{s_2} \partial_t
   \calI(\tilde{t}(r),\tilde{z}(r))\tilde{t}'(r)
   \mathrm{d}r
   \end{aligned}
\end{equation}
for all $0 \leq s_1 \leq s_2 \leq S$.  In \eqref{parametrized-energy-ineq-intro},  the term
\[
\widetilde{\mathcal{M}}_0 (\tilde{t}',\tilde{z}',  -\rmD_z \calI
(\tilde{t},\tilde{z})) =\begin{cases}
\calR_1 (\tilde{z}') + I_{\partial \calR_1 (0)} (-\rmD_z \calI
(\tilde{t},\tilde{z}))  &\text{if $\tilde{t}'>0$,}
\\
\calR_1 (\tilde{z}') +  \norm{\tilde{z}'}_{L^2(\Omega)} \twodis(-\rmD_z \calI
(\tilde{t},\tilde{z}), \partial\calR_1(0))  &\text{if $\tilde{t}'=0$,}
\end{cases}
\]
($ \twodis(-\rmD_z \calI
(\tilde{t},\tilde{z}), \partial\calR_1(0))$ denoting the $L^2(\Omega)$-distance of $-\rmD_z \calI
(\tilde{t},\tilde{z})$ from $\partial\calR_1(0)$)
 \bnnc enforces  the \emph{local stability} condition $-\rmD_z \calI
(\tilde{t},\tilde{z}) \in \partial \calR_1 (0)$ in the case of
  purely  rate-independent evolution, i.e.\
 when $\tilde{t}'>0$. \ennc When the (slow) external time, encoded by the function $\tilde t$, is frozen, the system jumps.
 Then, the system may switch to a \emph{viscous} regime.  
 We refer to Sec.\ \ref{ss:7.2} for further details on this.

\paragraph{Plan of the paper.}
 In Section \ref{section2} we specify all our assumptions, prove the crucial regularity Lemma~\ref{lem_reg_babadjan},
 and collect all
 properties of the reduced energy $\calI$ which shall be used in the
 subsequent analysis. Then, in Section \ref{s:3}  we  introduce and
 motivate our notion of \emph{weak solution} for the
Cauchy problem associated with the viscous equation \eqref{dndia-eps}, state
 Theorem \ref{thm:ex-viscous} (=existence of weak solutions and a priori estimates uniform w.r.\ to
 the viscosity parameter $\epsi$), and prove Theorem \ref{thm:chain-rule}, providing the
  chain rule \eqref{ch-rule-intro}. In Section~\ref{s:4}  we set up the time-discretization scheme for \eqref{viscous-dne-intro} and prove the higher differentiability result
  yielding \eqref{enh-reg-z-intro}. Section \ref{s:5} is devoted to
  the proof of a series of a priori estimates on the discrete solutions, most of which uniform both w.r.t.\ $\tau$
  \emph{and} $\epsi$. In particular, the discrete version of the $\mathrm{BV}$-estimate \eqref{bv-est-intro}
  is derived.
 We prove Theorem\ \ref{thm:ex-viscous} by passing to the limit as $\tau \to 0$
   in the time discrete scheme, also
   exploiting Young measure techniques which are recapped in Appendix \ref{s:a-1}.
 Finally, in Section \ref{s:7} we develop the vanishing viscosity
 analysis of \eqref{viscous-dne-intro}.
%
%

\section{Preliminaries}
\label{section2}
\subsection{Set-up}
\label{ss:2.1}

\paragraph{Notation.}
For a given Banach space $X$, we shall denote by
$\pairing{X}{\cdot}{\cdot}$ the duality pairing between $X^*$ and
$X$, and, if $X$ is a Hilbert space, we shall use the symbol
 $\inner{X}{\cdot}{\cdot}$
for its scalar product.
 For matrices $A,B \in \R^{m\times d}$ the inner product is
defined by $A:B= \tr(B^\top A)= \sum_{i=1}^m\sum_{j=1}^d a_{ij}
b_{ij}$.

Let $d\geq 3$ and let $\Omega\subset\R^d$ be a bounded domain with a closed Dirichlet boundary
$\Gamma_D\subset\partial\Omega$ and Neumann boundary
$\Gamma_N=\partial\Omega\backslash \Gamma_D$.
 Further  assumptions on the regularity of
$\Omega$ and on the Dirichlet boundary $\Gamma_D$ 
will be specified in Sections \ref{sec:red_energy} and \ref{ss:3.1} (cf.\
(A$_\Omega 1$) and (A$_\Omega 2$)).
The letter $Q$  shall stand for the space-time cylinder $\Omega \times
(0,T)$. The following function spaces and notation shall be used for
$\sigma\geq 0$, $p\in  [1,\infty]$:
\begin{itemize}
\item $W^{\sigma,p}(\Omega)$ Sobolev-Slobodeckij spaces,
  \item
  $W^{1,p}_{\Gamma_D}(\Omega):= \Set{u\in
    W^{1,p}(\Omega)}{u\big|_{\Gamma_D}=0}$ and
  $W_{\Gamma_D}^{-1,p}(\Omega):=\big(W_{\Gamma_D}^{1,p'}(\Omega)\big)^*$
  the dual space, $\frac1p+\frac1{p'}=1$.
\end{itemize}

We shall denote by
 $u : \Omega \to \R^d$ the displacement, and by $z:
\Omega \to \R$ the (scalar) damage variable. The corresponding state
spaces are
 \begin{align}
 \label{state-u}
 \calU&:=\Set{v\in W^{1,2}(\Omega,\R^d)}{v\big|_{\Gamma_D}=0}=
 W_{\Gamma_D}^{1,2}(\Omega,\R^d),
\\
 \label{state-z}
 \calZ&:= W^{1,\il}(\Omega),
\end{align}
with $\il>d$. 
On the space $\calZ$  the $q$-Laplacian operator is defined as follows
\begin{align*}
 A_q:\calZ\rightarrow\calZ^*, \quad \langle A_q(z),v\rangle_\calZ:=\int_\Omega
(1 + \abs{\nabla z}^2)^{\frac{q-2}{2}}\nabla z\cdot\nabla v\dx \quad\text{for }z,v\in
\calZ.
\end{align*}
 \paragraph{\bf Useful inequalities.} We collect here some inequalities which
shall
 be extensively used in the following.
First of all,
\begin{equation}
\label{lions-magenes}
\begin{aligned}
\text{let } p_* >d\, : \ \text{ then } \
  \forall\, \rho >0 \ \exists\, \, C_\rho>0  \
 & \forall\, z \in W^{1,2}(\Omega)\, :
 \\ &
 \qquad \|z\|_{L^{2p_*/(p_*{-}2)}
(\Omega)} \leq \rho \|z\|_{W^{1,2}(\Omega)} + C_\rho
\|z\|_{L^2(\Omega)}.
\end{aligned}
 \end{equation}
 This follows from the
the  compact and continuous   embeddings $W^{1,2}(\Omega) \Subset  L^{2p_*/(p_*{-}2)}
(\Omega)  \subset L^2(\Omega)$ (due to $p_*>d$), on account of
\cite[Lemma 8]{simon87}.

Secondly, let us recall that with $G_q(A)=\frac1q (1 + \abs{A}^2)^\frac{q}{2}$
for $A\in \R^d$, as a consequence of \cite[Lemma~8.3]{ana:Giu94}  the
following monotonicity and convexity estimates are valid
 for $q\geq 2$ and $A,\, B\in \R^d$,
 \begin{align}
 \label{e:from-D-thesis}
 & (\rmD G_q(A) - \rmD G_q(B))\cdot(A-B)\geq c_q(1+\abs{A}^2
+\abs{B}^2)^\frac{q-2}{2}\abs{A-B}^2
\\
 &
 \begin{aligned}
 G_q(B) - G_q(A) -(1 +
\abs{A}^2)^\frac{q-2}{2}A \cdot(B-A)
 & \geq c_q(1+\abs{A}^2 +\abs{B}^2)^\frac{q-2}{2}\abs{A-B}^2
\\
&
\geq
 \wt c_q(\abs{A-B}^2+\abs{A-B}^q). \label{q-convexity}
 \end{aligned}
 \end{align}

 Observe that \eqref{e:from-D-thesis} implies for all $z_1,z_2\in \calZ$:
\begin{align}\label{e2.22}
 \langle A_q z_1 - A_q z_2, z_1-z_2\rangle_{\calZ}
\geq c_q\int_\Omega(1+\abs{\nabla z_1}^2 + \abs{\nabla
z_2}^2)^\frac{q-2}{2}\abs{\nabla (z_1-z_2)}^2\dx.
\end{align}
  Moreover, for all $z_1,z_2$ and $w\in \calZ$
 \begin{equation}\label{A7-D}
\abs{ \pairing{\calZ}{A_q z_1 - A_q z_2}{w}} \leq
c\int_\Omega(1+\abs{\nabla z_1}^2 + \abs{\nabla
z_2}^2)^\frac{q-2}{2} \abs{\nabla (z_1-z_2)}\abs{\nabla w}\,\dd x.
 \end{equation}

\paragraph{The energy functional.}
 The energy $\calE = \calE(t,u,z)$ consists of two contributions. The first one,
$\calI_1= \calI_1 (z)$, only depends on the damage variable. The second one,
 $\calE_2 = \calE_2(t,u,z)$,
is given by the sum of an elastic energy of the type
$\int_\Omega g(z)W(\varepsilon(x,u + u_D(t)))\dx
$, with $g$ and $W$ suitable constitutive functions
(cf.\ Assumption \ref{assumption:energy} below)
and $u_D  \in \rmC^0([0,T];W^{1,2}(\Omega,\R^d))$ a Dirichlet datum (see
\eqref{a:data}), and of the external loading term.
\begin{assumption}
\label{assumption:energy}
We consider
\begin{equation}\label{e:Iq}
\calI_1: \calZ \to \R \ \text{ defined by } \
\calI_1(z):=\calI_\il (z)+ \int_\Omega f(z)\dx \  \text{ with }
\calI_\il (z):= \frac 1{\il} \int_\Omega
(1+\abs{\nabla z}^2)^{\frac{\il}{2}}\,\dd x, \  q>d,
\end{equation}
and $f$ fulfilling
\begin{equation}
\label{ass-eff} f\in \rmC^2(\R),  \quad  \text{such that} \quad  \exists\,
K_1,\, K_2 >0 \quad  \forall\, x \in \R : \quad 
\ f(x)\geq K_1\abs{x} - K_2.
\end{equation}
 A typical choice for $f$  is  $f(z)=z^2$, see
\cite{cra:Gia05}.


As for $\calE_2$,
linearly elastic materials are considered with an elastic energy
density
\[
W(x,\eta)= \frac{1}{2} \bbC(x) \eta : \eta, \quad \text{for $\eta \in
\R^{d\times d}_\text{sym}$ and almost every $x\in \Omega$.}
\]
Hereafter, we shall suppose for the elasticity tensor
that
\begin{subequations}
\begin{align}
\label{elast-1} &
 \bbC \in
 \rmC_\mathrm{lip}^0
(\Omega,\Lin(\R^{d\times d}_\text{sym},\R^{d\times
    d}_\text{sym}))
\DDDS \text{ with } \bbC(x)\xi_1:\xi_2= \bbC(x)\xi_2:\xi_1 \text{ for
  all }x\in \Omega,\xi_i\in \R^{d\times
    d}_\text{sym}, \DDDE
\\ &
\label{elast-2} \exists\, \gamma_0>0 \quad  \text{for all }\xi\in
\R^{d\times d}_\text{sym} \text{ and almost all }x\in \Omega:  \ \
\bbC(x)\xi \colon \xi\geq \gamma_0\abs{\xi}^2.
\end{align}
\end{subequations}
 Let $g:\R\rightarrow \R$ be a further constitutive function
 such that
 \begin{equation}
 \label{new-g}
 g\in \mathrm{C}^{2}(\R),  \text{  with } g' 
 \in L^\infty(\R), \ \text{ and } \exists\, \gamma_1, \, \gamma_2>0 \, : \
\forall\, z \in \R\, : \
\gamma_1 \leq g(z) \leq \gamma_2.
 \end{equation}
 Then, we
 take the
elastic energy
\begin{equation}
\label{elastic-energy}
\begin{aligned}
\calE_2: [0,T]\times \calU \times \calZ \to \R  \ \text{ defined by
} \ \calE_2(t,u,z):=  \int_\Omega g(z)W(x,\varepsilon(u + u_D(t)))\dx
-\pairing{\calU}{\ell(t)}{u}
\end{aligned}
\end{equation}
 where
$\varepsilon(u)=\frac{1}{2}(\nabla u + \nabla u^T)$ is the
symmetrized strain tensor and $\ell\in \rmC^0([0,T],\calU^*)$ an external loading
(cf.\ \eqref{a:data} later on for further requirements).
\end{assumption}
\noindent
For $u\in \calU$ and $z\in \calZ$ the
stored energy is then defined as
\begin{align}
\calE(t,u,z)= \calI_1(z) +\calE_2(t,u,z).
\end{align}
\paragraph{\bf Reduced energy.}
Minimizing the stored energy with respect to the displacements we
obtain the reduced energy
\begin{equation}
\begin{aligned}
\label{reduced-energy}  \calI:[0,T]\times\calZ\rightarrow \R \
\text{ given by } \ \calI(t,z)= \calI_1(z) + \calI_2(t,z)
\text{ with }  \calI_2(t,z)=\inf \Set{\calE_2(t,v,z)}{v\in
\calU}.
\end{aligned}
\end{equation}

 Let us now shortly comment on the conditions  in Assumption
 \ref{assumption:energy}.
\begin{remark}
\label{justification of our assumptions}
\upshape
As already mentioned in the introduction,
our model does not allow for \emph{complete} damage: this is reflected in the coercivity \eqref{elast-2},
and in the strict positivity \eqref{new-g} of  the constitutive function $g$.

 The Lipschitz continuity \eqref{elast-1} of the coefficient matrix
$\mathbb{C} $, as well as the smoothness of $g$, shall be exploited in the forthcoming Lemma
\ref{lem_reg_babadjan}, providing higher integrability for $\varepsilon (u)$.
For proving this result, which  will be at the core
of all the subsequent analysis, we have to stay with a \emph{quadratic} elastic energy.

Relying on these regularity properties for $\mathbb{C}$ and $g$, we are also going to prove
 higher differentiability for $z$,
 cf.\ the crucial Theorem \ref{app_reg_thm_z} later on.

 Nonetheless, let us stress that significant  damage models fall within the scope of the above conditions: for example,
 the Ambrosio-Tortorelli model, whose quasistatic evolution was discussed in \cite{cra:Gia05}, as an approximation of the
 Francfort-Marigo model  \cite{Fra-Mar} for crack propagation.
Observe that,
in  the energy functional considered for the rate-independent model of  \cite{cra:Gia05}, the index in \eqref{e:Iq} is $q=2$.
  Instead, in the more recent
 \cite{BabMil12}, which deals with the (metric) \emph{gradient flow} of the Ambrosio-Tortorelli functional, it is assumed that $q>d$ like in the present setting.
 \end{remark}

 \subsection{
 Geometric assumptions and regularity of the displacement field }
\label{sec:red_energy} 


For the analysis of the time-dependent damage model higher integrability
properties of the gradients of the displacement  field   are needed, and hence
the domain $\Omega$ and the data should be more regular than stated above.

With $\bbC$ as in \eqref{elast-1}--\eqref{elast-2}, $g$ from \eqref{new-g} and
$z\in W^{1,q}(\Omega)$
let $L, L_{g(z)}:W^{1,2}_{\Gamma_D}(\Omega;\R^d)\rightarrow
W^{-1,2}_{\Gamma_D}(\Omega;\R^d)$ be
the operators associated with the bilinear forms describing linear
elasticity, i.e.
\begin{align}
\forall u,v\in W^{1,2}_{\Gamma_D}(\Omega;\R^d):\qquad
\langle  L u ,v\rangle:= \int_\Omega \bbC \varepsilon(u):
\varepsilon(v)\dx,\quad
\langle  L_{g(z)} u ,v\rangle:= \int_\Omega g(z)\bbC \varepsilon(u):
\varepsilon(v)\dx.
\end{align}

 A good compromise between the smoothness needed for our analysis and
nevertheless allowing for polyhedral domains and changing boundary conditions
is formulated in the following

\paragraph{\bf Assumption on the domain.}
\begin{itemize}
 \item[(A$_\Omega 1$)] $\Omega\subset\R^d$ is a bounded domain,  and $\Omega$
and $\Gamma_D\subset\partial\Omega$ ($\Gamma_D$ is  closed and with positive
measure)
are chosen in such a way that the following two conditions are satisfied:
\begin{itemize}
 \item[(i)] The spaces $W^{1,p}_{\Gamma_D}(\Omega;\R^d)$, $p\in (1,\infty)$, form an
interpolation scale.
\item[(ii)]
There exists $p_*>d$ such that for all $p\in [2,p_*]$ the
operator $L: W^{1,p}_{\Gamma_D}(\Omega;\R^d)\rightarrow
W^{-1,p}_{\Gamma_D}(\Omega; \R^d)$ is an isomorphism.
\end{itemize}
\end{itemize}
For an abstract   definition of interpolation scales we refer to \cite{Tri78},
while
in Example \ref{example_nonsmooth} here below we present
 nonsmooth, nonconvex  domains  with mixed boundary
conditions satisfying  (A$_\Omega 1$).

Observe that the isomorphism property stated in (A$_\Omega 1$)  is
also
valid for all $p\in [p_*',2]$, 
 and that the operator norms are  uniformly bounded, i.e. with
 \begin{equation}
 \label{new-notation}
 \calX_p:=
W^{1,p}_{\Gamma_D}(\Omega;\R^d) \text{ and  } \calY_p:=W^{-1,p}_{\Gamma_D}(\Omega;\R^d)
\end{equation}
it holds
\begin{align}
\label{est_unif_op}
 \sup_{p\in [p_*',p_*]}
\norm{L}_{\calX_p\rightarrow \calY_p}
+\norm{L^{-1}}_{\calY_p\rightarrow \calX_p }=:c_{
p_* } <\infty.
\end{align}
%
Here, $\norm{L}_{\calX\rightarrow\calY}$ denotes the operator norm of
$L:\calX\rightarrow\calY$.

Lemma \ref{lem_reg_babadjan} below, which plays a key role in  the subsequent
analysis, states the following important fact: The isomorphism property
\eqref{est_unif_op}
is invariant with respect to a perturbation of the coefficient matrix $\bbC$ by
multiplying it with  (H\"older-)continuous functions $g(z)$, with $g$ as in
\eqref{new-g} and $z\in W^{1,q}(\Omega)$ for $q>d$. Furthermore,   uniform
bounds can
  be derived that are independent of $p\in [p_*',p_*]$ and explicit in $z$,
  relying on an  iteration  argument presented in \cite{BabMil12}.

\begin{lemma}
\label{lem_reg_babadjan}
 Let (A$_\Omega$1) be satisfied, $g$ as in \eqref{new-g} and $q>d$. Let
furthermore $p_*>d$ be chosen according to (A$_\Omega$1) and let $k_*\in \N$
be the smallest number with $k_*>\frac{dq}{2(q-d)}$.

Then for all $p\in [p_*',p_*]$ and all $z\in W^{1,q}(\Omega)$ the operator
\begin{align}
 L_{g(z)}: W^{1,p}_{\Gamma_D}(\Omega;\R^d)\rightarrow
W^{-1,p}_{\Gamma_D}(\Omega;\R^d)
\end{align}
is an isomorphism.
Moreover, there exists a constant
 $c_{q,p_*}>0$ such that
for all 
$z\in W^{1,q}(\Omega)$ and all $p\in [p_*',p_*]$ it
holds (cf.\ notation \eqref{new-notation})
\begin{align}
\label{crucial-reg-estimate}
 \norm{L_{g(z)}^{-1}}_{\calY_{p}\rightarrow\calX_{p}} \leq c_{q,p_*}
(1 + \norm{\nabla
z}_{L^q(\Omega)})^{k_*  \frac{p_*\abs{p-2}}{p(p_*-2)}}.
\end{align}
\end{lemma}

 Observe that
 \begin{equation}
 \label{interesting-exponent-added}
 \sup_{p\in [p_*',p_*]} \frac{p_*\abs{p-2}}{p(p_*-2)} \leq
1\,.
\end{equation}
\begin{proof}
It is sufficient to prove the lemma for $p=p_*$. The other assertions follow
with interpolation and duality arguments.
The proof extends the recursion argument from \cite{BabMil12}, 
where it is carried out for smooth domains and $W^{2,2}(\Omega;\R^d)$, to the
$W^{1,p}(\Omega;\R^d)$-setting  and to domains satisfying (A$_\Omega 1$).

Let $p_*>2$, $q>d$ and $k_*$ be chosen as stated in Lemma
\ref{lem_reg_babadjan}. Define $q_*$ via the relation
\begin{align*}
 p_*=\frac{2dq_*}{d q_* + 2 k_*(d-q_*)}, \quad\text{i.e.} \quad
q_*=\frac{2k_* p_* d}{p_*(2 k_* -d)+ 2d}.
\end{align*}
Observe that $ q_*\in (d,q]$.

 Clearly, $\calY_p\subset W^{-1,2}_{\Gamma_D}(\Omega;\R^d) $ if $p\geq 2$.
Moreover, for all $z\in W^{1,q}(\Omega)$ the functions $g(z),g(z)^{-1}$ are
multiplicators for the spaces $W^{-1,p}_{\Gamma_D}(\Omega;\R^d)$
 and the following estimate is valid: there exists a
constant $c>0$ such that for all $p\in [2,p_*]$, all $z\in W^{1,q}(\Omega)$
and all $b\in \calY_p$ it holds
\begin{align}
 \norm{g(z)^{-1} b}_{\calY_p} \leq c(1 + \norm{\nabla
z}_{L^q(\Omega)})\norm{b}_{\calY_p}.
\label{est_mult_gf}
\end{align}
For $b\in\calY_{p_*}$, let $u\in W^{1,2}_{\Gamma_D}(\Omega ;\R^d )$ be the unique
function
satisfying for all $v\in W^{1,2}_{\Gamma_D}(\Omega;\R^d)$
\begin{align}
\label{eqn_lin_elast1}
 \langle L_{g(z)}u,v\rangle=\int_\Omega g(z)\bbC
\varepsilon(u):\varepsilon(v)\dx = \langle
b,v \rangle_{W^{1,2}(\Omega ;\R^d )} \,.
\end{align}
Due to the multiplier property of $g(z)$, using the product rule and choosing $v
= \frac1{g(z)} \wt v $ in \eqref{eqn_lin_elast1}, it follows that for all $\wt
v\in W^{1,2}_{\Gamma_D}(\Omega;\R^d)$ we have
\begin{align}
\label{eqn_lin_elast2}
\begin{aligned}
 \int_\Omega \bbC \varepsilon(u):\varepsilon(\wt v)\dx &  =  \left \langle \frac{1}{g(z)}
b,\wt v \right\rangle_{W^{1,2}(\Omega ;\R^d )}   +
\int_\Omega \frac{g'(z)}{g(z)}\bbC \varepsilon(u)\nabla z\cdot \wt v\dx \\ &  =:
\left\langle\frac{1}{g(z)} b,\wt v\right\rangle_{W^{1,2}(\Omega ;\R^d )}   +\int_\Omega h(u,z)\cdot \wt v\dx.
\end{aligned}
\end{align}
Since $\frac{1}{g(z)} b\in \calY_{p_*}$ with estimate \eqref{est_mult_gf} (for
$p=p_*$),
 in the following we only have to
concentrate on the term $h(u,z)$, which belongs to $L^{\alpha_0}(\Omega)$ with
$
 \alpha_0 = \frac{2q_*}{2+q_*}\in (1,2).$
Hence, by embedding it follows that
 $h(z,u)\in L^{\alpha_0}(\Omega)
\subset \calY_{p_1}$
with
$p_1= \frac{2 q_*
d}{dq_* - 2(q_*-d)}.$
Observe that $2 <p_1\leq p_*$.
Thanks to assumption (A$_\Omega$1) and estimate \eqref{est_unif_op} it follows
 from \eqref{eqn_lin_elast2}
that $u\in
W^{1,p_1}_{\Gamma_D}(\Omega;\R^d)$ with
\begin{align*}
 \norm{u}_{W^{1,p_1}(\Omega)}
 \leq
 \wt c
\left( 1+\norm{\nabla z}_{L^q(\Omega)}\right)\left(
\norm{u}_{W^{1,2}(\Omega)} +
\norm{b}_{\calY_{p_1}}\right)
\leq c\left( 1+\norm{\nabla z}_{L^q(\Omega)}\right)
\norm{b}_{\calY_{p_1}},
\end{align*}
where the last inequality ensues from the estimate due to the Lax-Milgram
lemma, applied to the solution $u$ of \eqref{eqn_lin_elast1}. The constants
$\wt c, c$ are independent of $z$ and $p_1$.

We now iterate this argument. Assume that $u\in W^{1,p_k}_{\Gamma_D}(\Omega;\R^d)$
 for some $p_k\in [2,p_*)$. Then, again by embedding, we have $h(z,u)\in
L^{\alpha_k}(\Omega)$ with $\alpha_k=\frac{p_k q_*}{p_k + q_*}$, and
$L^{\alpha_k}(\Omega) \subset \calY_{p_{k+1}}$ with
\begin{align}
 p_{k+1}=\frac{d p_k q_*}{d q_* + p_k(d - q_*)}\equiv \frac{2d q_*}{d q_* +
2(k+1)(d - q_*)}.
\end{align}
The last identity follows by induction starting with $p_0=2$.  This argument can
be repeated as long
as $k<k_*$, since for these values of $k$ it holds $p_k<p_{k+1}\leq p_*$.
Observe that $p_{k_*}=p_*$. This implies that $L^{\alpha_{k_*-1}}(\Omega)
\subset \calY_{p_*}$ and hence, again by (A$_\Omega$1),
 we have $ u\in W^{1,p_*}_{\Gamma_D}(\Omega;\R^d)$.
The estimate for $u$ follows recursively, namely
\begin{align*}
\norm{u}_{W^{1,p_*}(\Omega;\R^d)}
&\leq c \left(
(1+\norm{\nabla z}_{L^q(\Omega)})\norm{b}_{\calY_{p_*}}
+ \norm{u}_{W^{1 ,p_{k_*-1}}(\Omega)}\norm{\nabla z}_{L^{q_*}(\Omega)}
   \right)
\\
&\leq c_{k_*} (1 + \norm{\nabla
z}_{L^q(\Omega)})^{k_*}\norm{b}_{\calY_{p_*}}.
\end{align*}
The remaining norm estimates in Lemma \ref{lem_reg_babadjan} follow from
interpolation theory. 
\end{proof}
\begin{example}
\label{example_nonsmooth}
A sufficient condition  such that
the interpolation scale property (A$_\Omega$1)(i) is satisfied is to assume
that $\Omega\subset\R^d$ is a bounded domain with Lipschitz boundary and that
the boundary sets  $\Gamma_D$ and $\Gamma_N$ are regular in the
sense of Gr\"oger
viz.,  loosely speaking, that the hypersurface separating $\Gamma_D$ and
$\Gamma_N$  is Lipschitz,
see \cite{ell:Gro89,GGKR02_interpol} for more details. A slightly more
general geometric setting is characterized in \cite{HDJKR12}. 

In order to obtain also the  isomorphism property (A$_\Omega$1)(ii),  one can
apply the regularity theory for linear elliptic systems in polyhedral domains.
Sufficient conditions on the geometry of $\Omega$, the Dirichlet and the
Neumann
boundaries can be identified for instance with the help of
  \cite[Theorem 7.1]{MR03_polyhedral}
(applied for $\vec\beta=0$, $\vec\delta=0$, $l=1$ in the notation of
\cite[Section 7]{MR03_polyhedral}).

For example, for  the Lam{\'e}-operator (i.e.\ linear
isotropic elasticity) sufficient conditions for (A$_\Omega$1) to hold  are the
following: $d=3$, $\Omega\subset \R^3$ is a bounded domain of polyhedral type
(see
\cite[Section 7.1]{MR03_polyhedral}),  and  on each face either Dirichlet boundary
conditions or Neumann boundary conditions are prescribed. Furthermore,
the interior opening angles along Dirichlet-Dirichlet and along Neumann-Neumann
edges are less than $2\pi$ (i.e.\ no cracks),  and  the interior opening angle
along Dirichlet-Neumann edges is less than or equal to $\pi$ (more general
 situations are possible). Then the singular exponents along the edges of the
polyhedron $\Omega$ satisfy the conditions required in \cite[Theorem
7.1]{MR03_polyhedral} in order to allow for $p_*>3$ in (A$_\Omega$1).  We refer
for example to \cite{sin:Nic92} for estimates of the singular exponents along
edges in different geometric settings. 
Concerning the singular exponents
 associated with the vertices of the polyhedron, one  has
to guarantee that the
strip $\Set{\lambda\in \C}{-\frac12 <\Re \lambda \leq 0}$ contains at most the
singular exponent $\lambda=0$. In the case of pure Dirichlet conditions in the
vicinity of a given vertex there is no further geometric restriction in order to
guarantee this property,   \cite{sin:MP84}. In case of pure Neumann conditions
in the vicinity of a given vertex one has to assume that the boundary locally is
the graph of a function that is positively homogeneous of degree one,
\cite{sin:KM88}. In case of mixed boundary conditions in the vicinity of a
vertex, a sufficient condition is to assume that  the domain is convex in the
vicinity of the vertex and that at most one face belongs to the Dirichlet
boundary or that at most one face belongs to the Neumann boundary (see
\cite{sin:Nic92} for a more general condition, an example is illustrated in
 Figure \ref{fig_fich}(ii)).
We refer to \cite{sin:Kne02} 
 for an overview on the literature on  estimates for the corner and edge
singularities associated with the Laplace- and
Lam{\'e}-operator on  three-dimensional polyhedral domains. Clearly,
(A$_\Omega$1) as well as Lemma \ref{lem_reg_babadjan} can be extended to
coefficient matrices $\bbC$ with piecewise constant entries if certain
geometric conditions are satisfied.
The Fichera corner plotted in  Figure \ref{fig_fich} is  an example for a
 nonconvex, nonsmooth domain with mixed
boundary conditions that is admissible with respect to assumption (A$_\Omega$1),
in
connection with the Lam{\'e}-operator.
\end{example}
\begin{figure}[t]
\setlength{\unitlength}{1cm}
\begin{picture}(3,3)
\put(0,0){\includegraphics[width=3cm]{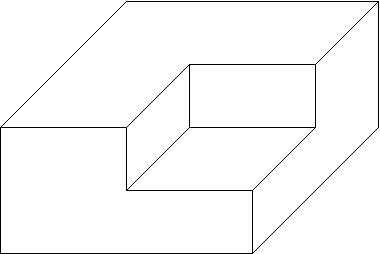}}
\put(0.7,0.1){\small $\Gamma_1$}
\end{picture}
\caption{Admissible domain if for example: (i)
Dirichlet-conditions on the bottom plane and Neumann conditions on the
remaining part of $\partial\Omega$ or (ii) $\Gamma_D=\Gamma_1$ and
 Neumann conditions on the rest.}
\label{fig_fich}
\end{figure}

\subsection{Properties of the energy functional}
\label{ss:2.3}
In what follows, we prove the
continuity and differentiability properties of
$\calI$ needed for our analysis. The following results shall also provide
fine estimates for  $|\partial_t \calI |$ and  for suitable norms of $\rmD_z \calI$, in terms of quantities which
continuously depend on $\norm{z}_{\calZ}$, and which are therefore bounded on sublevels of $\calI$.

 Hereafter, we shall work under these
additional conditions on the data $\ell$ and $u_D$:
\begin{assumption}
\label{ass:init}
We require that
\begin{align}
\label{a:data}
 \ell\in \rmC^{1,1}([0,T];W^{-1,p_*}_{\Gamma_D}(\Omega;\R^d)),
\quad u_D\in \rmC^{1,1}([0,T];W^{1,p_*}(\Omega;\R^d))\,\,\text{ with $p_*>d$
from  (A$_\Omega$1).}
\end{align}
\end{assumption}

From now on, in order to shorten the notation we introduce for $z_1,z_2\in \calZ$
and $k_*$ as
in Lemma \ref{lem_reg_babadjan} the quantity
\begin{align}
\label{def:Pzz}
 P(z_1,z_2):=
(1 + \norm{\nabla z_1}_{L^q(\Omega)} + \norm{\nabla z_2}_{L^q(\Omega)})^{k_*}.
\end{align}

 Our first  result   is based on Lemma
\ref{lem_reg_babadjan}.
\begin{lemma}[Existence of minimizers and their regularity]
\label{l:ex_min}
$ $\\
 Under Assumptions \ref{assumption:energy}, \ref{ass:init},
and (A$_\Omega$1),
 for every $(t,z)\in [0,T]\times \calZ$ there exists
a unique $\umin(t,z)\in \calU$, which minimizes $\calE(t,z,\cdot)$.

 Moreover, there exist $c_0>0$ 
%
such that for all
    $p\in [p_*',p_*]$ 
 and $(t,z)\in [0,T]\times \calZ$ it holds that
 $\umin(t,z)\in W^{1, p}_{\Gamma_D}(\Omega;\R^d)$, and
\begin{align}
\label{e:w1p} 
  \norm{\umin(t,z)}_{W^{1, p}(\Omega;\R^d)}   \leq
c_0
 P(z,0)^\frac{p_*\abs{p-2}}{p(p_*-2)}
\big(\norm{\ell(t)}_{W^{-1,
p}_{\Gamma_D}(\Omega;\R^d)}
+  \norm{u_D(t)}_{W^{1, p}(\Omega;\R^d)} \big),
\end{align}
with $P$ as in \eqref{def:Pzz}, 
 and $p_*$ the exponent from  (A$_\Omega$1)(ii).
Furthermore, the
following coercivity inequality for $\calI$ is valid: There exist
constants $c_1,c_2>0$ such that for
 all $(t,z)\in [0,T]\times \calZ$ it holds
 \begin{align}
\label{est_coerc1} \calI(t,z)&\geq c_1\big(
 \norm{\nabla z}^\il_{L^\il(\Omega)}
 + \norm{z}_{L^1(\Omega)}
 + \norm{\umin(t,z)}_{W^{1,2}(\Omega;\R^d)}^2
 \big) -c_2.
\end{align}
 \end{lemma}
 \noindent
  Observe that, on the right-hand side of \eqref{e:w1p} the dependence
 on $\|z\|_{\calZ}$
 of the
 quantity which bounds $ \norm{\umin(t,z)}_{W^{1, p}(\Omega;\R^d)}$
  is very explicitly displayed. In particular, observe that for $p=2$ we have no dependence on  $\|z\|_{\calZ}$  as
 $P(z,0)^0=1$, while for the extreme case $p=p_*$ we have $P(z,0)$, cf.\ \eqref{interesting-exponent-added}.

Our next result provides an estimate for the difference of two minimizers
$\umin(t_1,z_1)$ and $\umin(t_2,z_2)$, in terms of the data $t_i$ and $z_i$,
$i=1,2$.
\begin{lemma}[Continuous dependence on the data]
\label{l:cddata}
$ $\\
Under Assumptions \ref{assumption:energy},  \ref{ass:init}, and (A$_\Omega$1),
 there exists a constant $c_3>0$ such that for all $\ell$ and $u_D$  as in
 \eqref{a:data},  all $t_1,t_2\in [0,T]$ and
 all 
  {$\wt p\in [p_*',p_*)$ }
it holds
 with
 $r= p_* \wt p (p_* - \wt p)^{-1}$ and all $z_1,z_2\in
\calZ$
 \begin{multline}
\label{est_cont1}
\norm{\umin(t_1,z_1) - \umin(t_2,z_2)}_{W^{1,\wt
    p}(\Omega;\R^d)}
\\
\leq c_3 \big(\abs{t_1-t_2} + \norm{z_1-z_2}_{L^r(\Omega)} \big)
P(z_1,z_2)^2
\big(  \norm{\ell}_{\rmC^1([0,T]; W_{\Gamma_D}^{-1,p_*}(\Omega;\R^d))} +
\norm{u_D}_{\rmC^1([0,T];W^{1,p_*}(\Omega;\R^d))}  \big).
\end{multline}
\end{lemma}
\begin{remark}
\label{rmk:exponents-later-use}
\upshape
Observe that  for $\wt p\in [p_*',p_*)$ we have $r=p_* \wt p (p_* - \wt p)^{-1}
\in [ \frac{p_*}{p_*-2}, \infty)$, and $r$ is strictly increasing with respect to
$\wt p$.  In particular, for $\wt p= 3 p_*(3 + p_*)^{-1}$ we have $r=3$ and for
$\wt p:= 6 p_*(6 + p_*)^{-1}$ we have $r=6$.
\end{remark}

\begin{proof}
For $i=1,2$, let $u_i:=\umin(t_i,z_i)\in W^{1, p_*}(\Omega;\R^d)$, with $p_*$
from (A$_\Omega$1). From the corresponding Euler-Lagrange
equations 
 written for $u_{i}$, $i=1,2$, with
algebraic manipulations we obtain that $u_1-u_2$ satisfies for   all
$v\in W^{1,p_*}_{\Gamma_D}(\Omega;\R^d)$
\begin{equation}
\label{instrumental}
\begin{aligned}
  &   \int_\Omega g(z_1)\bbC \varepsilon(u_1-u_2){\colon}
\varepsilon(v)\dx
  =\int_\Omega
\big(g(z_2)-g(z_1)\big)\bbC\varepsilon(u_2){\colon}\varepsilon(v)\dx
\\
& \quad  \quad \qquad \qquad -\int_\Omega\big(g(z_1)\bbC\varepsilon(u_D(t_1)) -
g(z_2)\bbC\varepsilon(u_D(t_2))\big){\colon}\varepsilon(v)\dx +
\pairing{\calU}{\ell(t_1)-\ell(t_2)}{v}.
\end{aligned}
\end{equation}
Hence, by density and Lemma \ref{lem_reg_babadjan},
 the function $u_1-u_2$ fulfills for all $\wt p\in [p_*', p_*]$ and all
 $v\in W^{1,\wt p'}_{\Gamma_D}(\Omega;\R^d)$ the relation
\[
\int_\Omega g(z_1)\bbC \varepsilon(u_1-u_2):\varepsilon(v)\dx=
\pairing{W^{1,\wt p'}_{\Gamma_D}(\Omega;\R^d)}{\tilde{\ell}_{1,2}}{v},
\]
where $\tilde{\ell}_{1,2} \in W_{\Gamma_D}^{-1,\wt p}(\Omega;\R^d)$
subsumes the terms on the right-hand side of \eqref{instrumental}.
Therefore, \eqref{crucial-reg-estimate} gives
\[
\norm{u_1-u_2}_{W^{1,\wt p}(\Omega;\R^d)} \leq c_0
P(z_1,0)
\norm{\tilde{\ell}_{1,2}}_{W_{\Gamma_D}^{-1,\wt p}(\Omega;\R^d)},
\]
whence we deduce the estimate
\begin{multline}
\label{est-inter} \norm{u_1-u_2}_{W^{1,\wt p}(\Omega;\R^d)} \leq
c_0
P(z_1,0)
\big( \norm{\ell(t_1)-\ell(t_2)}_{W_{\Gamma_D}^{-1,\wt p}(\Omega;\R^d)}
+
  \norm{(g(z_1)-g(z_2))\bbC \varepsilon(u_2)}_{L^{\wt p}(\Omega;\R^d)} \\
+  \norm{g(z_1)\bbC \varepsilon(u_D(t_1)) -
  g(z_2)\bbC \varepsilon(u_D(t_2))}_{L^{\wt p}(\Omega;\R^d)}  \big).
\end{multline}
 Now, the Lipschitz continuity of $g$ and H\"older's inequality
imply that
\begin{equation}
\label{auxiliary}
\begin{aligned}
 \norm{(g(z_1)-g(z_2))\bbC \varepsilon(u_2)}_{L^{\wt p}(\Omega;\R^d)}   & \leq
 C \norm{z_1-z_2}_{L^r(\Omega)} \norm{\varepsilon
(u_2)}_{L^{p_*}(\Omega;\R^d)}
\\ &  \leq C'
 P(z_2,0)
\norm{z_1-z_2}_{L^r(\Omega)}
\end{aligned}
\end{equation}
with $r= p_* \wt p (p_* - \wt p)^{-1}$,
where the second inequality follows from condition
  \eqref{a:data} and from estimate~\eqref{e:w1p}.
 We
use \eqref{auxiliary} to estimate the second term on the right-hand
side of~\eqref{est-inter}. In a similar way the third summand is
treated, where we use again 
 \eqref{a:data}.
%
%
%
%
\end{proof}

We now address the differentiability properties of
$\calI$ as a function of  the time variable $t$. For the proof of Lemma
\ref{l:diff_time} below,
   the calculations are similar to those in
\cite[Lemma 2.3]{krz}, 
taking into account estimates \eqref{e:w1p} and \eqref{est_cont1}, therefore we choose not to detail
them.
\begin{lemma}[Differentiability and growth w.r.\ to time]
\label{l:diff_time}
$ $\\
 Under Assumptions \ref{assumption:energy},  \ref{ass:init}, and (A$_\Omega$1),
  for every $z\in \calZ$ the map $t\mapsto
\calI(t,z)$ belongs to $\rmC^{1}([0,T],\R)$ with
\begin{align}
\label{form-derivative} \partial_t\calI(t,z)= \int_\Omega g(z)\bbC
\varepsilon(\umin(t,z) + u_D(t)){\colon}\varepsilon(\dot u_D(t)) \dx -
\langle \dot\ell(t),\umin(t,z)\rangle_{W^{1,2}_{\Gamma_D}(\Omega;\R^d)}.
\end{align}
Moreover, there exists a constant $c_4>0$ such that for all $t\in
[0,T]$, $z\in \calZ$ and $u_D$, $\ell$ with \eqref{a:data} we
have
\begin{align}
\label{stim3} \abs{\partial_t\calI(t,z)}&\leq c_4 
\big(\norm{u_D}^2_{\rmC^1([0,T];W^{1,2}(\Omega;\R^d))} +
\norm{\ell}^2_{\rmC^1([0,T];W^{-1,2}_{\Gamma_D}(\Omega;\R^d))} \big).
\end{align}
Finally,
 there exists a constant $c_5>0$ depending on
$\norm{\ell}_{\rmC^{1,1}([0,T];W^{-1,p_*}_{\Gamma_D}(\Omega;\R^d))}$ and
$\norm{u_D}_{\rmC^{1,1}([0,T];W^{1,p_*}(\Omega))}$ such that
for all $r\in [\frac{p_*}{p_*-2},\infty)$,
for all $t_i\in
  [0,T]$ and $z_i\in \calZ$  we have
\begin{align}
\label{stim5} \abs{\partial_t\calI(t_1,z_1) -
\partial_t\calI(t_2,z_2)} \leq
c_5
P(z_1,z_2)^2
\big(\abs{t_1-t_2} +
\norm{z_1-z_2}_{L^{r}(\Omega)}\big).
\end{align}
\end{lemma}

The differentiability of $\calI$ with respect to $z$ will be studied
in the $\calZ-\calZ^*$ duality. In particular,
$\rmD_z\calI(t,\cdot): \calZ\to \calZ^*$ shall denote the
G\^ateaux-differential of the functional $\calI(t,\cdot)$. We have the following
result, whose proof is
completely analogous to the proof of \cite[Lemma 2.4]{krz}.
\begin{lemma}[G\^ateaux-differentiability]
\label{l:gateaux}
$ $\\
Under Assumptions \ref{assumption:energy}, \ref{ass:init}, and
 (A$_\Omega$1), for all $t\in [0,T]$ the functional
$\calI(t,\cdot):\calZ \rightarrow\R$ is G\^ateaux-differentiable at all
$z\in \calZ$, and for all $\eta\in \calZ$ we have
\begin{align}
\label{form-gateau-derivative} \pairing{\calZ}{\rmD_z\calI(t,z)}{\eta}=
\pairing{\calZ}{ A_q z}{\eta} +
\int_\Omega f'(z)\eta\dx + \int_\Omega g'(z)\wt W(t,\nabla
\umin(t,z)) \eta\dx,
\end{align}
where we use the abbreviation $\wt W(t,\nabla v)= W(\nabla v +
\nabla u_D(t)) =\frac{1}{2} \bbC\varepsilon(v +
u_D(t)){ : }\varepsilon(v + u_D(t))$. In particular, the following
estimate holds with a constant $c_6$ that depends on the data $\ell, u_D$, but
is independent of $t$ and $z$:
\begin{equation}
\label{esti-gateau}  \forall\, (t,z) \in
[0,T]\times \calZ\, :  \ \ \norm{\rmD_z\calI(t,z)}_{\calZ^*} \leq
  c_6 \left( \|z\|_{\calZ}^{q-1} +
\norm{f'(z)}_{L^\infty(\Omega)}
+ 1 \right).
\end{equation}
\end{lemma}
%
%
%
We define
\begin{equation}
\label{itilde}
\wt \calI(t,z):= \calI_2(t,z) + \int_\Omega f(z)\dx \quad \text{for
all $(t,z) \in [0,T]\times \calZ$}
\end{equation}
with $\calI_2$ from \eqref{reduced-energy} as the part of the reduced energy
collecting all lower order terms.
Accordingly, $\rmD_z \calI$ from \eqref{form-gateau-derivative} decomposes as
\begin{equation}
\label{der-decomp}
\rmD_z \calI(t,z) = A_q z + \rmD_z\wt{\calI}(t,z) \quad \text{for all } (t,z) \in [0,T]\times \calZ.
\end{equation}
 In Lemma~\ref{l:diff_contz}
below we prove that the maps $(t,z) \mapsto \wt{\calI}(t,z(t))$
and that $(t,z) \mapsto \rmD_z\wt{\calI}(t,z)$ are Lipschitz continuous
w.r.t.\
a suitable \emph{Lebesgue} norm. In view of this, and in order to
emphasize  that, in \eqref{der-decomp}, $\rmD_z\wt{\calI}(t,z) $ is a lower order term w.r.t.\ $A_q z$,
 from
now on we shall often resort to the following
\begin{notation}[Abuse of notation for $ \rmD_z\wt{\calI}(t,z)$]
\label{not:abuse}
In view of  \eqref{form-gateau-derivative}, the term
$\rmD_z\wt\calI(t,z)$ can be identified with an element of  $L^\mu(\Omega)$ for
some $\mu\geq 1$. The quantity $\norm{\rmD_z\wt
\calI(t,z)}_{L^\mu(\Omega)}$ will be interpreted in this sense, and with the symbol
$\rmD_z\wt\calI$, we shall denote both
the derivative of $\wt\calI$ as an operator and the corresponding
density in $L^1(\Omega)$.   Accordingly,
we shall write
\begin{equation}
\label{abuse-integral}
\text{ for  a given } v \in L^{\mu'}(\Omega) \quad
\int_\Omega  \rmD_z\wt\calI(t,z)
v  \,\dd x \quad \text{ in place of } \quad
 \langle  \rmD_z\wt\calI(t,z),
v\rangle_{L^{\mu'}(\Omega)}.
\end{equation}
\end{notation}

 For $h\in \rmC^0(\R)$ and $z_1,z_2\in\calZ$ let
\begin{equation}
\label{C_h}
C_{h}(z_1,z_2)= \max \Set{ |h(s)|}{ \ |s| \leq
 \|z_1\|_{L^\infty(\Omega)} + \|z_2 \|_{L^\infty(\Omega)} }.
\end{equation}
This notation will be used along the proof of the following lemma.
\begin{lemma}[Local Lipschitz continuity of $\wt \calI$ and $\rmD_z\wt{\calI}$]
\label{l:diff_contz}
$ $\\
Under Assumptions \ref{assumption:energy}, \ref{ass:init}, and
 (A$_\Omega$1),
%
there exists a constant $c_7>0$
depending on
$\norm{\ell}_{\rmC^{1,1}([0,T];W^{1,-p_*}_{\Gamma_D}(\Omega;\R^d))}$ and
$\norm{u_D}_{\rmC^{1,1}([0,T];W^{1,p_*}_{\Gamma_D}(\Omega;\R^d))}$
such that for all  $t_i\in [0,T]$ and all $z_i\in \calZ$ it holds
\begin{equation}
\begin{aligned}
\label{Lip-cont-I}
\left| \wt{\calI} (t_1, z_1) - \wt{\calI} (t_2, z_2) \right|\leq c_7 (C_{f'}(z_1,z_2) +
 P(z_1,z_2)^{2}) \left( \abs{t_1-t_2}+\norm{z_1
-z_2}_{L^{2p_*/(p_*-2)}(\Omega)}
  \right),
\end{aligned}
\end{equation}
with $C_{f'}(z_1,z_2)$  as in \eqref{C_h}  (corresponding to  $
h(x)=f'(x) $).
Further, for every $\mu\in [1,p_*/2)$,
\begin{equation}
\begin{aligned}
\label{enhanced-stim-7}
 \norm{\rmD_z \wt \calI(t_1,z_1) - \rmD_z \wt \calI(t_2,z_2) }_{L^\mu(\Omega)}
&\leq c_7
\big(C_{f''}(z_1,z_2)
\\ & \qquad
+(1 + C_{g''}(z_1,z_2)) P(z_1,z_2)^3 \big)
%
%
%
\big( \abs{t_1-t_2}+\norm{z_1 -z_2}_{L^r(\Omega)}    \big),
\end{aligned}
\end{equation}
where $r=p_*\mu(p_* - 2\mu)^{-1}$,
and  for $\mu\in  [1,p_*/2]$,
\begin{equation}
\label{estimate-for-DI}
	\forall\, (t,z) \in [0,T] \times \calZ \, : \qquad \| \rmD_z \wt
\calI(t,z)  \|_{L^{\mu}(\Omega)}  \leq
c_7(1 + \norm{f'(z)}_{L^\infty(\Omega)} + P(z,0)^2).
\end{equation}
\end{lemma}
\begin{proof}
 In order to prove estimate \eqref{Lip-cont-I}, with elementary calculations we
observe that
\[
\begin{aligned}
\left| \wt\calI(t_1,z_1) - \wt\calI(t_2,z_2) \right| &  \leq
\int_\Omega |f(z_1)-f(z_2)| \,\dd x + \int_\Omega |g(z_1) - g(z_2)| |\wt
W(t_1,\nabla u_1)| \,\dd x
\\ & \quad + \int_\Omega |g(z_2)|  |\wt W(t_1,\nabla u_1)-\wt W(t_2,\nabla u_2)
| \,\dd x
 + |\pairing{\calU}{\ell(t_1) - \ell(t_2)}{u_1}|
 \\
 & \quad
+  |\pairing{\calU}{\ell(t_2)}{u_1-u_2}| \doteq I_1 +I_2 +I_3+I_4+I_5,
\end{aligned}
\]
where $u_i:=u_\text{min}(t_i,z_i)\in W^{1,p_*}(\Omega;\R^d)$ for $i=1,2$.
Since $f \in \rmC^1(\R)$ (cf.\ \eqref{ass-eff}),
\begin{equation}
\label{f-Lip}
 I_1 \leq C_{f'} (z_1,z_2) \| z_1 - z_2\|_{L^1(\Omega)} .
\end{equation}
Moreover, using \eqref{elast-1}, \eqref{new-g}, \eqref{a:data}, and the H\"older
inequality,
\[
\begin{aligned}
 I_2    \leq \| g(z_1) - g(z_2) \|_{L^{p_*/(p_*-2)} (\Omega)} \| \wt
W(t_1,\nabla u_1)  \|_{L^{p_*/2}(\Omega)}
 & \leq C \|z_1 - z_2 \|_{L^{p_*/(p_*-2)} (\Omega)} \| u_1
\|_{W^{1,p_*}(\Omega;\R^d)}^2
\\ & \leq C P(z_1,0)^2 \|z_1 - z_2 \|_{L^{p_*/(p_*-2)} (\Omega)}
\end{aligned}
\]
where  the constant $C$ also incorporates the data
and the last
inequality follows from
\eqref{e:w1p}  (with $p=p_*$).
Analogously,
\begin{align*}
 I_3  & \leq C  \int_\Omega |g(z_2)| |\varepsilon(u_1 +u_D(t_1))
+\varepsilon(u_2 +u_D(t_2))   | |\varepsilon(u_1 +u_D(t_1)) - \varepsilon(u_2
+u_D(t_2))   |\,\dd x
\\ &
 \leq C (\|u_1 + u_2\|_{W^{1,2}(\Omega;\R^d)} + 1)( \|u_1 - u_2\|_{W^{1,2}(\Omega;\R^d)}
+ \|u_D(t_1) - u_D(t_2)\|_{W^{1,2}(\Omega;\R^d)} )
\\
&
\leq C P(z_1,z_2)^2  (|t_1-t_2| +  \|z_1 - z_2 \|_{L^{2p_*/(p_*-2)} (\Omega)})
\end{align*}
due to \eqref{new-g} and  \eqref{a:data} and, for the last inequality, to
\eqref{e:w1p}  (with $p=2$),
  and \eqref{est_cont1} with  $p=2$, whence  $r= 2p_*/(p_*-2)$.
Finally,
\[
\begin{aligned}
& I_4 \leq \| \ell(t_1) - \ell(t_2)\|_{W^{-1,p_*}(\Omega;\R^d)} \| u_1
\|_{W^{1,p_*}(\Omega;\R^d)} \leq C |t_1-t_2|P(z_1,z_2),
\\
& I_5 \leq  \|  \ell(t_2)\|_{W^{-1,2}(\Omega;\R^d)} \| u_1-u_2
\|_{W^{1,2}(\Omega;\R^d)}
 \leq C P(z_1,z_2)^2  (|t_1-t_2| +  \|z_1 - z_2 \|_{L^{2p_*/(p_*-2)} (\Omega)})
\end{aligned}
\]
where the first estimate is due to \eqref{a:data} and \eqref{e:w1p}, and the
second one follows from $\ell \in \rmC^0 ([0,T]; W_{\Gamma_D}^{-1,2}(\Omega;\R^d)
)$ and again  \eqref{est_cont1}.
Collecting the above calculations, we conclude  \eqref{Lip-cont-I}.

Since $f'$ is  locally Lipschitz,
 for the
proof of \eqref{enhanced-stim-7} we confine ourselves to  investigating the
properties of
$\rmD_z\calI_2$, given by \eqref{reduced-energy}.
Let $\mu\in [1,p_*/2)$.  We have
\begin{align*}
&\norm{\rmD_z\calI_2(t_1,z_1)
-
  \rmD_z\calI_2(t_2,z_2)}_{L^{\mu}(\Omega)}
\\
&\qquad\leq
C_{g''}(z_1,z_2)
\norm{z_1 -z_2}_{L^r(\Omega)}\norm{\wt W(t_1,\nabla
  u_1)}_{L^{p_*/2}(\Omega)}
+ c \norm{\wt W(t_1,\nabla u_1) -
  \wt W(t_2,\nabla u_2)}_{L^{\mu}(\Omega)}
\\
&\qquad\leq
 C(1 + C_{g''}(z_1,z_2))P(z_1,z_2)^3
%
\big( \abs{t_1 -t_2}+\norm{z_1 -z_2}_{L^r(\Omega)} \big)
\end{align*}
and $C$ depends on the data $\ell$ and $u_D$.
 Indeed, the first inequality follows from
the form of  $\rmD_z \calI_2$ (cf.\ \eqref{form-gateau-derivative}).  The second one
ensues from \eqref{new-g} and
 the H\"older inequality for the first term, which is then estimated by means
  of \eqref{e:w1p} Lemma
\ref{l:ex_min}. For the term $\norm{\wt W(t_1,\nabla u_1) -
  \wt W(t_2,\nabla u_2)}_{L^{\mu}(\Omega)}$, we  again use the H\"older inequality. Ultimately,  this leads us to estimate the quantity
$  \| u_1 +u_2\|_{W^{1,p_*} (\Omega;\R^d)}$, for which we use  \eqref{e:w1p},  and the quantity
$  \| u_1 - u_2\|_{W^{1,\wt p} (\Omega;\R^d)}$ with
$\wt p= \mu p_*/(p_* - \mu)$ , for which we use
    \eqref{est_cont1}
   (observe that $\mu \leq \frac{p_*}{2}$),
   with
 $r=
p_*\mu/(p_* - 2\mu)$
 (indeed, $r=p_* \wt p/(p_* - \wt p)$).
 With completely analogous
calculations, we prove \eqref{estimate-for-DI}.
\end{proof}

{F}rom \eqref{enhanced-stim-7}
we deduce  estimate  \eqref{very-useful-later} below,
 which pops up in several of the  calculations  in Sec.\ \ref{s:5}.
\begin{corollary}\label{c:very-useful-later}
Under Assumptions \ref{assumption:energy}, \ref{ass:init}, and
 (A$_\Omega$1),  for every $w\in\calZ$  there holds
\begin{equation}
\label{very-useful-later}
\begin{aligned}
 &
\! \!\!\! \!\!  \!\!  \!
\abs{\langle  \rmD_z \wt\calI (t_1, z_1) - \rmD_z \wt\calI (t_2,
 z_2) , w \rangle_\calZ}
\\ & \, \! \!\!\!  \!\!  \!\!  \!   \leq c_7 
(C_{f''}(z_1,z_2) +(1 + C_{g''}(z_1,z_2)) P(z_1,z_2)^3)
 (|t_1-t_2| +
\|z_1-z_2\|_{L^{2p_*/(p_*-2)}(\Omega)} )
\| w \|_{L^{2p_*/(p_*-2)}(\Omega)}.
\end{aligned}
\end{equation}
\end{corollary}
\begin{proof} To check \eqref{very-useful-later}, we use the H\"older inequality
and  estimate
 \[
\abs{\langle  \rmD_z \wt\calI (t_1, z_1) - \rmD_z \wt\calI (t_2,
 z_2) , w \rangle_\calZ}
  \leq \|\rmD_z \wt\calI (t_1, z_1) - \rmD_z \wt\calI (t_2,
z_2)\|_{L^{2p_*/(p_*+2)} (\Omega)} \| w  \|_{L^{2p_*/(p_*-2)} (\Omega)},
 \]
(cf.\ Notation \ref{not:abuse}),
which in combination with  \eqref{enhanced-stim-7}
 for $\mu =2p_*/(p_*+2) $ and $r=2p_*/(p_*-2)  $ implies~\eqref{very-useful-later}.
\end{proof}
%
%
%
\noindent We also have \bnnc the following monotonicity property for $\rmD_z\calI$. \ennc
\begin{corollary}
\label{cor:rosenheim}
$ $\\
Under Assumptions \ref{assumption:energy}, \ref{ass:init}, and
 (A$_\Omega$1),
there exist constants $c_8,\,c_{9}>0$ such that for all $t \in [0,T]$
and $z_i \in \calZ$, $i=1,2$, we have
\begin{equation}
\label{e:strong-monot-sez2}
\norm{z_1{-}z_2}_{L^2(\Omega)}^2+ \pairing{\calZ}{\rmD_z\calI(t,z_1){-}\rmD_z\calI(t,z_2)}{z_1{-}z_2} \geq c_8 \norm{z_1{-}z_2}_{W^{1,2}(\Omega)}^2 -
c_{9} \norm{z_1{-}z_2}_{L^{2} (\Omega)}^2\,.
\end{equation}
\end{corollary}
\begin{proof}
We observe that by \eqref{e2.22} and \eqref{very-useful-later}  there holds
\[
\begin{aligned}
 &  \norm{z_1{-}z_2}_{L^2(\Omega)}^2 + \pairing{\calZ}{\rmD_z\calI(t,z_1){-}\rmD_z\calI(t,z_2)}{z_1{-}z_2}
 \\
 &  =
 \norm{z_1{-}z_2}_{L^2(\Omega)}^2 + \pairing{\calZ}{A_q z_1 - A_q z_2} {z_1 - z_2} +
\pairing{\calZ}{\rmD_z\wt\calI(t,z_1){-}\rmD_z\wt\calI(t,z_2)}{z_1{-}z_2}
\\ & \geq
 \norm{z_1{-}z_2}_{L^2(\Omega)}^2 + c_q\int_{\Omega}(1 + |\nabla z_1|^2 +  |\nabla z_2|^2 )^{\frac{q-2}{2}}
 |\nabla z_1 - \nabla z_2|^2\, \dd x - c \|z_1-z_2\|^2_{L^{2p_*/(p_*-2)}(\Omega)}
\end{aligned}
\]
where $\wt\calI$ is defined as in \eqref{itilde}.
Then, \eqref{e:strong-monot-sez2} follows upon using \eqref{lions-magenes}.
\end{proof}
\noindent It is not difficult to check that Corollary \ref{cor:rosenheim} indeed implies that the functional $z \mapsto \calI(t,z)$ is
$\lambda$-convex w.r.t.\ the $L^2(\Omega)$-norm,  for some $\lambda \in \R$, viz.\ that
\[
\exists\, \lambda \in \R  \ \forall\, z_0,\, z_1 \in \calZ \ \forall\, \theta \in (0,1)\, : \quad
\calI(t, z_\theta) \leq (1-\theta)\calI(t,z_0) +\theta \calI(t,z_1) - \frac{\lambda \theta (1-\theta)}{2} \norm{z_0{-}z_1}_{L^{2} (\Omega)}^2\,.
\]
However, this property does not automatically guarantee the validity of the chain rule for $\calI$, cf.\ the discussion at the beginning of Sec.\
\ref{ss:3.1}.

As a summary of the previous lemmata we obtain
\begin{corollary}[Fr\'echet differentiability of $\calI$]
\label{coro-fre}
$ $\\
Under Assumptions \ref{assumption:energy}, \ref{ass:init}, and
 (A$_\Omega$1),
 the functional  $\calI $  is Fr{\'e}chet
differentiable on $[0,T]\times \calZ$ and
\begin{equation}
\label{strong-continuity} \text{$t_n\rightarrow t$ and
$z_n\to z$ strongly in $\calZ$ implies }   \rmD_z\calI(t_n,z_n) \to
\rmD_z\calI(t,z) \ \text{strongly  in} \  \calZ^*.
\end{equation}
Furthermore,
\begin{equation}
\label{weak-continuity}
\begin{aligned}
&
\text{$t_n\rightarrow t$ and
$z_n\rightharpoonup z$ weakly in $\calZ$
 implies }
\\
&
\liminf_{n \to \infty}\calI(t_n,z_n)\geq  \calI(t,z), \quad
\wt\calI(t_n,z_n)\rightarrow \wt\calI(t,z), \quad
\partial_t\calI(t_n,z_n)\rightarrow \partial_t\calI(t,z), \\
&
 \rmD_z\wt\calI(t_n,z_n) \to
\rmD_z\wt\calI(t,z) \ \text{strongly in} \  \calZ^*.
\end{aligned}
\end{equation}
\end{corollary}
\begin{proof}
Observe that $z_n \to z $ in $\calZ$ implies $\rmD_z\calI_\il (z_n) \to
\rmD_z\calI_\il(t,z) $ in $\calZ^*$.  Therefore, in view of Lemma
\ref{l:diff_contz},
  the G\^ateaux-differential  $\rmD_z \calI $
fulfills \eqref{strong-continuity}, which yields that $\calI$ is Fr\'echet
differentiable.
 The  continuity property~\eqref{weak-continuity} of
$\partial_t\calI$ and $\rmD_z \wt\calI$ is an immediate consequence of
estimates \eqref{stim5} and \eqref{enhanced-stim-7}, and of the compact
embedding $\calZ \Subset L^r (\Omega)$ for all $1 \leq r \leq \infty$.
\end{proof}

%
%


\section{The viscous problem}
\label{s:3}
We now address the analysis of the viscous $L^2$-regularization of \eqref{dndia-eps}
of the rate-independent system \eqref{dndia}. To this aim,
we introduce the \emph{viscous} dissipation
potential
\begin{equation}
\label{def-viscous-intro}
\calR_{\epsilon} = \calR_1 + \calR_{2,\epsilon}  \quad\text{ with }
\calR_{2,
\epsilon}(\eta)=\frac{\epsilon}{2}
\norm{\eta}^2_{L^2(\Omega)},
\end{equation}
with $\calR_1$ from \eqref{def_R1}.
We denote by
 $\partial \calR_{\epsilon} : \calZ \rightrightarrows
\calZ^*$  its subdifferential (in the sense of convex
analysis),
in the duality between $\calZ^*$ and $\calZ$,
 and consider
\emph{viscous} doubly nonlinear evolution equation
\begin{equation}
\label{visc-eps-dne}
 \partial \calR_{\epsilon}(z'(t)) +
\rmD_z\calI(t,z(t))\ni 0 \quad \text{in $\calZ^*$ } \ \foraa\, t \in
(0,T),
\end{equation}
with the initial condition, featuring  $z_0 \in \calZ$,
\begin{equation}
\label{Cauchy-condition} z(0)=z_0.
\end{equation}
It follows from \cite[Cor.~IV.6]{aubin-ekeland84} that
\begin{equation}
\label{aforementioned-form}
 \partial \calR_{\epsilon}(\eta) = \partial \calR_1 (\eta) + \epsilon \eta \qquad \text{for all }
 \eta \in \calZ.
 \end{equation}
  Thus, also taking into account formula \eqref{form-gateau-derivative}
for $\rmD_z \calI$, we see that
\eqref{visc-eps-dne} translates into
\begin{equation}
\label{eps-reformulation}
\partial \calR_1 (z'(t)) + \epsilon  z'(t)
  + A_\il(z(t)) + f'(z(t)) + g'(z(t))\wt W(t,\nabla \umin(t,z(t))) \ni 0
 \text{ in }\calZ^*   \  \foraa \  t \in (0,T).
 \end{equation}
 \subsection{Weak solutions: definition and existence result}
 \label{ss:3.1}
 We are going to prove an existence result for a suitable weak solution
notion for
the Cauchy problem associated with \eqref{visc-eps-dne}.
 Before defining such a concept, let us expound the reasons why 
we 
do not treat 
\eqref{visc-eps-dne} as a pointwise-in-time differential inclusion in
$\calZ^*$.
 Indeed, \eqref{visc-eps-dne} is equivalent to $-A_{\il} z(t) - \rmD_z
\wt{\calI}(t,z(t)) \in \partial \calR_{\epsilon}(z'(t))$ for almost all $t \in
(0,T)$ with $\wt{\calI}$ from \eqref{itilde}, viz.
\begin{equation}
\label{duality-problem}
\calR_\epsi(w) - \calR_\epsi (z'(t)) \geq \pairing{\calZ}{-A_{\il} z(t) - \rmD_z
\wt{\calI}(t,z(t))}{w-z'(t)} \quad \text{for all } w \in \calZ, \quad \foraa\, t
\in (0,T).
\end{equation}
In fact,  \eqref{duality-problem} implies the information that $z'(t) \in \calZ$ for almost all $t \in (0,T)$.
However, as we are going to show in what follows, the best spatial
regularity for $z'(t)$ we can obtain is 
$z'(t) \in W^{1,2}(\Omega)$, 
which is less than 
 $z'(t) \in \calZ$. For achieving the latter, 
given a sequence of approximate solutions
$(z_k)_k$ to
\eqref{visc-eps-dne} (in our case, constructed by time-discretization), one
would have to dispose of a $\calZ$-estimate for  $(z_k')_k$, uniform w.r.t.\ $k
\in \N$. This seems to be
out of reach, due to the doubly nonlinear character of \eqref{visc-eps-dne}, and
in particular to the fact that the multivalued, unbounded operator
$\partial\calR_\epsi$
acts on $z'(t)$.

Another possibility would be to interpret the duality pairing $\pairing{\calZ}{A_\il z(t) +  \rmD_z \wt{\calI}(t,z(t))}{w-z'(t)}$
as a duality pairing between Lebesgue spaces, i.e.\ $\int_\Omega (A_\il z(t) +  \rmD_z \wt{\calI}(t,z(t)))(w-z'(t)) \,\dd x $.
For this to make sense, it is needed that $z'(t) \in L^\sigma (\Omega)$  and
$A_\il z(t) +  \rmD_z \wt{\calI}(t,z(t)) \in L^{\sigma'} (\Omega)$ for some
$\sigma \in [1,\infty)$.
 This boils down to proving that $A_\il z(t) \in L^{\sigma'} (\Omega)$, since
the term $\rmD_z \wt{\calI}(t,z(t))$ may be considered of lower order due to
estimate
 \eqref{enhanced-stim-7}, and indeed $\int_\Omega  \rmD_z \wt{\calI}(t,z(t))
z'(t)\, \dd x $ makes sense thanks to \eqref{reg-z-below} (cf.\ Notation \ref{not:abuse} and the discussion in
the proof of Proposition \ref{l:ch-rule-tilde} ahead).

 However, an estimate for $(A_\il z_k)_k $ in $L^{\sigma'}(\Omega)$ (for a
sequence  of approximate solutions  $(z_k)_k$), is out of grasp,
 in our opinion, in the present context. Only for $\sigma=2$ it would be possible to estimate $(A_\il z_k)_k  $
 in  $L^\infty (0,T; L^{2}(\Omega))$, by testing (the approximate version of
\eqref{visc-eps-dne}) by the quantity $\partial_t (A_\il z_k + f'(z_k))$. Let us
mention that this technique
 is by now standard in the analysis of doubly nonlinear equations of the type
\eqref{visc-eps-dne} and dates back to 
 \cite{bfl}. Nonetheless,
 to carry out the  calculations attached to this test,   one  would need to \bnnc exploit \ennc
  elliptic regularity results for $u$, which hold in smooth domains,
 while in this paper we aim to work under minimal regularity requirements on $\Omega$.

 Because of these reasons, we need to resort to the weak solution
 concept in Definition \ref{def:wsol} below, where for general $q>d$
  the duality pairing $\pairing{\calZ}{A_\il z(t) }{z'(t)}$ is
replaced by  the quantity
 \begin{equation}
 \label{quantity-added}
 \int_\Omega (1+ |\nabla z(t)|^2)^{\frac{q-2}{2}} \nabla z(t) \cdot \nabla z'(t) \,\dd x.
 \end{equation}
\begin{definition}[Weak solution]
\label{def:wsol}
We say that
\begin{equation}
\label{reg-z-below}
z \in L^\infty (0,T;W^{1,\il}(\Omega))\cap W^{1,2} (0,T; W^{1,2}(\Omega))
\end{equation}
fulfilling
\begin{equation}
\label{mixed-estimate}
\int_0^T \int_\Omega (1+|\nabla z(r)|^2)^{\frac{\il-2}{2}} |\nabla z' (r)|^2 \,\dd x  \,\dd r <\infty \,,
\end{equation}
is a \emph{weak solution} of
\eqref{visc-eps-dne},  if it complies with the variational inequality
\begin{equation}
\label{weak-def-sol}
\begin{aligned}
\calR_\epsi(w) - \calR_\epsi (z'(t)) \geq  &  \pairing{\calZ}{-A_{\il} z(t) }{w}  +  \int_\Omega (1+ |\nabla z(t)|^2)^{\frac{q-2}{2}} \nabla z(t) \cdot \nabla z'(t)\, \dd x
\\ & \quad
-  \int_{\Omega}\rmD_z \wt{\calI}(t,z(t)) (w-z'(t))\, \dd x 
 \quad \text{for all } w \in \calZ \qquad \foraa\, t \in (0,T)\,.
\end{aligned}
\end{equation}
\end{definition}
\noindent Observe that \eqref{quantity-added} is well defined as soon as $z$ fulfills \eqref{mixed-estimate},
cf.\ \eqref{guarantees} below.
 Hereafter, we shall refer to \eqref{mixed-estimate} as ``mixed estimate'', for
it involves both $z$ and $z'$. In fact, \eqref{mixed-estimate} shall result from the a priori estimates on the time-discretization of \eqref{visc-eps-dne}, contained in Lemma~\ref{l:enhanced-reg}.

The regularity \eqref{mixed-estimate} also guarantees the validity of the following \emph{chain rule} formula
\begin{theorem}
\label{thm:chain-rule}
Under Assumptions  \ref{assumption:energy}, \ref{ass:init}, and
 (A$_\Omega$1), for every curve $z$ fulfilling \eqref{reg-z-below} and \eqref{mixed-estimate}, there holds:
 \begin{enumerate}
 \item
 the map $t \mapsto \calI (t,z(t)) $ is absolutely continuous on $(0,T)$;
 \item the following chain rule formula is valid:
 \begin{equation}
 \label{chain-rule-formula}
 \begin{aligned}
  & \frac{\dd }{\dd t} \calI (t,z(t)) - \partial_t \calI (t,z(t))
  \\
  &
   \quad =  \int_\Omega
(1+ |\nabla z(t)|^2)^{\frac{q-2}{2}} \nabla z(t) \cdot \nabla z'(t) \,\dd x
 + \pairing{W^{1,2}(\Omega) 
} {\rmD_z \wt \calI(t,z(t))}{
z'(t) }
\\
&
\quad
 =  \int_\Omega
(1+ |\nabla z(t)|^2)^{\frac{q-2}{2}} \nabla z(t) \cdot \nabla z'(t) \,\dd x
+ \int_\Omega \rmD_z \wt \calI(t,z(t))
z'(t)   \,\dd x
 \quad  \foraa\, t \in (0,T),
 \end{aligned}
 \end{equation}
 where for the second equality we refer to Notation \ref{not:abuse}.
\end{enumerate}
\end{theorem}
\noindent
We postpone the proof of this result to Sec.\ \ref{ss:3.2}, and right away point
out a remarkable consequence of formula \eqref{chain-rule-formula}. Namely, that
the variational inequality in
\eqref{weak-def-sol} is equivalent to the energy inequality associated with \eqref{visc-eps-dne}.
The latter inequality involves the
Fenchel-Moreau convex conjugate $\calR_\epsi^*$ taken in the $\calZ-\calZ^*$
duality, and defined by
\[
\calR_\epsilon^* (\xi)
  =\sup \left\{ \pairing{\calZ}{\xi}{w} - \calR_\epsilon(w)\, : \ w \in \calZ\right\}.
\]
 In  \eqref{repre-conjugate} later on we \bnnc give \ennc the explicit formula  for $\calR_\epsilon^*$.
\begin{proposition}
\label{prop:equivalence}
Under Assumptions  \ref{assumption:energy}, \ref{ass:init}, and
 (A$_\Omega$1),
a  curve $z$ fulfilling \eqref{reg-z-below} and \eqref{mixed-estimate} is a weak solution of
\eqref{visc-eps-dne} in the sense of Definition \ref{def:wsol} if and only if it fulfills for all $0\leq s
\leq t \leq T$ the energy inequality
\begin{equation}
\label{energy-inequality}
\int_s^t \calR_\epsi (z'(r)) \,\dd r + \int_s^t \calR_\epsi^* (-\rmD_z \calI (r,z(r))) \,\dd r + \calI (t,z(t)) \leq  \calI (s,z(s))+ \int_s^t \partial_t \calI (r,z(r)) \,\dd r.
\end{equation}
\end{proposition}
\begin{proof}
Taking into account that $w \in \calZ$ is arbitrary,
  \eqref{weak-def-sol} rephrases as
\[
\begin{aligned}
 \calR_\epsi (z'(t))
 & + \sup_{w\in \calZ} \left(
  -\pairing{\calZ}{A_{\il} z(t) }{w}  - \pairing{\calZ} { \rmD_z \wt{\calI}(t,z(t))}{ w }  - \calR_\epsi(w)  \right)
\\ &   +  \int_\Omega (1+ |\nabla z(t)|^2)^{\frac{q-2}{2}} \nabla z(t) \cdot \nabla z'(t)  \,\dd x + \int_\Omega \rmD_z \wt{\calI}(t,z(t))  z'(t) \,\dd x    \leq 0 \qquad \foraa\, t \in (0,T).
\end{aligned}
\]
Then, in view of  the definition of $\calR_\epsi^*$ and the chain rule formula \eqref{chain-rule-formula},  the above inequality is equivalent to
\[
 \calR_\epsi (z'(t)) + \calR_\epsi^* (-\rmD_z \calI(t,z(t))) + \frac{\dd}{\dd t }\calI(t,z(t)) \leq \partial_t \calI(t,z(t)) \qquad
  \foraa\, t \in (0,T),
 \]
i.e.\ \eqref{energy-inequality} upon integrating in time.
\end{proof}
\begin{remark}[Failure of energy identity]
\label{r:failure-en-id}
 It remains an open problem to improve \eqref{energy-inequality}
to an energy identity. 
 This would result from the following chain of
inequalities
\begin{equation*}
\begin{array}{ll}
&\int_s^t \calR_\epsi (z'(r))  \,\dd r + \int_s^t \calR_\epsi^* (-\rmD_z \calI (r,z(r)))  \,\dd r
\\ &  \leq  \calI (s,z(s)) - \calI (t,z(t)) + \int_s^t \partial_t \calI (r,z(r))  \,\dd r
 \\
 & \stackrel{(1)}{=} - \int_s^t \int_\Omega (1+ |\nabla z(r)|^2)^{\frac{q-2}{2}} \nabla z(r) \cdot \nabla z'(r) \,\dd x \,\dd r
- \int_s^t \int_\Omega  \rmD_z \wt \calI(r,z(r)) z'(r) \,\dd x \, \dd r
\\
&
\stackrel{(2, ?)}{=} -\int_s^t \langle \rmD_z \calI (r,z(r)),   z'(r) \rangle_{\calZ}  \,\dd r
\\
&
\stackrel{(3)}{\leq}  \int_s^t \calR_\epsi (z'(r)) \,\dd r + \int_s^t \calR_\epsi^* (-\rmD_z \calI (r,z(r)))  \,\dd r \,.
\end{array}
\end{equation*}
Now,  while $(1)$ follows from   an  integrated version of \eqref{chain-rule-formula} on the right-hand side of \eqref{energy-inequality}
 and $(3)$ from an elementary convex analysis inequality, \bnnc $(2, ?)$ (which is open, at the moment) \ennc
implies the information that $z'(t)\in \calZ$ for almost all $t \in (0,T)$, which we do not dispose of. 
Observe that, with this argument we would also conclude that $z$ fulfills the subdifferential inclusion \eqref{visc-eps-dne},
cf.\  the proofs of \cite[Thm.\ 4.4]{mrs2013}, \cite[Thm.\
3.1]{krz}.
Therefore, the validity of \eqref{visc-eps-dne}  and of the related
energy identity is at the moment open  for general
$q>d$. 
\end{remark}

\noindent
We are now in the position of stating our existence result for the Cauchy
problem associated with  \eqref{visc-eps-dne}.
In fact, we need to impose a further, natural condition on the domain $\Omega$. This
is exploited in the proof of fine spatial regularity estimates on the discrete solutions, which lead to the enhanced
regularity \eqref{enhanced-reg} below for $z$, and will enable us to perform the passage to  the limit in the time-discretization scheme of~\eqref{visc-eps-dne}.
\begin{theorem}[Existence of weak solutions, $\epsi>0$]
\label{thm:ex-viscous}
Under Assumptions \ref{assumption:energy}, \ref{ass:init}, and
 (A$_\Omega$1),
 suppose in addition that
 \begin{itemize}
 \item[(A$_\Omega$2)] $\Omega\subset \R^d$ is a bounded domain and satisfies
the uniform cone condition.
\end{itemize}
Suppose that the initial datum $z_0 \in \calZ$ also fulfills
\begin{equation}
\label{enhanced-initial-datum}
\rmD_z \calI(0,z_0) \in L^2 (\Omega).
\end{equation}
 Then,
   \begin{enumerate}
\item
 for every
$\epsilon>0$ there exists a weak solution (in the sense of Definition \ref{def:wsol})  $z_\epsilon \in
L^\infty (0,T;W^{1,\il}(\Omega))\cap W^{1,2} (0,T; W^{1,2}(\Omega)) $ to the Cauchy problem
\eqref{visc-eps-dne}--\eqref{Cauchy-condition},
fulfilling \eqref{mixed-estimate}
as well as the enhanced regularity
\begin{equation}
\label{enhanced-reg}
z_\epsilon \in L^{2\il} (0,T; W^{1+\beta, \il}(\Omega)) \qquad
\text{for every } \beta \in \left[0,  \frac{1}{q}\Big(
  1-\frac{d}q\Big)  \right). 
\end{equation}
 If in addition $f$ and $g$ comply with
\eqref{def-gg} (cf.\ Proposition~\ref{rem:nice} ahead)
 and if $z_0\in [0,1]$,
 then  $z_\epsilon(t,x)\in[0,1]$ for all $(t,x) \in [0,T] \times \Omega$.
\item
There exists a family of viscous solutions $(z_\epsilon)_{\epsi>0}$
and a constant $C_0>0$ such that the following estimates hold \emph{uniformly} w.r.t.\ $\epsi$
\begin{align}
\label{crucial-elle1-esti} &  \sup_{\epsi>0}
\| z_\epsi \|_{W^{1,1} (0,T; L^2 (\Omega))}
 \leq C_0,
\\
&
\label{crucial-elle2-esti}
 \sup_{\epsi>0}
 \| z_\epsi \|_{L^{2 \il } (0,T; W^{1+\beta, \il}(\Omega)) \cap L^\infty (0,T; W^{1, \il}(\Omega)) } \leq C_0
 \qquad \text{for every } \beta \in \left[0, \frac{1}{q}\Big(
   1-\frac{d}q \Big) \right), 
 \\
 &
 \label{crucial-elle3-esti}
 \sup_{\epsi> 0} \int_0^T \| z_\epsi(t) \|_{W^{1+\beta,
     \il}(\Omega)}^{\il}
 \| z_\epsi'(t) \|_{L^{2}(\Omega)} \,\dd t \leq C_0 \qquad \text{for
   every } \beta \in \left[0,  \frac{1}{q}\Big( 1-\frac{d}q \Big) \right),
 \\
 &
 \label{crucial-elle3-mixed}
  \sup_{\epsi>0}
  \int_0^T \left(  \int_\Omega (1+ |\nabla
    z_\epsi(t)|^2)^{\frac{\il-2}{2}} |\nabla z_\epsi' (t)|^2 \,\dd x
  \right)^{\frac12} \,\dd t \leq C_0.
\end{align}
\end{enumerate}
\end{theorem}

The \emph{proof},  which is given in Section \ref{s:6}, relies on the time-discretization
 analysis performed in Section~\ref{s:4} and on the a priori
estimates provided in Section~\ref{s:5}. Indeed we prove via time discretization
that there exists a function $z_\epsilon$ as in \eqref{reg-z-below},  satisfying
the energy inequality \eqref{energy-inequality}. Moreover, for
$z_\epsilon$ the \emph{mixed estimate} \eqref{mixed-estimate} holds. Therefore, thanks
to Proposition~\ref{prop:equivalence}, $z_\epsilon$ turns out to be a weak
solution to the Cauchy problem \eqref{visc-eps-dne}--\eqref{Cauchy-condition}.

The uniform w.r.t.\ $\epsi$ estimates
\eqref{crucial-elle1-esti}--\eqref{crucial-elle3-mixed}
are the starting point for the vanishing viscosity analysis which we develop in
Section~\ref{s:7}. We shall prove them
in Section~\ref{s:5} arguing on
the time-discretization of \eqref{visc-eps-dne} and thus deduce them only for the viscous solutions $z_\epsi$ to
\eqref{visc-eps-dne}, which arise in the limit of the time-discretization scheme
of Sec.~\ref{s:4}. 

The additional condition
 \eqref{enhanced-initial-datum} on  the initial datum $z_0$
 is needed
  in order to prove
the enhanced regularity estimates for $z$, as well as the uniform discrete
$W^{1,1}$-estimate (see Sections \ref{ss:5.4} and \ref{ss:5.5.}).

\paragraph{A discussion on the interpretation of weak solutions.}  
For $\xi\in W^{1,q}(\Omega)$ let
\begin{align}
\label{def_weighted-norm-1}
\norm{\xi}_{\nabla z(t)}:=\left( \norm{\xi}_{L^2(\Omega)}^2 +
\int_\Omega (1+\abs{\nabla z(t)}^2)^{\frac{q-2}{2}}\abs{\nabla
\xi}^2\dx\right)^\frac{1}{2}
\end{align}
and  define $\calV_{\nabla z(t)}(\Omega):=
\overline{\calZ}^{\norm{\cdot}_{\nabla z(t)}}$. Observe that the set
$\calZ_-:=\Set{z\in \calZ}{z\leq 0}$ is dense in
$\calV_{\nabla z(t),-}:=\Set{v\in \calZ_{\nabla z(t)}}{v\leq
  0}$.
 This implies that the conjugate functional of $\calR_\epsilon$
 calculated with respect to the $\calZ-\calZ^*$ duality (which in this
 context we denote by $\calR_\epsilon^{*_\calZ}$), and the conjugate
 functional $\calR_\epsilon^{*_{\calV_{\nabla
      z(t)}}}$ with respect to the $\calV_{\nabla z(t)}-\calV_{\nabla
   z(t)}^*$ duality, coincide on $\calV_{\nabla
   z(t)}^*$.
Now, let
 $z\in L^\infty(0,T;W^{1,q}(\Omega))\cap W^{1,2}(0,T;W^{1,2}(\Omega))$
 be a weak solution
to the Cauchy problem
\eqref{visc-eps-dne}--\eqref{Cauchy-condition} in the sense of
Definition \ref{def:wsol}, with the enhanced
 regularity
\eqref{enhanced-reg}, and assume in addition that $z'(t)\in
\calV_{\nabla z(t)}$ for almost all $t\in (0,T)$. As it  will be discussed below
this is not a trivial assumption, and at the moment it is open whether
the solution $z$ satisfies this assumption at all.

Now
 we can verify
directly relying on Section \ref{ss:2.3} that $\rmD_z\calI(t,z(t))\in
\calV_{\nabla z(t)}^*$. Having this, 
with the additional assumption 
 that $z'(t)\in
\calV_{\nabla z(t)}$ for almost all $t\in (0,T)$,
from the local version of \eqref{energy-inequality}
 in combination with the chain rule \eqref{chain-rule-formula}
 and the
Young-Fenchel inequality for conjugate functionals we deduce that for
almost all $t \in (0,T)$ it holds
\begin{align*}
\calR_\epsilon(z'(t)) & +
\calR_\epsilon^{*_\calZ}(-\rmD_z\calI(t,z(t)))
 \leq \langle
-\rmD_z\calI(t,z(t)), z'(t)\rangle_{\calV_{\nabla z(t)}}\\
&\leq \calR_\epsilon(z'(t))  + \calR_\epsilon^{*_{\calV_{\nabla
      z(t)}}} (-\rmD_z\calI(t,z(t)))
= \calR_\epsilon(z'(t))  + \calR_\epsilon^{*_\calZ}(-\rmD_z\calI(t,z(t))).
\end{align*}
Hence, for almost all $t \in (0,T)$ the inclusion
\begin{align*}
0&\in \partial\calR_\epsilon(z'(t)) + \rmD_z\calI(t,z(t)) \quad \foraa\, t \in (0,T).
\end{align*}
is satisfied  in the $\calV_{\nabla z(t)}-\calV_{\nabla
  z(t)}^*$ duality and
 in \eqref{energy-inequality} we have equality instead of an
inequality.

However, proving that 
 $z'(t)\in\calV_{\nabla z(t)}$ is at the moment an open problem: Due to the mixed estimate, for almost all $t$ the function
$z'(t)$ belongs to the Banach space $\calW_{\nabla z(t)}:=\Set{v\in
  H^1(\Omega)}{\norm{v}_{\calV_{\nabla z(t)}}<\infty}$.
If the weight $\omega(t):=(1+\abs{\nabla
  z(t)}^2)^{\frac{q-2}{2}}$ can be shown to be a Muckenhoupt weight,
then the spaces $\calV_{\nabla z(t)}$ and $\calW_{\nabla z(t)}$
coincide, see for instance  \cite{ChiadoPiat-SerraCassano}. 
 However, we do not see how to deduce this
property for our solution. Another possibility would be to prove directly
from the construction of the solutions (via a time-incremental
procedure), that $z(t)\in \calV_{\nabla z(t)}$. But also this is not
clear for us.

\subsection{Proof of the chain rule of Theorem \ref{thm:chain-rule}}
\label{ss:3.2}
 Recalling the decomposition $\calI(t,z) = \calI_q (z) +
\wt\calI(t,z)$, we separately address the chain rule
properties of the functionals $\calI_q$ and $\wt \calI$.

As for the latter, we observe that the Fr\'echet differentiability stated in Corollary \ref{coro-fre} allows us to conclude the validity of the chain
rule formula \eqref{ch-rule-tilde}, only if the curve $z$ is in $W^{1,1}([0,T];\calZ)$, which is not granted by \eqref{reg-z-below} and
\eqref{mixed-estimate}. In the proof of Proposition \ref{l:ch-rule-tilde} below, we in fact exploit the finer estimates on $\wt \calI$ and $\rmD_z \wt\calI$
provided by Lemma \ref{l:diff_contz}, and combine them with the regularity \eqref{reg-z-below}  for $z$.  Note that the \emph{mixed estimate} \eqref{mixed-estimate}
is not needed.
\begin{proposition}
\label{l:ch-rule-tilde}
Under Assumptions \ref{assumption:energy}, \ref{ass:init}, and
 (A$_\Omega$1), for every curve $z$ fulfilling \eqref{reg-z-below}
 the map $t \mapsto \wt\calI (t,z(t)) $ is absolutely continuous on $(0,T)$ and
there holds (cf.\ Notation \ref{not:abuse})
\begin{equation}
\label{ch-rule-tilde}
\frac{\dd }{\dd t} \wt\calI (t,z(t)) - \partial_t \wt\calI (t,z(t))
  =  
 \int_\Omega \rmD_z \wt\calI(t,z(t))   z'(t) \,\dd x  \qquad \foraa\, t \in
(0,T).
\end{equation}
\end{proposition}
\begin{proof}
For any fixed $z \in L^\infty (0,T;W^{1,\il}(\Omega))\cap W^{1,2} (0,T;
W^{1,2}(\Omega)) $, the map $t \mapsto \wt \calI (t,z(t))$ is absolutely
continuous on $[0,T]$: indeed,
it follows from \eqref{Lip-cont-I} that
\[
\left| \wt\calI(t,z(t)) - \wt\calI(s,z(s)) \right| \leq C (|t-s| + \| z(t) -
z(s)\|_{L^{2p_*/(p_*-2)}}) \leq  C (|t-s| + \| z(t) - z(s)\|_{W^{1,2}(\Omega)})
\]
where the last inequality follows from  the continuous embedding
$W^{1,2}(\Omega)\subset L^{2p_*/(p_*-2)}(\Omega) $, cf.\ \eqref{lions-magenes}.
We now prove the chain rule formula \eqref{ch-rule-tilde}.
The integral
on the right-hand side of \eqref{ch-rule-tilde} is well defined
since
 $\rmD_z\wt\calI(t,z(t))\in
L^\mu(\Omega)$ and  $z'(t)\in L^{\mu'}(\Omega)$
 with $\mu= 2 p_*/(p_*+2)$ (and $\mu'=2p_*/(p_*-2)$), cf.\  Lemma~\ref{l:diff_contz}).
 In fact, since
 $
L^\mu(\Omega) \subset W^{1,2}(\Omega)^*$,
it follows that $
\rmD_z\wt\calI(t,z(t)) $ can
 be identified with an element in $W^{1,2}(\Omega)^*$ for a.a.\,$t \in (0,T)$.
We fix $t \in (0,T)$, out of a negligible set, such that $\exists\, \frac{\dd
}{\dd t} \wt\calI (t,z(t)) $, and compute
\[
\begin{aligned}
h^{-1}(\wt\calI(t & +h, z(t+h))  -\wt\calI(t, z(t)))
\\
&  =
h^{-1}(
\wt\calI(t+h,
z(t+h)) -\wt\calI(t, z(t+h))
)
 +
h^{-1}
(\wt\calI(t, z(t+h)) -\wt\calI(t,
z(t))
)
\\& = \frac1h \int_t^{t+h} \partial_t \wt\calI (s,z(t+h)) \,\dd s
\\ & \quad + \frac1h  \int_\Omega \int_0^1 \rmD_z \wt\calI (t,(1-\theta) z(t)
+\theta z(t+h)) (z(t+h) - z(t)) \,\dd\theta \,\dd x
\doteq I_h^1 +I_h^2.
\end{aligned}
\]
We have that
\[
I_h^1 =  \frac1h \int_t^{t+h} \partial_t \wt\calI (s,z(t))\,\dd s  +  \frac1h
\int_t^{t+h}\left(  \partial_t \wt\calI (s,z(t+h)) -\partial_t \wt\calI (s,z(t))
\right)   \,\dd s.
\]
The first term on the right-hand side converges to $\partial_t \wt\calI
(s,z(t))$ as $h \to 0$, while the second one tends to zero in view of
\eqref{stim5}
and of the fact that $z \in \rmC^0 ([0,T]; \rmC^0(\overline{\Omega}))$ by interpolation in \eqref{reg-z-below}. In order to take the
limit as $h \to 0$ of $I_h^2$, we first of all observe that
for almost all $t \in (0,T)$ $\frac{z(t+h) - z(t)}h \to z'(t)$ in $W^{1,2}
(\Omega) \subset L^{2p_*/(p_*-2)}(\Omega)$  as $h \to 0$, due to the fact that
$z' \in L^2 (0,T; W^{1,2} (\Omega))$ (cf., e.g.,  \cite[Prop.\ A.6]{Brez73OMMS}).
 Moreover,
in view of \eqref{enhanced-stim-7},  the family
$j_h (t,\cdot) := \int_0^1 \rmD_z \wt\calI (t,(1-\theta) z(t,\cdot) +\theta
z(t+h,\cdot)) \,\dd \theta $ converges to $j(t,\cdot):= \rmD_z \wt\calI (t,
z(t,\cdot)) $
in $ L^{2p_*/(p_*+2)}(\Omega)$ as $h \to 0$. Hence,
 $I_h^2 \to  \int_\Omega   \rmD_z \wt\calI(t,z(t)) z'(t) \,\dd x $ as
$h \to 0$, and \eqref{ch-rule-tilde} follows.
\end{proof}
For the functional $\calI_q$, we have the following result.
\begin{proposition}
\label{l:ch-rule-iq}
For every curve $z$ fulfilling \eqref{reg-z-below} and \eqref{mixed-estimate}
 the map $t \mapsto \calI_q (z(t)) $ is absolutely continuous on $(0,T)$ and
there holds
\begin{equation}
\label{ch-rule-iq}
\frac{\dd }{\dd t} \calI_q (z(t)) =  \int_\Omega (1+ |\nabla z(t)|^2)^{\frac{q-2}{2}} \nabla z(t) \cdot \nabla z'(t) \,\dd x \qquad \foraa\, t \in (0,T).
\end{equation}
\end{proposition}
We are going to deduce Proposition~\ref{l:ch-rule-iq} from applying the result below to $F:= \nabla z$.
\begin{lemma}
\label{th:ch-rule-bis}
 Define $\calG_q :  L^q (\Omega; \R^d )\to [0,\infty)$ by
$
\calG_q (F):=
\int_\Omega G_q (F(x))  \,\dd x =
 \frac1q \int_\Omega (1+|F(x)|^2)^{\frac q2} \,\dd x$.
If $F \in L^\infty (0,T;L^{q}(\Omega;\R^d)) \cap W^{1,1} (0,T;L^2(\Omega;\R^d))$
fulfills
\begin{equation}
\label{mixed-estimate-F}
\int_0^T \int_\Omega (1+|F|^2)^{\frac{q-2}{2}} |F_t|^2 \,\dd x \,\dd t <\infty,
\end{equation}
then the map $t \mapsto \calG_q(F(t)) $ is absolutely continuous on $(0,T)$, and there holds \begin{equation}
\label{ch-ruke-F}
\frac{\dd }{\dd t} \calG_q (F(t)) =
 \int_\Omega (1+|F(t)|^2)^{\frac{q-2}{2}} F(t) \cdot F_t (t) \,\dd x \qquad \foraa\, t \in (0,T)\,.
\end{equation}
\end{lemma}
\begin{proof}
We split the proof in three claims.

\paragraph{\bf Claim $1$:} \emph{There holds for all $0 \leq s \leq t\leq T $ and for almost all $x\in \Omega$
\begin{equation}
\label{formula-x}
\begin{aligned}
 G_q (F(t,x)) - G_q (F(s,x)) = \int_s^t
 (1+|F(r,x)|^2)^{\frac{\il-2}{2}} F(r,x) \cdot F_t (r,x) \,\dd r \,.
 \end{aligned}
\end{equation}
}
Indeed, \eqref{formula-x} follows from integrating in time the chain rule at fixed $x$
\begin{equation}
\label{fixed-x-chain-rule}
\frac{\dd}{\dd t} G_q (F(t,x)) =(1+|F(r,x)|^2)^{\frac{\il-2}{2}} F(r,x) \cdot F_t (r,x)
\end{equation}
which in turn ensues from applying
\eqref{classical-ch-rule} below with $\eta(t)=F(t,x)$ (here $x$ is fixed outside a negligible set) and $\varphi= G_q$.
Indeed, \eqref{mixed-estimate-F} guarantees
\begin{equation}
\label{guarantees}
\begin{aligned}
& (t,x)\mapsto (1+|F(t,x)|^2)^{\frac{\il-2}{4}} F_t(t,x) \, \in L^2 (0,T;L^2(\Omega;\R^d)),
\\
&
(t,x)\mapsto (1+|F(t,x)|^2)^{\frac{\il-2}{4}} F(t,x) \, \in L^2 (0,T;L^2(\Omega;\R^d)).
\end{aligned}
\end{equation}
Hence, by the properties of Bochner integrals we have for almost all $x \in \Omega$ that
$ t \mapsto
(1+|F(t,x)|^2)^{\frac{\il-2}{4}} F_t(t,x) \in L^2 (0,T) $ and
$t \mapsto (1+|F(t,x)|^2)^{\frac{\il-2}{4}} F(t,x) \in L^2 (0,T) $, therefore
\[
\foraa\, x \in \Omega \quad t \mapsto  | (1+|F(t,x)|^2)^{\frac{\il-2}{2}}  F(t,x) |  | F_t(t,x) | \in L^1 (0,T),
\]
and we  can apply Lemma~\ref{ch-rule-added}.

\paragraph{\bf Claim $2$:}  \emph{the map
 $t
\mapsto \calG_q (F(t))$ is absolutely continuous on $[0,T]$.}
\\
Indeed,
integrating w.r.t.\ $x\in \Omega$ formula \eqref{formula-x} we find
\begin{equation}
\label{quoted-later}
\begin{aligned}
\calG_q (F(t)) - \calG_q (F(s))  = \int_s^t
\int_\Omega
 (1+|F(r,x)|^2)^{\frac{\il-2}{2}} F(r,x) \cdot F_t (r,x) \,\dd r \,\dd x \ \
 \text{ for all } 0 \leq s \leq t \leq T.
 \end{aligned}
 \end{equation}
 We use this to estimate the difference
  $|\calG_q(F(t))-\calG_q(F(s))| $. In view of \eqref{guarantees}
   (and the properties of
 Bochner integrals), the map
 \[
  t \mapsto \int_\Omega (1+|F(t,x)|^2)^{\frac{\il-2}{2}} F(t,x) \cdot F_t(t,x) \,\dd x \in L^1 (0,T),
 \]
 therefore the absolute continuity of $t
\mapsto \calG_q (F(t))$ follows from \eqref{quoted-later} and
 the absolute continuity property of the Lebesgue integral.

\noindent
\paragraph{\bf Claim $3$:} \emph{formula \eqref{ch-ruke-F} holds.}
Let us fix $t$ outside a negligible set such that
 $\frac{\dd }{\dd t} \calI(F(t))$ exists as limit of the difference quotient.
 Writing \eqref{quoted-later} at $t$ and $t+h$ yields
 \[
 \begin{aligned}
 &
\frac1h \left( \calG_q (F(t+h)) - \calG_q (F(t) ) \right)  = \frac1h \int_t^{t+h}
\int_\Omega
 (1+|F(r,x)|^2)^{\frac{\il-2}{2}} F(r,x) \cdot F_t (r,x) \,\dd r \,\dd x.
 \end{aligned}
 \]
 Then, it remains to observe that as $h \to 0$
 \[
  \frac1h \int_t^{t+h}
\int_\Omega
 (1+|F(r,x)|^2)^{\frac{\il-2}{2}} F(r,x) \cdot F_t (r,x) \,\dd r \,\dd x \to
\int_\Omega
 (1+|F(t,x)|^2)^{\frac{\il-2}{2}} F(t,x) \cdot F_t (t,x)  \,\dd x.
 \]
 This is true for almost all $t \in (0,T)$ thanks to the Lebesgue point property of the
 map $ t \mapsto \int_\Omega (1+|F(t,x)|^2)^{\frac{\il-2}{2}} F(t,x) \cdot F_t(t,x) \,\dd x \in L^1 (0,T)$.
\end{proof}

\noindent We conclude by stating, for the sake of completeness, the following auxiliary result.
\begin{lemma}
\label{ch-rule-added}
Given $\varphi \in \rmC^1 (\R^d;\R) $, for every $\eta \in W^{1,2}(0,T;\R^d)$ such that $t \mapsto |\nabla \varphi (\eta(t))| | \eta'(t)| \in L^1 (0,T) $ there holds
\begin{equation}
\label{classical-ch-rule}
\begin{gathered}
\text{the map } t \mapsto \varphi(\eta(t)) \text{ is absolutely continuous on $(0,T)$ and}
\\
\frac{\dd}{\dd t} \varphi(\eta(t))= \nabla \varphi(\eta(t)) \cdot \eta'(t)\quad \foraa\, t \in (0,T).
\end{gathered}
\end{equation}
\end{lemma}
\begin{proof}
The absolute continuity property can be shown by arguing in the very same way as in the proof of \cite[Thm.\ 1.2.5, page 28]{AGS2008}. The chain rule formula follows
from classical arguments.
\end{proof}

\section{Time-discretization for the viscous problem}
\label{s:4} We consider the following time-discrete incremental
minimization problem:
 Given $\epsilon>0$,
$z_0\in \calZ$ 
and a  uniform  partition $\{0=t_0^\tau<\ldots<t_N^\tau=T\}$ of the time
interval $[0,T]$ with fineness $\tau = 
t^\tau_{k+1} - t^\tau_k =T/N$   (cf.\ Remark \ref{rmk:fixed-time-step} ahead),
 the elements  $(z^\tau_{k})_{0\leq
k\leq
N}$ are determined through $z_0^\tau=z_0$ and
\begin{align}
\label{def_time_incr_min_problem_eps}
z_{k+1}^\tau \in
\Argmin\Set{\calI(t_{k+1}^\tau,z) +
\tau\calR_\epsilon\left(\frac{z
    -z_k^\tau}{\tau}\right)}{z\in\calZ}.
\end{align}
The existence of minimizers can be checked via the direct method in the calculus of variations, thanks
to the properties of the reduced energy $\calI$ formulated in
Section~\ref{sec:red_energy}. It follows from the
representation formula \eqref{aforementioned-form} for $\partial \calR_\epsi$
that, any family
 $\{z_1^\tau,\ldots,z_N^\tau\}\subset\calZ$
of minimizers of the incremental problem
\eqref{def_time_incr_min_problem_eps} satisfies for all
$k\in\{0,\ldots, N-1\}$ the \emph{discrete Euler-Lagrange} equation
\begin{align}
\label{pruni1:e1}  \partial\calR_1\left(\frac{z_{k+1}^\tau -
z_k^\tau}{\tau}\right) + \epsilon  \frac{z_{k+1}^\tau -
z_k^\tau}{\tau} + \rmD_z\calI(t_{k+1}^\tau,z_{k+1}^\tau) \ni 0 \qquad
\text{in } \calZ^*.
\end{align}
\begin{proposition}
\label{rem:nice} Under Assumptions \ref{assumption:energy}, \ref{ass:init}, and
 (A$_\Omega$1), for $\tau$ sufficiently small
 the minimum problem
\eqref{def_time_incr_min_problem_eps} admits a unique solution.

Suppose in addition that $f$ and $g$ comply with the following condition
\begin{equation}
\label{def-gg}
  f(0) \leq f(z), \quad  g(0) \leq g(z) \quad  \text{  for all } z \leq
 0,
\end{equation}
 and that  the initial datum $z_0$ fulfills
$ z_0(  x)\in [0,1]$  for almost all $x\in \Omega$.
   Then,
the minimizer $z_k^\tau$ from \eqref{def_time_incr_min_problem_eps}
   also fulfills
  $z_k^\tau(x) \in  [0,1]$ for almost all $x\in \Omega$.
\end{proposition}
\noindent The proof of uniqueness
 follows from standard arguments, exploiting  estimate
\eqref{e:strong-monot-sez2} from
Corollary \ref{cor:rosenheim}.
For the property $z_k^\tau(x)\in [0,1]$, we refer the reader to the proof of
\cite[Prop.\ 4.5]{krz}.
\begin{notation}
\label{not-interp} \upshape
 The following piecewise constant and piecewise
linear interpolation functions will be used in the sequel:
\begin{alignat*}{2}
\overline z_\tau (t) &= z_{k+1}^\tau&\quad&\text{for }t\in
(t_k^\tau,t_{k+1}^\tau],\\
\underline z_\tau(t)&=  z_k^\tau &&\text{for }t\in
[t_k^\tau,t_{k+1}^\tau),\\
\hat z_\tau(t)&=z_k^\tau + \frac{t - t_{k}^\tau}{\tau}
(z_{k+1}^\tau - z_k^\tau)&&\text{for }t\in [t_k^\tau,t_{k+1}^\tau].
\end{alignat*}
 Furthermore, we shall use the
notation
\begin{alignat*}{2}
 \tau(r) &= \tau &\quad&\text{for } r\in (t_k^\tau, t_{k+1}^\tau),
\\
\overline{t}_\tau(r)&= t_{k+1}^\tau &\quad&\text{for } r\in (t_k^\tau,
t_{k+1}^\tau],
\\
\underline{t}_\tau(r)&= t_{k}^\tau &\quad&\text{for } r\in [t_k^\tau,
t_{k+1}^\tau),
\end{alignat*}
as well as 
\begin{alignat*}{2}
\overline u_\tau (r) & = \umin (\overline{t}_\tau(r), \overline
z_\tau (r)) &  \quad&\text{for }r\in
(t_k^\tau,t_{k+1}^\tau],\\
\underline u_\tau (r) & = \umin (\underline{t}_\tau(r), \underline
z_\tau (r)) &&\text{for }r\in [t_k^\tau,t_{k+1}^\tau),
\\
\hat u_\tau(r)&= \underline u_\tau (r) + \frac{r -
\underline{t}_\tau(r)}{\tau} (\overline u_\tau (r) - \underline u_\tau
(r)) &&\text{for }r\in [t_k^\tau,t_{k+1}^\tau].
\end{alignat*}

 Clearly,
\begin{equation}
\label{quoted -later}
 \overline{t}_\tau(t), \, \underline{t}_\tau(t) \to t
 \quad \text{ as
$\tau \to 0$ for all $t \in [0,T]$.}
\end{equation}
 Moreover,
for any given function  $b$  which is piecewise constant on the
intervals $(t_i^\tau,t_{i+1}^\tau)$ we set
\[
 \triangle_{\tau(r)} b(r) = b(r) - b(s) \text{ for $r\in (t_k^\tau,
t_{k+1}^\tau)$ and $s\in (t_{k-1}^\tau, t_k^\tau)$.}
\]
\end{notation}

 With the above notation, \eqref{pruni1:e1} can be
reformulated as
\begin{equation}
\label{cont-reformulation}
 \partial\calR_1\left(\hat{z}_\tau'(t)\right) +  \epsilon  \hat{z}_\tau'(t)   + \rmD_z\calI(\overline{t}_\tau(t),
 \overline{z}_\tau(t)) \ni 0 \quad \foraa\, t \in (0,T),
\end{equation}
viz.\
\begin{equation}
\label{cont-reformulation-bis}
\begin{cases}
  \overline{\omega}_\tau (t)+  \epsilon
  \hat{z}_\tau'(t)
  + A_\il \overline{z}_\tau(t) +
   \rmD_z\wt \calI(\overline{t}_\tau(t),
 \overline{z}_\tau(t)) = 0,
 \\
 \overline{\omega}_\tau (t) \in
 \partial\calR_1\left(\hat{z}_\tau'(t)\right)
 \end{cases}
  \quad \foraa\, t \in (0,T).
\end{equation}
\begin{notation}
\upshape In what follows, we will  denote most of the positive constants
occurring in the calculations  by the symbols $c,\, C'$,
whose meaning may vary even within the same line.
Furthermore, the symbols $I_i,$
$S_i$, $F_i$,  $i=0,1,\ldots,$
 will be used as place-holders for  several integral terms popping in the various
 estimates: we warn the reader that we will not be self-consistent with the numbering, so that, for instance,
$I_1$ will appear several times with different meanings.
\end{notation}
 \subsection{Global higher differentiability of the time-discrete damage variable}
\label{ss:4.2} In this section we study the higher
differentiability of the solutions $z_k^\tau$ of the
time-incremental mi\-ni\-mi\-za\-tion problem
\eqref{def_time_incr_min_problem_eps}, which leads to Theorem\
\ref{app_reg_thm_z} below.
 Its proof relies on a
difference quotient argument in the spirit of
\cite{savare97,ell:EF99,kne:Kne05}.
Note that it is in the proof of Theorem~\ref{app_reg_thm_z}
that we need to resort to the additional condition (A$_\Omega$2) on
$\Omega$ stated in Theorem~\ref{thm:ex-viscous}.
%
We refer to \cite{sin:Gri85} for a precise definition of the
\emph{uniform cone condition}. Observe that this condition is
equivalent to the property that $\partial\Omega$ can locally be
represented as the graph of a Lipschitz continuous mapping.

 We address the higher differentiability of minimizers for
\eqref{def_time_incr_min_problem_eps} in a more general context. In
particular, in view of future developments
 we deal with an $L^\alpha$-viscosity term instead of $L^2$-viscosity.
Therefore,
 let $q>d$, $p>2$. 
  For $z,\zeta\in
W^{1,q}(\Omega)$, $w\in W^{1,p}_{\Gamma_D}(\Omega;\R^d)$, $\tau>0$,
$\epsilon\geq 0$ and $\alpha\geq 2$
 we
define
\begin{align}
 \label{app_def_F}
\calF(z;\alpha,\tau,\epsilon,\zeta,w):= \int_\Omega
\frac12 g(z)\bbC
\varepsilon(w):\varepsilon(w) +  f(z ) +\frac{1}{q}(1+\abs{\nabla
z}^2)^{\frac{q}{2}} \dx+ \calR_1(z-\zeta) + \frac{\epsilon
\tau}{\alpha}\norm{\frac{z-\zeta}{\tau}}^\alpha_{L^\alpha(\Omega)}
\end{align}
with $\calR_1$ as in \eqref{def_R1}.
 Observe that with this definition the time-incremental minimization problem
\eqref{def_time_incr_min_problem_eps} can be rewritten as
\begin{align}
 z_{k+1}^\tau\in
\Argmin\Set{\calF(z;2,\tau,\epsilon,z_k^\tau,u_\text{min}(t_{k+1}^\tau,
z)+u_D(t_{k+1}^\tau))
} { z\ \in \calZ}.
\end{align}
\begin{theorem}[Spatial differentiability of the damage variable]
\label{app_reg_thm_z}
Under Assumptions \ref{assumption:energy},
 (A$_\Omega$1), and (A$_\Omega$2), suppose
further that $w\in W^{1,p}_{\Gamma_D}(\Omega;\R^d)$ for some $p\geq 2$,
and  that $\tau>0$, $\epsilon\geq 0$, $\alpha\geq 2$ and
$q>d$.

Let $z\in \calZ=W^{1,q}(\Omega)$ be a minimizer of
$\calF(\cdot;\alpha,\tau,\epsilon,\zeta,w)$ over $\calZ$.
 Then for all $0\leq \beta<  \frac{1}{q}\left(1-\frac{d}{q}\right)$
 we have $z\in W^{1+\beta,q}(\Omega)$.
Moreover, there exists a constant $c_\beta>0$ such that
 \begin{align}
 \label{quoted-later-ohyes}
\norm{z}_{W^{1 + \beta,q}(\Omega)}
\leq c_\beta (1+\norm{z}_{W^{1,q}(\Omega)})
\left( 1
+ \norm{f'(z)}_{L^\infty(\Omega)}^{\frac{1}{q}}
+
\norm{w}^{\frac{2}{q}}_{W^{1,p}(\Omega;\R^d)} + \epsilon^{\frac{1}{q}}
\norm{\frac{z-\zeta}{\tau}}^{\frac{\alpha-1}{q}}_{L^\alpha(\Omega)}
\right),
\end{align}
and the constant $c_\beta$ is independent of
$\alpha,\epsilon,\tau,z,w$ and $\zeta$.
\end{theorem}
\begin{remark}
 For  $\epsilon=0$, Theorem \ref{app_reg_thm_z} yields a regularity
result for global energetic solutions associated with the energy
$\calI(\cdot,\cdot)$ from \eqref{reduced-energy} and the dissipation
potential $\calR_1$. Indeed,  let $z:[0,T]\rightarrow \calZ$ be a
global energetic solution associated with $\calI$ and $\calR_1$. The
stability condition,
 that is satisfied by  global energetic  solutions, implies that  for all $t \in [0,T]$
the function $z(t)$ minimizes
$\calF(\cdot;2,1,0,z(t),u_\text{min}(t,z(t)))$. Hence, by Theorem
\ref{app_reg_thm_z},
 for all $t\in [0,T]$ it
holds $z(t)\in W^{1+\beta,q}(\Omega)$ with $\sup_{t\in
[0,T]}\norm{z(t)}_{W^{1+\beta,q}(\Omega)} <\infty$. We refer to
\cite{MR06,TM10} for the analysis of damage models in the context of
global energetic solutions.
\end{remark}

\begin{proof}[Proof of Theorem \ref{app_reg_thm_z}]
 As already announced the proof relies on a difference quotient technique.
Since spatially shifted versions of the minimizer $z$ of $\calF$ not
necessarily lie below the function $\zeta$, we also have to shift
the function $z$ in ``vertical'' direction.

Let $\Omega\subset\R^d$ satisfy (A$_\Omega$2). Let $x_0\in
\partial\Omega$ be arbitrary and choose $e\in \R^d$ with $\abs{e}=1$
in such a way that there exist constants $R,h_0>0$ such that for all
$y\in \overline{\Omega}\cap B_R(x_0)$ and all $0< h\leq h_0$ we have
$y+he\in \Omega$. Since $\Omega$ satisfies the uniform cone
condition it is possible to find a basis of $\R^d$ such that every
basis vector has this property.

Let $\varphi\in \rmC_0^\infty(B_R(x_0))$ be a cut-off function with
$0\leq \varphi\leq 1$ and $\varphi\big|_{B_{R/2}(x_0)} \equiv 1$. Further,
let us define the transformation $T_h:\R^d\rightarrow \R^d$ by $T_h(x):=x
+ h \varphi(x) e$. If $h\in [0,h_0]$ is small enough, this mapping
is an isomorphism with $T_h(\Omega)\subset\Omega$ and
it coincides with the identity outside of the
ball $B_R(x_0)$. For $w\in W^{1,p}_{\Gamma_D}(\Omega;\R^d)$ and $\zeta\in
W^{1,q}(\Omega)$  let
\begin{align}
\label{def_min_pr_reg_z}
 z\in \Argmin\Set{\calF(\wt z; \alpha,\tau,\epsilon,\zeta,w)}{\wt z\in \calZ}.
\end{align}
From the definition of $\calR_1$ it follows that $z\leq \zeta$
almost everywhere in $\Omega$. Moreover, since $q>d$, we have $z\in
\rmC^{0,\gamma}(\overline{\Omega})$ with $\gamma = 1 - \frac{d}{q}>0$.
For $h>0$ let  $\delta_h:=
h^\gamma\norm{z}_{\rmC^{0, \gamma}(\overline\Omega)}\leq c
h^\gamma\norm{z}_{W^{1,q}(\Omega)}$.  Observe that
\begin{equation}
\label{diff-quot-later}
\abs{z(x) - z(T_h(x))}\leq \delta_h \qquad
\text{for all $x\in
\Omega$}.
\end{equation}
Hence, the function
\begin{equation*}
 z_h(x):= z(T_h(x)) - \delta_h
\end{equation*}
is an admissible test function for the minimization problem
\eqref{def_min_pr_reg_z} in the sense that $\calR_1(z_h-\zeta)$ is
finite. Indeed,  
\begin{align*}
 \zeta(x) - z_h(x)= \zeta(x) - z(x) + (z(x)- z(T_h(x))+\delta_h(x))\geq 0 \ \text{  for all $x\in
\Omega$}.
\end{align*}
Since $z$ is a minimizer, for all $\wt z\in W^{1,q}(\Omega)$ it
satisfies the variational inequality
\begin{align}
 \label{pr_vi_reg_z}
\calR_1&(\wt z - \zeta) - \calR_1( z-\zeta) \nonumber
\\
&\geq -\langle A_q z,\wt z - z\rangle_{\calZ} -\int_\Omega \big( \frac12
g'(z)\bbC
\varepsilon(w):\varepsilon(w) + f'(z) \big)(\wt z- z)\dx -\epsilon
\int_\Omega \abs{\frac{z-\zeta}{\tau}}^{\alpha-2}
\frac{z-\zeta}{\tau} (\wt z - z)\dx.
\end{align}
With the special choice $\wt z=z_h$, taking into account the definition of $\calR_1$ this variational inequality can
be rewritten as
\[\begin{aligned}
-\int_\Omega (1 + \abs{\nabla z}^2)^{\frac{q-2}{2}}\nabla z \cdot
\nabla (z\circ T_h -z)\dx \nonumber
 & \leq \int_\Omega \rho(z - z_h)\dx + \int_\Omega \big(\tfrac{1}{2}
 g'(z)\bbC
\varepsilon(w):\varepsilon(w) + f'(z) \big)( z_h- z)\dx
\nonumber\\
&\phantom{\leq} + \epsilon \int_\Omega
\abs{\frac{z-\zeta}{\tau}}^{\alpha-2} \frac{z-\zeta}{\tau} (z_h -
z)\dx.
\end{aligned}
\]
Now we apply inequality \eqref{q-convexity}
with $a=\nabla z$ and $b=\nabla (z\circ T_h)$, and setting
$\triangle_h z:= z\circ T_h -z$ we thus obtain  the estimate
(recall that $G_q (A) = \frac1q (1+|A|^2)^{\frac q2}$)
\[\begin{aligned}
 c_q \int_\Omega & (1 + \abs{\nabla z}^2 +
\abs{\nabla{z_h}}^2)^\frac{q-2}{2}\abs{\nabla\triangle_h z}^2\dx
\leq \int_\Omega  G_q
(\nabla  (z\circ T_h) )- G_q(\nabla z)\dx
\\
&\phantom{\leq }
 + \int_\Omega \rho\abs{\triangle_h z -\delta_h}\dx +
\int_\Omega \big( g'(z)\bbC \varepsilon(w):\varepsilon(w) + f'(z)
\big)( z_h- z)\dx
\nonumber\\
&\phantom{\leq} + \epsilon \int_\Omega
\abs{\frac{z-\zeta}{\tau}}^{\alpha-2} \frac{z-\zeta}{\tau} (z_h -
z)\dx
\\
&=: S_1 + S_2+S_3 + S_4.
\end{aligned}
\]
The goal is to show that there exists a $\beta\in (0,1)$ such that
the right-hand side can be estimated by $ch^\beta$. This estimate
then implies that $z\big|_{B_{R/2}(x_0)}$ belongs to the Nikolskii
space $\calN^{1 + \frac{\beta}2,2}(\Omega) \cap  \calN^{1 +
\frac{\beta}q,q} (\Omega)$, which is continuously embedded in $W^{1+
\frac{\beta}2-\delta,2}(\Omega)\cap W^{1 + \frac{\beta}q-\delta,q}(\Omega)$ for all
$\delta>0$.

Since by assumption we have $z\in W^{1,q}(\Omega)$, the term $S_2$
can be estimated as
\[
\begin{aligned}
 S_2&\leq \int_\Omega \rho\abs{\triangle_h z}\dx +\delta_h\abs{\Omega}
\leq c (\abs{h} + \abs{h}^\gamma)  \left( 1 + 
  \norm{z}_{W^{1,q}(\Omega)}\right),
\end{aligned}
\]
and the constant $c$ depends on $\Omega$ and the chosen
cut-off-function, but is independent of $h$ and $z$.

Taking into account \eqref{ass-eff}
 the second part of $S_3$ can be
estimated as follows
\[
\begin{aligned}
 \int_\Omega \abs{f'(z)}\abs{z_h-z}\dx & \leq c
 \norm{f'(z)}_{L^\infty(\Omega)}
\norm{z}_{W^{1,q}(\Omega)}
(\abs{h}^\gamma + \abs{h}).
\end{aligned}
\]
H\"older's inequality applied to the first component of $S_3$ yields
\[
\begin{aligned}
 \int_\Omega  g'(z)\bbC \varepsilon(w):\varepsilon(w)(z_h - z)\dx
\leq c \norm{w}_{W^{1,p}(\Omega;\R^d)}^2(\norm{\triangle_h
z}_{L^\frac{p}{p-2}(\Omega)} + \delta_h),
\end{aligned}
\]
where we have used that $p\geq 2$.  By \eqref{diff-quot-later}, the term in
brackets on the right-hand side is bounded by $c\delta_h\leq \wt c \abs{h}^\gamma
\norm{z}_{W^{1,q}(\Omega)}$. Putting together these estimates we
obtain
\[
\begin{aligned}
 \abs{S_3} &\leq
 c(\abs{h} + \abs{h}^\gamma)
\norm{z}_{W^{1,q}(\Omega)}
( \norm{f'(z)}_{L^\infty(\Omega)} +\norm{w}^2_{W^{1,p}(\Omega;\R^d)})
\end{aligned}
\]
and the constant $c$ is independent of $h,z,w$.

In a similar way we obtain for $S_4$, applying again H\"older's
inequality,
\[
\begin{aligned}
 \abs{S_4}&\leq c \epsilon
\norm{\frac{z-\zeta}{\tau}}^{\alpha-1}_{L^\alpha(\Omega)} \big(
\norm{\triangle_h z}_{L^{\alpha}(\Omega)} + \delta_h \big)
\leq c \epsilon
\norm{\frac{z-\zeta}{\tau}}^{\alpha-1}_{L^\alpha(\Omega)}
(\abs{h}+\abs{h}^\gamma)\norm{z}_{W^{1,q}(\Omega)}.
\end{aligned}
\]
It remains to estimate $S_1$. Here we use an argument
that relies on a change of variables in the first term
\DDDS (cf.\ \cite{ell:EF99,savare97,Kne04}): \DDDE
 With $y=T_h(x)$
it follows
\[
\begin{aligned}
\int_\Omega G_q(\nabla z (T_h(x)) \nabla T_h(x))\dx =
\int_{T_h(\Omega)}  G_q(\nabla z(y) \nabla
T_h (T_h^{-1}(y)))  \det \nabla_y T_h^{-1}(y) \dy
\end{aligned}
\]
Observe that due to the special choice of the vector $e$ it holds
$T_h(\Omega)\subset\Omega$ for $0\leq h<h_0$.
  Hence, since
 $G_q(\nabla z)\geq 0$
 almost everywhere, we arrive at
\[
\begin{aligned}
S_1&\leq  \int_{T_h(\Omega) }
 G_q(\nabla z (y)
(\nabla
T_h^{-1}(y))^{-1}) \det \nabla T_h^{-1}(y)  \dy
-  \int_{T_h(\Omega)}
 G_q(\nabla z)\dx.
\end{aligned}
\]
Elementary calculations (based on the fact that $\det\nabla T_h^{-1} \sim  (1 -
h\norm{\varphi}_{\rmC^1(\overline{\Omega})})$ and a Taylor expansion of
$G_q$) show that $S_1$ can be further estimated as
\[
\begin{aligned}
 S_1&\leq c \abs{h}(1+\norm{z}^q_{W^{1,q}(\Omega)}).
\end{aligned}
\]
Again, the constant $c$ is independent of $h$ and $z$.
Collecting all estimates we finally arrive at
\[
\begin{aligned}
\int_\Omega & (1 + \abs{\nabla z}^2 +
\abs{\nabla{z_h}}^2)^\frac{q-2}{2}\abs{\nabla\triangle_h z}^2\dx
\nonumber
\\
&\leq
c(\abs{h}+\abs{h}^\gamma)
(1 + \norm{z}^q_{W^{1,q}(\Omega)})
\left(1
+ \norm{f'(z)}_{L^\infty(\Omega)}
+
 \norm{w}^2_{W^{1,p}(\Omega;\R^d)}
+
\epsilon\norm{\frac{z-\zeta}{\tau}}^{\alpha-1}_{L^\alpha(\Omega)}\right).
\end{aligned}
\]
Since $x_0\in \partial\Omega$ was chosen arbitrarily, after covering
$\overline\Omega$ with a finite number of balls $B_{R_{x_0}}(x_0)$
we finally obtain that $z\in \calN^{1 + \frac{\gamma}{q},q}(\Omega)$
with
\[
\begin{aligned}
\norm{z}_{ \calN^{1 + \frac{\gamma}{q},q}(\Omega)} \leq c
(1+\norm{z}_{W^{1,q}(\Omega)})
\left( 1
+ \norm{f'(z)}_{L^\infty(\Omega)}^{\frac{1}{q}}
+
\norm{w}^{\frac{2}{q}}_{W^{1,p}(\Omega;\R^d)} + \epsilon^{\frac{1}{q}}
\norm{\frac{z-\zeta}{\tau}}^{\frac{\alpha-1}{q}}_{L^\alpha(\Omega)}
\right),
\end{aligned}
\]
and the constant $c$ is independent of $\alpha,\epsilon,\tau,z,w$
and $\zeta$.
\end{proof}

\section{A priori estimates}
\label{s:5} This section
is devoted to deriving for the approximate solutions $(\overline
z_\tau, \hat{z}_\tau, \overline{u}_\tau, \hat{u}_\tau )$
constructed from the time-incremental minimization problem
\eqref{def_time_incr_min_problem_eps}
 a number of a priori
estimates, uniform w.r.t.\ $\tau>0$. These will allow us to pass to
the limit in the approximate differential inclusion
\eqref{cont-reformulation} and conclude the existence of  weak viscous
 solutions
to (the Cauchy problem for) \eqref{visc-eps-dne}. In view of the
vanishing viscosity analysis in Sec.\ \ref{s:7}, in the
following we will specify which estimates are, in addition, uniform
w.r.t.\ $\epsilon>0$. However, for notational simplicity we shall omit to indicate the dependence of
the interpolants
$(\overline
z_\tau, \hat{z}_\tau, \overline{u}_\tau,   \hat{u}_\tau )$ on~$\epsilon$.
\subsection{Energy estimate}
\label{ss:5-enest} We start by stating the basic energy estimate
derived from the time-incremental minimization
\eqref{def_time_incr_min_problem_eps}. It holds uniformly with respect to $\tau$ and $\epsilon>0$.
\begin{lemma}
\label{l:basicenest} Under Assumptions \ref{assumption:energy}, \eqref{ass:init}, and (A$_\Omega$1),
for every $z_0 \in \calZ$ there exists a constant
$C_1>0$ such that for all $\tau>0$ and $\epsi>0$ there holds
\begin{align}
& \label{basic-enest-1}  \sup_{t \in [0,T]}
\calI (\overline t_\tau
(t), \overline z_\tau (t))  + \int_0^T \calR_\epsi (\hat{z}_\tau'
(r)) \,\dd r \leq C_1,
\\
& \label{basic-enest-2} \sup_{t\in[0,T]}\| \overline{z}_\tau(t)
\|_
{W^{1,\il} (\Omega) } + \sup_{t\in[0,T]}\| \hat{z}_\tau(t) \|_
{W^{1,\il}(\Omega) } \leq C_1.
\\
&  \label{basic-enest-3}  \| \overline{u}_\tau \|_{L^\infty (0,T;
W^{1,p_*} (\Omega;\R^d)) } \leq C_1.
\end{align}
\end{lemma}
\begin{proof}
From \eqref{def_time_incr_min_problem_eps} (with competitor
$z=z_k^\tau$) we deduce
\begin{equation}
\label{basic-k}
\begin{aligned}
\calI(t_{k+1}^\tau, z_{k+1}^\tau) + \tau_k \calR_\epsi
\left(\frac{z_{k+1}^\tau-z_{k}^\tau}{\tau_k} \right)   \leq
\calI(t_{k+1}^\tau, z_{k}^\tau)   = \calI(t_{k}^\tau,
z_{k}^\tau) + \int_{t_k^\tau}^{t_{k+1}^\tau} \partial_t \calI (s,
z_{k}^\tau) \,\dd s.
\end{aligned}
\end{equation}
Then, we observe that $\sup_{t \in [0,T]} |\partial_t \calI (t,
z_{k}^\tau)| \leq C$ thanks to \eqref{stim3} in Lemma \ref{l:diff_time}. Hence,
\eqref{basic-enest-1} follows upon adding \eqref{basic-k}  up for
$k=0, \ldots, N-1$. Observe that \eqref{basic-enest-1} yields
\eqref{basic-enest-2}  thanks to \eqref{est_coerc1} in Lemma \ref{l:ex_min} and the
Poincar\'e inequality. Finally, \eqref{basic-enest-3} follows from
\eqref{basic-enest-2} via estimate~\eqref{e:w1p}.
\end{proof}

\subsection{Higher spatial  differentiability for the damage variable}
\label{ss:5.2}
\noindent Theorem \ref{app_reg_thm_z} yields an enhanced
differentiability estimate for $\overline{z}_\tau$ and  $
\hat{z}_\tau$,
uniform w.r.t.\ $\tau$ and $\epsi$.
\begin{lemma}
 \label{prop-diff_z}
Under Assumptions \ref{assumption:energy}, \ref{ass:init}, (A$_\Omega$1) and
(A$_\Omega$2), for every
$\beta\in[0, \frac{1}{q}(1-\frac{d}{q}))$,  for every
$z_0 \in \calZ$
 it holds
\[
  \overline{z}_\tau(t), \, \hat{z}_\tau(t)\,  \in
  W^{1+\beta,q}(\Omega) \quad \text{for all $\tau>0$ and all $t\in (0,T]$}.
\]
Moreover there exists a constant $C_2>0$ such that for all
 $\beta\in [0, \frac{1}{q} (1-\frac{d}{q}))$,
  $\tau>0$,
and $\epsi>0$ there holds
\begin{equation}
\label{high-diff-zeta} \|\overline{z}_\tau\|_{L^{2\il} (0,T;
W^{1+\beta,q}(\Omega))} + \|\hat{z}_\tau\|_{L^{2\il} (0,T;
W^{1+\beta,q}(\Omega))} \leq   C_2\,.
\end{equation}
\end{lemma}
\begin{proof}
Applying Theorem \ref{app_reg_thm_z} with $\alpha=2$,
$\zeta=z_{k}^\tau$,  $w=u_\text{min}(t_{k+1}^\tau,z_{k+1}^\tau) +
u_D(t_{k+1}^\tau)$ and $p=p_*$, we find
\begin{align}
\label{unif_reg_est_z_discr}
\begin{aligned}
 \norm{z^\tau_{k+1}}_{W^{1+\beta,q}(\Omega)}
 \leq  & c_\beta (1+ \norm{z_k^\tau}_{W^{1,\il}(\Omega)})
 \\
 &
  \times \left( 1 +
\norm{u_\text{min}(t_{k+1}^\tau,z_{k+1}^\tau)+
u_D(t_{k+1}^\tau)}_{W^{1,p_*}(\Omega;\R^d)}^\frac{1}{q}
 +
\epsilon^\frac{1}{q}\norm{\frac{z_{k+1}^\tau-z_k^\tau}{\tau}}_{L^2(\Omega)}^{\frac{1}{q}
} \right),
\end{aligned}
\end{align}
with $c_\beta$ independent of $\tau$  \emph{and} $\epsi$. Taking into
account the previously proved uniform estimates
\eqref{basic-enest-2} and \eqref{basic-enest-3} for $z_k^\tau$ and
$u_\text{min}(t_{k+1}^\tau,z_{k+1}^\tau)$, we then have
\[
\norm{\overline{z}_\tau (t)}_{W^{1+\beta,q}(\Omega)}^{2q} \leq C
\left(1+ \epsilon^2 \norm{\hat{z}_\tau'(t)}_{L^2(\Omega)}^2 \right)
\qquad \foraa\, t \in (0,T).
\]
Then, \eqref{high-diff-zeta} follows from integrating the above
estimate on $(0,T)$ and using \eqref{basic-enest-1}.
\end{proof}

\subsection{Enhanced temporal regularity estimates}
\label{ss:5.4} The proof of the enhanced regularity estimates
\eqref{enhanced-est-1} and \eqref{enhanced-est-2} below relies on the higher regularity for
$z_0$ guaranteed by  condition  \eqref{enhanced-initial-datum}, i.e.,
$\rmD_z \calI (0,z_0) \in L^2(\Omega)$.
We also provide estimate \eqref{enhanced-est-3}, which shall be used in the proof of Lemma \ref{l:uniform-w11}, cf.\
\eqref{e:disc_L1_13} below. Observe that the bounds in
\eqref{enhanced-est-1}--\eqref{enhanced-est-3} \DDDS might \DDDE
explode as $\epsi \to 0$.
\begin{lemma}
\label{l:enhanced-reg}
 Under Assumptions \ref{assumption:energy}, \ref{ass:init}, (A$_\Omega$1) and (A$_\Omega$2), for every $z_0 \in \calZ$ such that
 \eqref{enhanced-initial-datum} is valid
 and for every $\epsi>0$ there exists a constant $C_3=C_3(\epsi)>0$, with $C_3(\epsi) \to \infty$ as $\epsi \to 0$,  such that for all $\tau>0$ there
holds
\begin{align}
\label{enhanced-est-1} & 
\int_0^T \int_\Omega (1+ |\nabla \hat{z}_\tau (r)|^2 )^{\frac{q-2}{2}}
|\nabla
\hat{z}_\tau'(r)|^2 \,\dd x \,\dd r
\leq C_3(\epsi),
\\
\label{enhanced-est-2} & \epsi \norm{\hat{z}_\tau'}_{L^\infty (0,T;
L^2(\Omega))}^2 \leq C_3(\epsi),
\\
& \epsilon\norm{\hat z_\tau'\left(\frac{t_1^\tau}{2}\right)}_{L^2(\Omega)}
\leq
 C_{3,1} \exp{(  C_{3,2}   \tau/\epsilon)}
\label{enhanced-est-3}
\end{align}
 where $t^\tau_1$ is the first non-zero node of the
partition of $[0,T]$
and the constants $ C_{3,1},\, C_{3,2}$ do not
depend on $\epsilon$ or $\tau$.
\end{lemma}
\begin{proof}
For $t\in (t_k^\tau,t_{k+1}^\tau)$ we define $\overline{h}_\tau(t)
:=\epsilon  \hat z_\tau'(t)   + A_\il\overline z_\tau(t) +
\rmD_z\wt\calI(\overline t_\tau(t),\overline z_\tau(t))$. 
 Hence, relation
\eqref{cont-reformulation} is equivalent to $-
\overline{h}_\tau(t)\in
\partial\calR_1(\hat z'_\tau(t))$ for $t\in
(t_k^\tau,t_{k+1}^\tau)$. By the $1$-homogeneity of $\calR_1$   we
deduce
\begin{align}
\forall \, t\in (t_k^\tau,t_{k+1}^\tau) &&-\calR_1(\hat z_\tau'(t))
=\langle \overline{h}_\tau(t),  \hat
z_\tau'(t)\rangle_{\calZ},
\label{pruni1:e5}\\
\forall\, r\in [0,T]\backslash\{t_0^\tau,\ldots,t_N^\tau\}
&&\calR_1(\hat z_\tau'(t))\geq
 \langle - \overline{h}_\tau(r),  \hat z_\tau'(t)\rangle_{\calZ}.
 \label{pruni1:e5bis}
\end{align}
Adding both relations and choosing $\rho\in (t_i^\tau,
t_{i+1}^\tau)$ and $\sigma\in (t_{i-1}^\tau, t_i^\tau)$,  it follows
\begin{align*}
0\geq \tau^{-1}\langle \overline{h}_\tau(\rho) -
\overline{h}_\tau(\sigma),\hat z_\tau'(\rho) \rangle_{\calZ}.
\end{align*}
This relation can be rewritten as
\begin{multline}
\label{pruni1:e2} \epsilon \tau^{-1}  \ddd{ \pairing{L^2(\Omega)}{
  \hat z_\tau'(\rho) -  \hat z_\tau'(\sigma)}
{\hat z_\tau'(\rho)}}{$=I_1$}{}   + \ddd{ \tau^{-1}
\pairing{\calZ}{A_\il \overline z_\tau(\rho) -
     A_\il \overline
z_\tau(\sigma))} {\hat z_\tau'(\rho)} }{$=I_2$}{}
\\\leq
\ddd{-\tau^{-1}  \pairing{\calZ}{\rmD_z\wt\calI(\overline
t_\tau(\rho), \overline z_\tau(\rho)) -
 \rmD_z\wt\calI(\overline t_\tau(\sigma), \overline z_\tau(\sigma))}{ \hat
 z_\tau'(\rho)}}{$=I_3$}{}.
\end{multline}
Now, we observe that
\begin{equation*}
I_1 \geq \frac12\int_\Omega \left(  |\hat z_\tau'(\rho)|^{2} - |\hat
z_\tau'(\sigma)|^{2}  \right) \,\dd x
\end{equation*}
whereas, relying on inequality \eqref{e:from-D-thesis},  we
find
\begin{equation}
\label{est-for-I2}
I_2 \geq c\int_\Omega (1+ |\nabla \overline{z}_\tau (\rho) |^2  +
|\nabla \overline{z}_\tau (\sigma) |^2 )^{\frac{\il-2}{2}}   |\nabla
\hat{z}'_\tau (\rho) |^2  \,\dd x
  \geq c \int_\Omega (1+ |\nabla \hat{z}_\tau (\rho) |^2 )^{\frac{\il-2}{2}}
  |\nabla
\hat{z}'_\tau (\rho) |^2  \,\dd x,
\end{equation}
 where the second inequality is due to the fact that

  $|\nabla \hat{z}_\tau (\rho)|^2\leq  2|\nabla \overline{z}_\tau (\rho)|^2 + 2 |\nabla \overline{z}_\tau (\sigma)|^2$.
Finally, relying on estimate \eqref{very-useful-later}, we obtain
\begin{equation}
\label{est-for-I3}
|I_3| \leq C  (1 +  \| \hat z_\tau'(\rho)
\|_{L^{2p_*/(p_*-2)}(\Omega)} ) \| \hat z_\tau'(\rho)
\|_{L^{2p_*/(p_*-2)}(\Omega)}\,.
\end{equation}
All in all,
inserting the above calculation
in \eqref{pruni1:e2} and multiplying by $\tau$
 we find
\begin{align}
\label{pruni1:2e} \frac{\epsilon}{2}\norm{\hat
  z_\tau'(\rho)}^2_{L^2(\Omega)}
&+ \tau C  \int_\Omega\big(1 + \abs{\nabla \hat z_\tau(\rho)}^2 \big)^\frac{q-2}{2}\abs{\nabla \hat z'_\tau(\rho)}^2\dx
\\
&\leq \frac{\epsilon}{2}\norm{\hat
  z_\tau'(\sigma)}^2_{L^2(\Omega)}
+\tau C (1+\norm{\hat
z'_\tau(\rho)}_{L^{2p_*/(p_*-2)}(\Omega)})\norm{\hat
  z'_\tau(\rho)}_{L^{2p_*/(p_*-2)}(\Omega)}.
\nonumber
\end{align}
Hence, ``integrating'' \eqref{pruni1:2e} on the time interval
$(t_0,t)$ with $t_0\in (0,t_1^\tau)$ and $t\in (t_k^\tau,t_{k+1}^\tau)$ we arrive at
\begin{align}
\frac{\epsilon}{2}\norm{\hat
  z_\tau'(t)}^2_{L^2(\Omega)}
&+  C\int_{t_1^\tau}^{\overline t_\tau(t)}
\int_\Omega\big(1 + \abs{\nabla \hat z_\tau(\rho)}^2 \big)^\frac{q-2}{2}\abs{\nabla \hat z'_\tau(\rho)}^2\,\dd x\,\dd\rho
\nonumber\\
&\leq \frac{\epsilon}{2}\norm{\hat
  z_\tau'(t_0)}^2_{L^2(\Omega)}
+C \int_{t_0}^{\overline t_\tau(t)}
 (1+\norm{\hat z'_\tau(\rho)}_{L^{2p_*/(p_*-2)}(\Omega)}^2) \,\d\rho,
\label{pruni1:3e}
\end{align}
where we have also used Young's inequality.
 For the first time step with $t_0\in (0,t_1^\tau)$
 and using  \eqref{enhanced-initial-datum}  
we obtain from \eqref{pruni1:e5}:
\[
\begin{aligned}
0=\calR_1(\hat z_\tau'(t_0)) + \langle \bar h_\tau(t_0),\hat
z_\tau'(t_0)\rangle_\calZ \geq \epsi\norm{\hat
z_\tau'(t_0)}_{L^2(\Omega)}^2 + \langle
\rmD_z\calI(t_1^\tau,z_1^\tau),\hat z_\tau'(t_0)\rangle_\calZ.
\end{aligned}
\]
We now use that
 $
\rmD_z\calI(t_1^\tau,z_1^\tau) = \rmD\calI_q(z_1^\tau)-\rmD\calI_q(z_0)
+ \rmD_z\wt\calI(t_1^\tau,z_1^\tau)-  \rmD_z\wt\calI(0,z_0)
+ \rmD_z\calI(0,z_0),
$
 and with Young's inequality we find
 \begin{equation}
 \label{1st-time}
\begin{aligned}
&\epsilon\norm{\hat z_\tau'(t_0)}_{L^2(\Omega)}^2
+c\tau \int_\Omega(1+\abs{\nabla\hat z_\tau(t_0)}^2 )^{\frac{q-2}{2}}\abs{\nabla\hat z_\tau'(t_0)}^2 \dx
\nonumber\\
&\quad \leq -\langle \rmD_z\calI(0,z_0),\hat z_\tau'(t_0)\rangle_\calZ +
\langle \rmD_z \wt\calI(0,z_0)-\rmD_z\wt\calI(t_1^\tau,z_1^\tau),
\hat z_\tau'(t_0)\rangle_\calZ
\nonumber\\
&\quad \leq \frac{\epsilon}{2}\norm{\hat
  z_\tau'(t_0)}_{L^2(\Omega)}^2
+ \epsilon^{-1} \norm{\rmD_z\calI(0,z_0)}_{L^{2}(\Omega)}^{2}
  + c\tau (1 + \norm{\hat z_\tau'(t_0)}_{L^{{2p_*/(p_*-2)}}(\Omega)})\norm{\hat
    z_\tau'(t_0)}_{L^{{2p_*/(p_*-2)}}(\Omega)}.
\nonumber
\end{aligned}
\end{equation}
We  
sum the above estimate with \eqref{pruni1:3e}.
Adding the term $ \int_{0}^{\overline t_\tau(t)} \|\hat
z'_\tau(\rho)\|_{L^{2}(\Omega)}^2 \,\dd \rho $ to both sides of the
resulting inequality, we obtain
\begin{align}
 & \frac{\epsilon}{2}\norm{\hat z_\tau'(t)}^2_{L^2(\Omega)}
+ 
 C_1
\int_{0}^{\overline t_\tau(t)} \norm{\hat
z'_\tau(\rho)}_{L^{2}(\Omega)}^2 \,\d\rho
+ C_1 \int_0^{\bar t_\tau(t)} \int_\Omega(1+\abs{\nabla \hat
z_\tau(\rho)}^2)^{\frac{q-2}{2}} \abs{\nabla \hat z'(\rho)}^2\,\dx\,\d\rho
\nonumber\\
&\leq
 \epsilon^{-1}
\norm{\rmD_z\calI(0,z_0)}_{L^{2}(\Omega)}^{2}
+  C_1  \int_{0}^{\overline
t_\tau(t)} \|\hat z'_\tau(\rho)\|_{L^{2}(\Omega)}^2 \,\dd \rho
+ C_2
\int_{0}^{\overline t_\tau(t)}
 (1+\norm{\hat z'_\tau(\rho)}_{L^{2p_*/(p_*-2)}(\Omega)}^2)
\,\d\rho
\nonumber \\
 &
 \leq
C + \epsilon^{-1} \norm{\rmD_z\calI(0,z_0)}_{L^{2}(\Omega)}^{2}
+
 C
\int_{0}^{\overline t_\tau(t)} \|\hat
z'_\tau(\rho)\|_{L^{2}(\Omega)}^2 \,\dd \rho
 +  \frac{ C_1}4
\int_{0}^{\overline t_\tau(t)} \norm{\hat
z'_\tau(\rho)}_{W^{1,2}(\Omega)}^2 \,\d\rho,
 \label{pruni1:4e}
\end{align}
where for the last inequality we have applied estimate
\eqref{lions-magenes} in such a way as to absorb $ \frac{C_1}4
\int_{0}^{\overline t_\tau(t)} \norm{\hat
z'_\tau(\rho)}_{W^{1,2}(\Omega)}^2 \d\rho$
 into the corresponding term on the left-hand side.
Now with the Gronwall inequality, we conclude that for 
 all $t\in [0,T]\backslash\{t_0^\tau,\ldots,
t_N^\tau\}$
\begin{align}
\label{more-general-below} \epsilon\norm{\hat
z_\tau'(t)}^2_{L^2(\Omega)} \leq \left(C' +
  \frac{1}{2\epsilon}\norm{\rmD_z\calI(0,z_0)}^2_{L^2(\Omega)}\right)
 \exp(C\overline{t}_\tau(t)/\epsilon),
\end{align}
from which we derive 
 \eqref{enhanced-est-2},  \eqref{enhanced-est-3} and \eqref{enhanced-est-1}.
\end{proof}

 The following result, providing a $W^{1,2} (0,T;W^{1,2} (\Omega;\R^d)) $-estimate for
 $\hat{u}_\tau$ that is not uniform w.r.t.\ $\epsi$,
 is a direct consequence of estimate \eqref{enhanced-est-1} of Lemma \ref{l:enhanced-reg}, via Lemma \ref{l:cddata}.
\begin{lemma}
\label{l:highinteu} Under Assumptions \ref{assumption:energy}, \ref{ass:init}, (A$_\Omega$1), and (A$_\Omega$2), for every $z_0 \in \calZ$  such that \eqref{enhanced-initial-datum} is valid 
there exists a constant $C_4 = C_4(\epsi)>0$,   with $C_4(\epsi)  \to \infty$ as $\epsi \to 0$,     such that for all $\tau>0$ there holds
\begin{align}
&
 \label{basic-highinte-1}
 \| \hat{u}_\tau \|_{W^{1,2} (0,T;
W^{1,2} (\Omega;\R^d)) } \leq  C_4(\epsi).
\end{align}
\end{lemma}

\subsection{A uniform discrete $\BV$-estimate }
\label{ss:5.5.}
 The following estimates will be used to pass to the vanishing viscosity limit $\epsi\to0$ and therefore
are uniform both w.r.t.\ $\tau$ \emph{and} w.r.t.\  $\epsi$.
\begin{lemma}
\label{l:uniform-w11}
 Under Assumptions \ref{assumption:energy},  \ref{ass:init}, (A$_\Omega$1) and (A$_\Omega$2),
 for every $z_0 \in \calZ$ such that  \eqref{enhanced-initial-datum} 
 is valid  there exists a constant $C_5>0$ such that for all $\tau>0$  \emph{and} $\epsi>0$
 with $\tau\leq 2\epsilon$
 there holds
\begin{align}
\label{W11} & \int_0^T
 \norm{\hat{z}_\tau'(t)}_{W^{1,2}(\Omega))} \,\dd t  \leq C_5,
 \\
 &
 \label{mixed-uniform-epsi}
 \int_0^T  \norm{\hat{z}_\tau'(t)}_{L^{2}(\Omega))} \,\dd t +  \int_0^T \left(
\int_\Omega (1+  |\nabla \hat{z}_\tau (r)|^2
)^{\frac{q-2}{2}} |\nabla \hat{z}_\tau'(r)|^2 \,\dd x \right)^{\frac{1}{2}} \,\dd r  \leq C_5.
\end{align}
\end{lemma}
\noindent
Note that, in comparison with the previous \eqref{enhanced-est-1}, formula
\eqref{mixed-uniform-epsi} has an $L^1$-character, in the sense that it can be
rewritten as
\begin{equation}
\label{mixd-estim-later}
\|\mixed{\tau} \|_{L^1 (0,T)} \leq C \quad \text{with} \quad
\mixed{\tau} (t):= \left( \norm{\hat{z}_\tau'(t)}^2_{L^{2}(\Omega))} +
\int_\Omega (1+  |\nabla \hat{z}_\tau (r)|^2 
)^{\frac{q-2}{2}} |\nabla \hat{z}_\tau'(t)|^2 \,\dd x \right)^{\frac{1}{2}}.
\end{equation}
\begin{proof}
We start from \eqref{pruni1:e2}, written for $\rho=m_k$ and
$\sigma=m_{k-1}$, where $m_k:=\frac{1}{2}(t_{k-1}^\tau + t_k^\tau)$
and
 $k\in
\{2,\ldots,N\}$. Adding the term $\norm{\hat
z_\tau'(m_k)}^2_{L^2(\Omega)}$ on both sides, we obtain
\begin{multline}
\label{e:disc_L1_1}
\frac{\epsilon}{\tau}\langle \hat z_\tau'(m_k) -
\hat z_\tau'(m_{k-1}), \hat z_\tau'(m_k)\rangle_{L^2(\Omega)} +
\tau^{-1}\langle A_q\overline{z}_\tau (m_k) -
 A_q\overline{z}_\tau (m_{k-1}),\hat z_\tau'(m_k)\rangle_{\calZ}
+\norm{\hat z_\tau'(m_k)}^2_{L^2(\Omega)}
\\
\leq
-\tau^{-1} \langle \rmD_z\wt\calI(t_k,
 \overline{z}_\tau (m_k)) -
\rmD_z\wt\calI(t_{k-1},   \overline{z}_\tau (m_{k-1})), \hat
z_\tau'(m_k)\rangle_{\calZ}
+\norm{\hat z_\tau'(m_k)}^2_{L^2(\Omega)},
\end{multline}
where $\wt\calI$ is defined as in \eqref{itilde}.
 Thanks to estimate \eqref{e2.22} and the fact that
   $|\nabla\hat{z}_\tau(m_k)|^2\leq 2  |\nabla\overline{z}_\tau(m_k)|^2 + 2  |\nabla\overline{z}_\tau(m_{k-1})|^2$,
   the left-hand side of \eqref{e:disc_L1_1} can be bounded by
\begin{equation}\label{eq:lhs}
\text{L.H.S. }  \geq \frac{\epsilon}{2\tau} \norm{\hat
z_\tau'(m_k)}_{L^2(\Omega)} \left(  \norm{\hat
z_\tau'(m_k)}_{L^2(\Omega)} - \norm{\hat
z_\tau'(m_{k-1})}_{L^2(\Omega)} \right)  + \mixed{k}^2,
\end{equation}
where we use the place-holder  (cf.\ notation \eqref{mixd-estim-later})
\[
 \mixed{k}^2 := c_q\int_{\Omega} (1+ |\nabla\hat{z}_\tau(m_k)|^2 )^{\frac{q-2}{2}}
|\nabla\hat z'_\tau(m_k)|^2 \,\dd x + \norm{\hat z_\tau'(m_k)}^2_{L^2(\Omega)}
\]
with a constant $c_q\in (0,1]$.
Using   the previously proved estimate \eqref{est-for-I3} for the
first term on the
right-hand side of \eqref{e:disc_L1_1},
 and the fact that $W^{1,2}(\Omega)$ is compactly
embedded in $L^{{2p_*}/{p_*-2}}(\Omega)\subseteq L^1(\Omega)$, we have
that the right-hand side of \eqref{e:disc_L1_1}
can be bounded as follows  (see the proof of \cite[Proposition~4.3]{krz})
\begin{align*}
\text{R.H.S.}& \leq \frac{ c_q }{2}\norm{\hat
z_\tau'(m_k)}^2_{W^{1,2}(\Omega)} +
C\left(1 + \norm{\hat z_\tau'(m_k)}_{L^1(\Omega)} \calR_1(\hat
 z_\tau'(m_k))\right)
\\&  \leq
\frac{1}{2}\mixed{k}^2
+ C\left(1 + \norm{\hat z_\tau'(m_k)}_{L^2(\Omega)} \calR_1(\hat
 z_\tau'(m_k))\right).
\end{align*}
Hence, estimate \eqref{e:disc_L1_1}  yields
\begin{align*}
&\frac{\epsilon}{2\tau}
\norm{\hat z_\tau'(m_k)}_{L^2(\Omega)}
\left(  \norm{\hat z_\tau'(m_k)}_{L^2(\Omega)} -
\norm{\hat z_\tau'(m_{k-1})}_{L^2(\Omega)}
\right)
+ \frac{1}{2}\mixed{k}^2
\\ &\quad \leq
C\left(1 + \norm{\hat z_\tau'(m_k)}_{L^2(\Omega)} \calR_1(\hat
 z_\tau'(m_k))
\right),
\end{align*}
where the constant $C$ is independent of $\tau,k$ and $\epsilon$.
Multiplying this inequality by $4\tau/\epsilon$ and taking into account that $\mixed{k}^2  \geq \norm{\hat z_\tau'(m_k)}^2_{L^2(\Omega)}$
we arrive at
\begin{equation}
\label{e:disc_L1_3}
\begin{split}
&2
\norm{\hat z_\tau'(m_k)}_{L^2(\Omega)}
\left(  \norm{\hat z_\tau'(m_k)}_{L^2(\Omega)} -
\norm{\hat z_\tau'(m_{k-1})}_{L^2(\Omega)}
\right)
+
\frac{\tau}{\epsilon} \norm{\hat z_\tau'(m_k)}^2_{L^2(\Omega)}
+ \frac{\tau}{\epsilon} \mixed{k}^2 
\\ &\leq
\frac{4\tau C}{\epsilon}
+\frac{4\tau C}{\epsilon}
\norm{\hat z_\tau'(m_k)}_{L^2(\Omega)} \calR_1(\hat z_\tau'(m_k)),
\end{split}
\end{equation}
which is valid for all $2\leq k\leq N$.
We define now for $0\leq i\leq N-1$
\begin{gather*}
a_i = \norm{\hat z_\tau'(m_{i+1})}_{L^2(\Omega)}, \quad
b_i =
(\tau/\epsilon)^\frac{1}{2} \mixed{i+1},\quad 
c_i= (4\tau C/\epsilon)^\frac{1}{2}, \quad
d_i=\frac{2\tau C}{\epsilon}\calR_1(\hat z_\tau'(m_{i+1})), \quad \gamma
=\frac{\tau}{2\epsilon},
\end{gather*}
 so that \eqref{e:disc_L1_3} can be rewritten as
$
2a_i(a_i - a_{i-1}) + 2\gamma a_i^2 + b_i^2 \leq c_i^2 + 2 a_i d_i, $
which holds for $1\leq i\leq N-1$.
At this point, we follow the lines of the proof of \cite[Proposition~4.3]{krz}
(cf.\ also \cite[Lemma 3.4]{MZ10}), which relies on a time-discrete Gronwall
estimate with weights (see \cite[Lemma 4.1]{krz}, \cite[Lemma 3.17]{NSV00}): this is
where the assumption $\tau\leq 2\epsilon$ is needed. Hence,
  we
arrive at
\begin{equation}
\label{e:disc_L1_13}
\sum_{k=2}^N \tau \mixed{k}  
  \leq
  C  \left(T + \epsilon \norm{\hat z_\tau'(m_1)}_{L^2(\Omega)}
  + \sum_{k=2}^N\tau \calR_1(\hat z_\tau'(m_k))\right).
\end{equation}
Thanks to \eqref{enhanced-est-3} from Lemma \ref{l:enhanced-reg}
we conclude that \eqref{mixed-uniform-epsi} and therefore \eqref{W11} hold.
\end{proof}

For later use
we pin down a crucial consequence of  the higher differentiability estimate \eqref{unif_reg_est_z_discr} for $\overline{z}_\tau$,
 and of
the uniform $W^{1,1} (0,T; L^2(\Omega))$-estimate for $\hat{z}_\tau$, combined with \eqref{basic-enest-1}.
\begin{lemma}
\label{l:eps-mixed-unif}
Under Assumptions \ref{assumption:energy}, \ref{ass:init}, (A$_\Omega$1) and (A$_\Omega$2),
for every $z_0 \in \calZ$ with \eqref{enhanced-initial-datum},
there exists a constant $C_6>0$ such that for all $\beta\in[0,
\frac{1}{q}(1-\frac{d}{q}))$, $\tau>0$,
and $\epsi>0$ there holds
\begin{equation}
\label{high-diff-zeta-mysterious}
\int_0^T
\|\overline{z}_\tau(t)\|_{
W^{1+\beta,q}(\Omega)}^\il \|\hat{z}_\tau' (t) \|_{L^{2}
(\Omega)}  \,\dd t \leq C_6\,.
\end{equation}
\end{lemma}
\begin{proof}
From \eqref{unif_reg_est_z_discr}, again taking into account
the previously proved uniform estimates
\eqref{basic-enest-2} and \eqref{basic-enest-3} for $z_k^\tau$ and
$u_\text{min}(t_{k+1}^\tau,z_{k+1}^\tau)$, we also gather
\[
 \norm{z^\tau_{k+1}}_{W^{1+\beta,q}(\Omega)}^{\il}
\leq C\left( 1 +
\epsilon\norm{\frac{z_{k+1}^\tau-z_k^\tau}{\tau}}_{L^2(\Omega)} \right),
\]
whence
\begin{equation}
\label{to-be-integrated}
 \norm{z^\tau_{k+1}}_{W^{1+\beta,q}(\Omega)}^{\il}\norm{\frac{z_{k+1}^\tau-z_k^\tau}{\tau}}_{L^2(\Omega)}
\leq C\left(  \norm{\frac{z_{k+1}^\tau-z_k^\tau}{\tau}}_{L^2(\Omega)}+
\epsilon\norm{\frac{z_{k+1}^\tau-z_k^\tau}{\tau}}_{L^2(\Omega)}^2 \right).
\end{equation}
Then, \eqref{high-diff-zeta-mysterious}  follows by integrating \eqref{to-be-integrated} in time, taking into account
the basic energy estimate \eqref{basic-enest-1} as well as estimate \eqref{W11}.
\end{proof}

\begin{remark}
\upshape
\label{rmk:fixed-time-step}
Observe that the a priori estimates from Lemmas
\ref{l:basicenest}--\ref{l:highinteu} could be obtained also in the case of a time-discretization scheme with  \emph{variable} time step $\tau_k= t_{k+1}^\tau - t_k^\tau$, with fineness $\tau = \sup_{0\leq k\leq N}(t^\tau_{k+1} - t^\tau_k)$. Accordingly,  part 1 of Theorem~\ref{thm:ex-viscous}  could be extended to  the  variable time step framework, like in \cite{krz}.
The reason why we have confined ourselves to a \emph{constant} time step is in fact  related to the validity of some calculations in the proof of Lemma  \ref{l:uniform-w11}.
\end{remark}


\section{Proof of Theorem \ref{thm:ex-viscous}  on the existence of viscous solutions}
\label{s:6}

In this section $\epsi>0$ is fixed and the limit as $\tau$ tends to zero is discussed.
  In order to pass to the limit in the time-discretization scheme of the viscous problem,
as in \cite{krz} we are going to adopt a \emph{variational} approach, along the lines of \cite{mrs2013}. Namely, instead of
taking the limit of the discrete subdifferential inclusion \eqref{cont-reformulation}, we shall pass to the limit in the discrete energy
inequality \eqref{discr-enineq} derived in Lemma~\ref{l:discr-enineq} below.   Observe that,
one of the peculiarities of this problem is that we have not used inequality \eqref{discr-enineq} to deduce the basic energy  estimates for the approximate solutions  like it could be
expected. In fact, the last remainder term
on the right-hand side of \eqref{discr-enineq}
prevents us from doing so. Instead, relying on the a priori estimates obtained in Section \ref{s:5} and on suitable compactness arguments (see the forthcoming Prop.~\ref{prop:compactness}), we are going to show that this remainder tends to zero,  cf.\ \eqref{conve-4} ahead).

\begin{lemma}[Discrete energy inequality]
\label{l:discr-enineq}
$ $\\
Under Assumptions \ref{assumption:energy}, \ref{ass:init}, and (A$_\Omega$1), 
the discrete solutions of \eqref{cont-reformulation} satisfy the \emph{discrete energy inequality} for all $0 \leq s \leq t \leq T$
\begin{equation}
\label{discr-enineq}
\begin{aligned}
&
\int_{\underline{t}_{\tau}(s)}^{\overline{t}_{\tau}(t)}
\left(\calR_\epsi (\hat{z}'_{\tau})(r)+\calR_\epsi^* (-\rmD_z
\calI (\overline{t}_{\tau}(r),\overline{z}_{\tau}(r)))  \right)
\,\mathrm{d}r +
\calI( t,\hat{z}_{\tau}(t))
\\ &
\leq \calI( s,\hat{z}_{\tau}(s)) +
\int_{\underline{t}_{\tau}(s)}^{\overline{t}_{\tau}(t)}
\partial_t \calI (r,\hat{z}_{\tau}(r)) \, \mathrm{d}r
\\
&+ C \sup_{t\in[0,T]}\|\overline{z}_{\tau}(t)-\hat{z}_{\tau}(t)\|_{L^{2p_*/(p_*-2)}(\Omega)}
\int_{\underline{t}_{\tau}(s)}^{\overline{t}_{\tau}(t)}(|(\overline{t}_{\tau}(r)-r | + \|\overline{z}_{\tau}(r)-\hat{z}_{\tau}(r)\|_{L^{2p_*/(p_*-2)}(\Omega)} )\,\dd r.
\end{aligned}
\end{equation}
\end{lemma}
\begin{proof}
From \eqref{cont-reformulation} and  as a consequence of the Fenchel-Moreau theorem we get  
 \begin{equation}\label{ing-1}
  \calR_\epsilon \left(\hat{z}_\tau'(r)\right) + \calR_\epsilon^* \left(- \rmD_z\calI(\overline{t}_\tau(r),
 \overline{z}_\tau(r)) \right) = \pairing{\calZ}{- \rmD_z\calI(\overline{t}_\tau(r),
 \overline{z}_\tau(r))}{\hat{z}_\tau'(r)} \qquad \foraa\, r \in (0,T)\,.
 \end{equation}
 On the other hand, since $\hat{z}_\tau\in \rmC_{\mathrm{lip}}^0([0,T];\calZ)$,
 the \emph{standard} chain rule  yields
\[
\frac{\dd }{\dd t } \calI (r,\hat {z}_\tau (r)) =  \partial_t \calI (r,\hat{z}_{\tau}(r)) +  \pairing{\calZ}{ \rmD_z\calI(r,
 \hat{z}_\tau(r))}{\hat{z}_\tau'(r)} \quad \foraa\, r \in (0,T)
\]
Thus the right-hand side of \eqref{ing-1} can be rewritten as
 \begin{equation}\label{ing2}
 \begin{aligned}
  &
  \pairing{\calZ}{ \rmD_z\calI(\overline{t}_\tau(r),
 \overline{z}_\tau(r))}{\hat{z}_\tau'(r)}   \\ & =  \frac{\dd }{\dd t } \calI (r,\hat {z}_\tau (r))  -   \partial_t \calI (r,\hat{z}_{\tau}(r))
 +  \pairing{\calZ}{\rmD_z \calI (\overline{t}_{\tau}(r), \overline{z}_{\tau}(r)) -  \rmD_z \calI (r, \hat{z}_{\tau}(r)) }{\hat{z}'_{\tau}(r)}.
 \end{aligned}
 \end{equation}
 Then, combining \eqref{ing-1} and \eqref{ing2} and
  integrating on the interval
$({\underline{t}_{\tau}(s)},{\overline{t}_{\tau}(t)} )$ we get
\begin{equation}\label{discr-enid}
\begin{aligned}
&\int_{\underline{t}_{\tau}(s)}^{\overline{t}_{\tau}(t)}
\left(\calR_\epsi (\hat{z}'_{\tau})(r)+\calR_\epsi^* (-\rmD_z
\calI (\overline{t}_{\tau}(r),\overline{z}_{\tau}(r)))  \right)
\,\mathrm{d}r +
\calI( t,\hat{z}_{\tau}(t))
\\ &
= \calI(s,\hat{z}_{\tau}(s)) +
\int_{\underline{t}_{\tau}(s)}^{\overline{t}_{\tau}(t)}
\partial_t \calI (r,\hat{z}_{\tau}(r)) \, \mathrm{d}r
-\int_{\underline{t}_{\tau}(s)}^{\overline{t}_{\tau}(t)}
\pairing{\calZ}{\rmD_z \calI (\overline{t}_{\tau}(r),
  \overline{z}_{\tau}(r)) -  \rmD_z \calI (r, \hat{z}_{\tau}(r))
}{\hat{z}'_{\tau}(r)}
\,\mathrm{d}r\,.
\end{aligned}
\end{equation}
Let us estimate now the 
last term on the right-hand side:
\begin{equation*}
\begin{aligned}
&\int_{\underline{t}_{\tau}(s)}^{\overline{t}_{\tau}(t)}\pairing{\calZ}{\rmD_z \calI (\overline{t}_{\tau}(r),
  \overline{z}_{\tau}(r)) -  \rmD_z \calI (r, \hat{z}_{\tau}(r))
}{\hat{z}'_{\tau}(r)}  \,\dd r
\\
&= \int_{\underline{t}_{\tau}(s)}^{\overline{t}_{\tau}(t)}\pairing{\calZ}{A_q \overline{z}_{\tau}(r) - A_q \hat{z}_{\tau}(r)}{\hat{z}'_{\tau}(r)}\,\dd r +  \int_{\underline{t}_{\tau}(s)}^{\overline{t}_{\tau}(t)}\pairing{\calZ}
{\rmD_z \wt{\calI} (\overline{t}_{\tau}(r),  \overline{z}_{\tau}(r)) -  \rmD_z \wt{\calI} (r, \hat{z}_{\tau}(r))}
{\hat{z}'_{\tau}(r)}\,\dd r
\\
&=: F_1 + F_2.
\end{aligned}
\end{equation*}
 Now, from the definition of $ \hat{z}_{\tau}$ and \eqref{e2.22} it follows that $F_1\geq0$.
To estimate $F_2$, we use \eqref{very-useful-later} from Corollary \ref{c:very-useful-later}
and, observing that $P(\overline{z}_{\tau},\hat{z}_{\tau})$ (for $P(z_1,z_2)$ defined as in \eqref{def:Pzz})
is bounded uniformly in $\tau$ thanks to \eqref{basic-enest-2}, we get
\[
\begin{aligned}
&|F_2|\leq C \int_{\underline{t}_{\tau}(s)}^{\overline{t}_{\tau}(t)}(|(\overline{t}_{\tau}(r)-r | + \|\overline{z}_{\tau}(r)-\hat{z}_{\tau}(r)\|_{L^{2p_*/(p_*-2)}(\Omega)} )
\|\overline{z}_{\tau}(r)-\hat{z}_{\tau}(r)\|_{L^{2p_*/(p_*-2)}(\Omega)}\,\dd r
\\
&\leq C \sup_{t\in[0,T]}\|\overline{z}_{\tau}(t)-\hat{z}_{\tau}(t)\|_{L^{2p_*/(p_*-2)}(\Omega)}
\int_{\underline{t}_{\tau}(s)}^{\overline{t}_{\tau}(t)}(|(\overline{t}_{\tau}(r)-r | + \|\overline{z}_{\tau}(r)-\hat{z}_{\tau}(r)\|_{L^{2p_*/(p_*-2)}(\Omega)} )\,\dd r,
\end{aligned}
\]
which together with \eqref{discr-enid} and the fact that $-F_1\leq 0$ gives \eqref{discr-enineq}.
\end{proof}

 \noindent As a consequence of the a priori estimates of Sec.\ \ref{s:5}, we have the following result.
\begin{proposition}[Compactness]
\label{prop:compactness}
$ $\\
Under Assumptions \ref{assumption:energy},  \ref{ass:init}, (A$_\Omega$1), and (A$_\Omega$2),
for every $z_0\in\calZ$ such that $\rmD_z\calI(0,z_0)\in L^2(\Omega)$ and
 for every sequence of time-steps $(\tau_j)_j$ tending to $0$
 there exist a (not-relabeled) subsequence and
 \[
 \begin{gathered}
z
\in  
L^\infty (0,T;\calZ) \cap W^{1,2} (0,T; W^{1,2}(\Omega))
\end{gathered}
\]
fulfilling the mixed estimate \eqref{mixed-estimate},  as well as the enhanced regularity \eqref{enhanced-reg},
and such that the following convergences hold:  for all
$\beta\in [0, \frac{1}{q}(1-\frac{d}{q}))$
\begin{gather}
\label{conve-1}
\overline{z}_{\tau_j}, \, \hat{z}_{\tau_j}  \weaksto z
\qquad
   \text{in }  L^{2q} (0,T; W^{1+\beta, q }(\Omega))\cap L^\infty
(0,T;\calZ),
\\
\label{conve-2}
\hat{z}_{\tau_j}  \weakto z \qquad
   \text{in }   W^{1,2} (0,T;W^{1,2}(\Omega)),
\\
\label{conve-3}
  \hat{z}_{\tau_j}   \to z \qquad
  \text{strongly in } L^{2q}(0,T;W^{1+\beta, q}(\Omega)),
  \\
\label{conve-4}
\sup_{t\in[0,T]}\norm{\overline{z}_{\tau_j} (t) - \hat{z}_{\tau_j}(t)}_{W^{1,2}(\Omega)}  \leq C(\epsi)\sqrt{\tau_j} 
  \\
\label{conve-5}
\sup_{t\in [0,T]}
\norm{ \rmD_z
\calI (\overline{t}_{\tau_j}(t),\overline{z}_{\tau_j}(t)) -  \rmD_z
\calI (t, \hat{z}_{\tau_j}(t))}_{\calZ^*}\leq C(\epsi)\sqrt{\tau_j}.
\end{gather}
Therefore,  \eqref{conve-3}, \eqref{conve-4} and \eqref{conve-5} imply
\begin{equation}
\label{conve-4-conseq}
\rmD_z
\calI (\overline{t}_{\tau_j}(t),\overline{z}_{\tau_j}(t)) \to  \rmD_z
\calI (t, z(t))  \qquad \text{strongly in }\calZ^* \  \foraa\ t\in (0,T).
\end{equation}
Moreover,
\begin{align}
\label{conve-6}
&
\hat{z}_{\tau_j}(t) \weakto z(t)     \  \text{ in } \calZ  &&    \text{for all } t \in [0,T],
\\
\label{conve-7}
&
\calI(t,\hat{z}_{\tau_j}(t) ) \to \calI(t,z(t)) &&  \text{for almost all } t \in (0,T).
\end{align}
\end{proposition}
\begin{proof}
Convergences \eqref{conve-1}--\eqref{conve-2} are a straightforward consequence of estimates \eqref{basic-enest-2}, \eqref{high-diff-zeta}, and \eqref{enhanced-est-1} via the Banach selection principle.

 Estimate \eqref{high-diff-zeta} implies that $\hat{z}_{\tau_j}
 \weakto z$ in $L^{2q} (0,T; W^{1+\beta, q }(\Omega))$ for every
 $\beta\in [0,\frac{1}{q}\big(1-\frac{d}{q}\big))$.
This fact,
 together with \eqref{conve-2} and \cite[Corollary 4]{simon87}
yields the strong convergence \eqref{conve-3} due to the compact
embedding $W^{1+\beta_1, q }(\Omega) \subset W^{1+\beta_2, q
}(\Omega)$ for $\beta_1>\beta_2$.
Now,  \eqref{conve-4} follows from  the bound
 $\| \hat{z}'_{\tau} \|_{L^2 (0,T; W^{1,2}(\Omega))} \leq C $ (cf.\ estimates \eqref{enhanced-est-1} and \eqref{enhanced-est-2}).

In order to prove  estimate \eqref{conve-5}, we notice that for every $w\in\calZ$
\begin{equation}
\label{e:c5.1}
\begin{aligned}
&|\pairing{\calZ}{ \rmD_z\calI (\overline{t}_{\tau_j}(t),\overline{z}_{\tau_j}(t)) -  \rmD_z
\calI (t, \hat{z}_{\tau_j}(t))}{w}|
\\
&\leq
|\pairing{\calZ}{A_q\overline{z}_{\tau_j}(t) - A_q\hat{z}_{\tau_j}(t)}{w}|
+ |\pairing{\calZ}{ \rmD_z\wt{\calI} (\overline{t}_{\tau_j}(t),\overline{z}_{\tau_j}(t)) -  \rmD_z
\wt{\calI} (t, \hat{z}_{\tau_j}(t))}{w}| =: F_1 + F_2.
\end{aligned}
\end{equation}
By estimate \eqref{A7-D} and a careful application of the H\"older inequality with $\frac12 +\frac{q-2}{2q}+\frac1q=1$ 
\begin{equation*}
\begin{aligned}
F_1 
 \leq & C\left( \int_\Omega (1+|\nabla \overline{z}_{\tau_j}(t)|^2 + |\nabla \hat{z}_{\tau_j}(t)|^2)^{\frac{q-2}{2}} |\nabla (\overline{z}_{\tau_j}(t) - \hat{z}_{\tau_j}(t))|^2\,\dd x\right)^{\frac{1}{2}}
 \\
& \times (1 + \|\overline{z}_{\tau_j}(t)\|_{W^{1,q}(\Omega)} +
\|\hat{z}_{\tau_j}(t)\|_{W^{1,q}(\Omega)} )^{\frac{q-2}{2}} \|\nabla
w\|_{L^q(\Omega)}.
\end{aligned}
\end{equation*}
Therefore, by using the energy estimate \eqref{basic-enest-2} we obtain
\begin{equation*}
\begin{aligned}
&\|A_q\overline{z}_{\tau_j}(t) - A_q\hat{z}_{\tau_j}(t)\|_{\calZ^*} \leq
C\left( \int_\Omega (1+|\nabla \overline{z}_{\tau_j}(t)|^2 + |\nabla\underline{z}_{\tau_j}(t)|^2)^{\frac{q-2}{2}} \tau^2|\nabla \hat{z}'_{\tau_j}(t)|^2\,\dd x\right)^{\frac{1}{2}}
\\
&
\leq C \sqrt{\tau} \int_0^T\left(\int_\Omega (1+|\nabla \overline{z}_{\tau_j}(t)|^2 + |\nabla\underline{z}_{\tau_j}(t)|^2)^{\frac{q-2}{2}} |\nabla \hat{z}'_{\tau_j}(t)|^2\,\dd x\right)^{\frac{1}{2}} \,\dd t
\leq     C \sqrt{\tau C_3(\epsi)  },
\end{aligned}
\end{equation*}
where for the second estimate we have again used the H\"older inequality and  \eqref{enhanced-est-1} for the last one.
All in all,
\begin{equation}
\label{e:c5.2}
\|A_q\overline{z}_{\tau_j}(t) - A_q\hat{z}_{\tau_j}(t)\|_{\calZ^*} \leq C(\epsi) \sqrt{\tau}.
\end{equation}
Now we estimate $F_2$. By Corollary \ref{c:very-useful-later} and the embedding of $W^{1,2}(\Omega)$ in $L^{2p_*/(p_*-2)}(\Omega)$
\begin{equation*}
\begin{aligned}
F_2&\leq C (|\overline{t}_{\tau_j}(t)-t| + \|\overline{z}_{\tau_j}(t) - \hat{z}_{\tau_j}(t)\|_{L^{2p_*/(p_*-2)}(\Omega)} )
\| w\|_{L^{2p_*/(p_*-2)}(\Omega)}\\
&\leq C(\tau + \|\overline{z}_{\tau_j}(t) - \hat{z}_{\tau_j}(t)\|_{W^{1,2}(\Omega)})\| w\|_{L^{2p_*/(p_*-2)}(\Omega)}.
\end{aligned}
\end{equation*}
Therefore, taking into account \eqref{conve-4} we get
\begin{equation}
\label{e:c5.3}
\|\rmD_z\wt{\calI} (\overline{t}_{\tau_j}(t),\overline{z}_{\tau_j}(t)) -  \rmD_z
\wt{\calI} (t, \hat{z}_{\tau_j}(t))\|_{\calZ*} \leq C(\tau + \sqrt{\tau}),
\end{equation}
and \eqref{e:c5.1}--\eqref{e:c5.3} give \eqref{conve-5}.

Now, from \eqref{conve-3} it follows that $ \hat{z}_{\tau_j}(t)   \to z(t)$ strongly in $W^{1+\beta, q}(\Omega)) $ for a.a.\ $t\in(0,T)$. Thus, by \eqref{strong-continuity} in
Corollary~\ref{coro-fre},
$\rmD_z \calI (t, \hat{z}_{\tau_j}(t)) \to \rmD_z \calI (t, z(t))$ strongly in $\calZ^*$ for a.a.\ $t\in(0,T)$.
 This,
together with 
\eqref{conve-5} yields \eqref{conve-4-conseq}.

The mixed estimate \eqref{mixed-estimate} follows
 from estimate \eqref{enhanced-est-1}
by lower semicontinuity of the
functional $(A,B)\mapsto \int_0^T\int_\Omega
(1+\abs{A}^2)^{\frac{q-2}{2}}\abs{B}^2 \,\dd x \,\dd t $,  which is
 convex in $B$,
observing that \eqref{conve-2} implies $\nabla \hat z_{\tau_j}'\rightharpoonup
\nabla z'$ in
$L^2((0,T)\times \Omega)$ and that \eqref{conve-3} implies $\nabla \hat
z_{\tau_j}  \to \nabla z$   in $L^1((0,T)\times \Omega)$ (see e.g.\
\cite[Theorem 3.23]{dacorogna08}).

Convergence \eqref{conve-6} follows from the fact that $L^\infty (0,T;\calZ)
\cap W^{1,2} (0,T; W^{1,2}(\Omega)) $ is compactly embedded in $\rmC^0([0,T];
\calX)$ for every $\calX$ such that
$\calZ \Subset \calX \subset W^{1,2}(\Omega)$ (cf., e.g.,
\cite{simon87}), combined with the estimate
$\sup_{j\in\N}\sup_{t\in[0,T]}\|\hat{z}_{\tau_j}\|_{W^{1,q}(\Omega)}\leq C$,
cf.~\eqref{basic-enest-2}.

Finally, from \eqref{conve-3} we get pointwise convergence in $W^{1+\beta,
q}(\Omega)) $ for a.a.\ $t$. Then, the continuity of $z\mapsto \calI(t,z)$
ensues~\eqref{conve-7}.
\end{proof}

  The convergences \eqref{conve-1}--\eqref{conve-7} are  sufficient to pass to the limit in the time-discretization scheme, and conclude
the existence of a weak solution (in the sense of Def.\ \ref{weak-def-sol}), to the Cauchy problem \eqref{visc-eps-dne}--\eqref{Cauchy-condition}.
In order  to   \bnnc deduce by lower semicontinuity arguments \ennc the uniform w.r.t.\ $\epsi$-estimates
\eqref{crucial-elle1-esti}--\eqref{crucial-elle3-mixed} for \emph{any} family of solutions $(z_\eps)$
arising from the time-discretization procedure of Sec.~\ref{s:4},    additional compactness arguments are needed, which we develop in the forthcoming
Lemma
\ref{lemma:eps-indep-est}. We postpone its statement and proof after
the proof of Theorem \ref{thm:ex-viscous}.

\begin{proof}[Proof of Theorem \ref{thm:ex-viscous}]
For fixed $\epsilon>0$ let  $(\tau_j)_{j\in \N}$ be a sequence along which the convergences in
Proposition~\ref{prop:compactness} are valid. Proposition~\ref{prop:compactness} also ensures that, for the limit curve $z$ fulfills the mixed estimate \eqref{mixed-estimate} holds.

First of all, we pass to the limit in the discrete energy inequality
\eqref{discr-enineq}.
Thanks to convergence \eqref{conve-6}, for all $t\in [0,T]$ it holds
that
 $\liminf_{j\to\infty}\calI( t,\hat{z}_{\tau_j}(t))\geq \calI(t,z(t))$
 while, from \eqref{conve-7}
$\calI(s,\hat{z}_{\tau_j}(s)) \rightarrow \calI(s,z(s))$
for a.a.\ $s\in (0,T)$.
The convergence of the term involving
 $\partial_t\calI$ is an immediate consequence of the convergence
 stated in \eqref{conve-3}, taking into account the
continuity properties of $\partial_t\calI$ (see estimate \eqref{stim5} in Lemma \ref{l:diff_time}).
Due to \eqref{conve-4-conseq} and the lower semicontinuity of $\calR_\epsilon^*$
we conclude that
\[
\liminf_{\tau_j}\int_{\underline{t}_{\tau_j}(s)}^{\bar
  t_{\tau_j}(t)} \calR_\epsilon^*(-\rmD_z\calI(\bar t_{\tau_j}(r),\bar
z_{\tau_j}(r)))\dr
\geq \int_{s}^{t} \calR_\epsilon^*(-\rmD_z\calI(r,
z(r)))\dr.
\]
Similarly, from \eqref{conve-2}, by lower semicontinuity it
follows that
\[
\liminf_{\tau_j} \int_{\underline{t}_{\tau_j}(s)}^{\bar
  t_{\tau_j}(t)} \calR_\epsilon(\hat z'_{\tau_j}(r))\dr \geq
\int_{s}^{t} \calR_\epsilon( z'(r))\dr.
\]
Moreover, the remainder term on the right-hand side of \eqref{discr-enineq}
tends to zero thanks to \eqref{conve-4}  and the embedding of $W^{1,2}(\Omega)$ in $L^{2p_*/(p_*-2)}(\Omega)$.
Altogether we arrive at the energy inequality
\begin{equation}
\label{epsi-lim-enineq}
\begin{aligned}
\int_s^t
\left(\calR_\epsi (z'(r))+\calR_\epsi^* (-\rmD_z
\calI (r,z(r))  \right)
\,\dd r +
\calI(t,z(t))
\leq \calI(s,z(s))
+
\int_s^t
\partial_t \calI (r,z(r)) \,\dd r,
\end{aligned}
\end{equation}
for all $t\in[0,T]$, 
 for $s=0$, and for almost all $0 <s<t$.

We now check that \eqref{epsi-lim-enineq} holds for all $0\leq s\leq t$. Let
$s_n\nearrow s$ be a sequence of points for which \eqref{epsi-lim-enineq} is satisfied. Thus,
\[
\begin{aligned}
&\int_{s_n}^s
\left(\calR_\epsi (z'(r))+\calR_\epsi^* (-\rmD_z \calI (r,z(r))  \right)
\,\dd r
\leq \calI(s_n,z(s_n)) - \calI(s,z(s)) +
\int_{s_n}^s \partial_t \calI (r,z(r)) \,\dd r
\\
& =
- \int_{s_n}^s \int_\Omega (1+ |\nabla z(r)|^2)^{(q-2)/2} \nabla z(r) \cdot \nabla z'(r) \,\dd x \,\dd r
-  \int_{s_n}^s  \int_\Omega \rmD_z \wt \calI(r,z(r)) z'(r) \,\dd x \,\dd r
\end{aligned}
\]
where the equality follows by an integrated-in-time version of the chain rule formula~\eqref{chain-rule-formula}.
Passing to the limit as $s_n\nearrow  s $
  and using the absolute continuity of the Lebesgue integral, from the second inequality
we derive $\calI(s_n,z(s_n)) \to
\calI(s,z(s)) $, and therefore we obtain
\eqref{energy-inequality} for all $s$ and $t$. Thanks to Proposition~\ref{prop:equivalence}, we conclude that
$z$ is a weak solution (in the sense of Def.\ \ref{weak-def-sol}), to the Cauchy problem \eqref{visc-eps-dne}--\eqref{Cauchy-condition}.

Estimates \eqref{crucial-elle1-esti}--\eqref{crucial-elle3-mixed} follow from Lemma~\ref{lemma:eps-indep-est} below.
\end{proof}

\noindent The proof of the following Lemma exploits Young measure tools, which we recall in Appendix \ref{s:a-1}.
\begin{lemma}
\label{lemma:eps-indep-est}
 Under Assumptions \ref{assumption:energy}, \ref{ass:init}, (A$_\Omega$1) and (A$_\Omega$2),
for every $z_0 \in \calZ$ such that $\rmD_z \calI
(0,z_0) \in L^2 (\Omega)$ and for every $\epsi>0$ estimates \eqref{crucial-elle1-esti}--\eqref{crucial-elle2-esti} hold.
In addition, $z$ from Proposition~\ref{prop:compactness} also fulfills
\begin{equation}
\label{liminf-1}
\int_0^T \|z(t)\|^q_{W^{1+\beta-\delta, q }(\Omega)} \|z'(t)\|_{L^2(\Omega)}\,\dd t\leq
\liminf_{j\to 0} \int_0^T \|\overline{z}_{\tau_j}(t)\|^q_{W^{1+\beta, q }(\Omega)} \|\hat{z}'_{\tau_j}(t)\|_{L^2(\Omega)}\,\dd t
\end{equation}
for all $\beta\in[0, \frac{1}{q}\big(1-\frac{d}{q}\big))$,
  and
\begin{equation}
\label{liminf-2}
\begin{aligned}
&
\int_0^T\left( \int_\Omega (1+|\nabla z(t)|^2)^{\frac{q-2}{2}} |\nabla z'(t)|^2 \, \dd x  \right)^{\frac12}\,\dd t
\\
 &\quad
 \leq
\liminf_{j\to 0} \int_0^T\left( \int_\Omega (1+ |\nabla \hat{z}_{\tau_j}(t)|^2  
)^{\frac{q-2}{2}} |\nabla \hat{z}'_{\tau_j}(t)|^2 \, \dd x  \right)^{\frac12}\,\dd t.
\end{aligned}
\end{equation}
As a consequence, estimates \eqref{crucial-elle3-esti}--\eqref{crucial-elle3-mixed} hold.
\end{lemma}
\begin{proof}
Estimates \eqref{crucial-elle1-esti}--\eqref{crucial-elle2-esti}
follow from \eqref{basic-enest-1}, \eqref{basic-enest-2}, and \eqref{high-diff-zeta} by convergences \eqref{conve-1}--\eqref{conve-2} and lower semicontinuity arguments.

In order to prove inequalities \eqref{liminf-1}  and
\eqref{liminf-2}, we resort to a Young measure argument, based on
Appendix \ref{s:a-1}.
Indeed, we can apply the compactness theorem \ref{thm.balder-gamma-conv}, with the space $V=W^{1+\beta,q}(\Omega) \times W^{1,2}(\Omega) $, to  the sequence $(\overline{z}_{\tau_j}, \hat{z}_{\tau_j}')_j$,
 bounded in
$L^{2} (0,T; W^{1+\beta,q}(\Omega) \times W^{1,2}(\Omega))$  for all
$\beta\in[0, \frac{1}{q}(1-\frac{d}{q}))$.
Therefore, up to a not relabeled subsequence, $(\overline{z}_{\tau_j}, \hat{z}_{\tau_j}')_j$ admits a limiting
Young measure $\mmu = \{ \mu_t \}_{t \in (0,T)} \in \mathscr{Y} (0,T;W^{1+\beta,q}(\Omega) \times W^{1,2}(\Omega)) $, such that for almost all $t \in (0,T)$
the measure $\mu_t$ is
concentrated on the  limit points of  $(\overline{z}_{\tau_j}(t), \hat{z}_{\tau_j}'(t))_j$, w.r.t.\ the
$W^{1+\beta,q}(\Omega) \times W^{1,2}(\Omega)$-weak topology.
 Now, by  \eqref{conve-3}--\eqref{conve-4}  
 we have that $\overline{z}_{\tau_j}(t) \to z(t)$ strongly in $W^{1,q}(\Omega)$.
 Therefore, denoting by $\pi_1$ the projection operator $(z,v) \in W^{1+\beta,q}(\Omega) \times W^{1,2}(\Omega) \mapsto z \in W^{1,q}(\Omega)$, it is immediate to check that the projection measure $(\pi_1)_{\sharp}(\mu_t)$ coincides with the Dirac delta $\delta_{z(t)}$. With a disintegration argument we in fact see that $\mu_t$ is of the form $ \delta_{z(t)} \otimes \nu_t$, and that the parameterized measure $\{\nu_t\}_{t \in (0,T)}$
 fulfills
 \begin{equation}
 \label{bary-v}
 \int_{W^{1,2}(\Omega)} v \,\dd \nu_t(v) = z'(t) \qquad \text{ for almost all $t \in (0,T)$.}
 \end{equation}
Then, the $\liminf$-inequality \eqref{heq:gamma-liminf} with the
normal
 integrand $\mathcal{H}(t,(z,v)) :=
\|z\|^q_{W^{1+\beta, q }(\Omega)} \|v\|_{L^2(\Omega)}
$
yields
\[
\begin{aligned}
 & \liminf_{j\to 0} \int_0^T \|\overline{z}_{\tau_j}(t)\|^q_{W^{1+\beta, q }(\Omega))} \|\hat{z}'_{\tau_j}(t)\|_{L^2(\Omega)}\,\dd t
 \\
 & \quad \geq \int_0^T \iint_{W^{1+\beta, q }(\Omega) \times W^{1,2}(\Omega)}
 \| z\|^q_{W^{1+\beta, q }(\Omega))} \|v\|_{L^2(\Omega)}\,\dd (\delta_{z(t)} \otimes \nu_t)(z,v)
 \,\dd t
 \\ &
 \geq \int_0^T \| z(t)\|^q_{W^{1+\beta, q }(\Omega))} \left\|\int_{W^{1,2}(\Omega)} v \,\dd \nu_t(v)
\right \|_{L^2(\Omega)}
 = \int_0^T \|z(t)\|^q_{W^{1+\beta, q }(\Omega)} \|z'(t)\|_{L^2(\Omega)}\,\dd t
\end{aligned}
\]
where the second estimate is due to Jensen's inequality and the last equality to
\eqref{bary-v}. This gives \eqref{liminf-1}.

As for \eqref{liminf-2},   we now consider the
 sequence
of gradients
 $(\nabla  \hat{z}_{\tau_j}, 
\nabla\hat{z}_{\tau_j}')_j$,  bounded in $L^2 (0,T; L^2(\Omega; \R^d) \times L^2(\Omega; \R^d) )$.
Relying on Theorem  \ref{thm.balder-gamma-conv}, we  associate with $(\nabla  \hat{z}_{\tau_j}, 
\nabla\hat{z}_{\tau_j}')_j$
its limiting Young measure
 $\widetilde{\mmu}= \{\widetilde{\mu}_t \}_{t \in (0,T)} \in
\mathscr{Y} (0,T;L^2(\Omega; \R^d) \times L^2(\Omega; \R^d) )$, concentrated on the set of the weak-$L^2(\Omega; \R^d) \times L^2(\Omega; \R^d) $ limit points of $(\nabla  \hat{z}_{\tau_j}, 
\nabla\hat{z}_{\tau_j}')_j$.
On account of the strong convergence \eqref{conve-3}, arguing as in the above lines we
 conclude that $\widetilde{\mmu}_t= \delta_{\nabla z(t)} \otimes \widetilde{\nu}_t$ for almost all $t \in (0,T)$, with
 $\{  \widetilde{\nu}_t \}_{t \in (0,T)}$
 satisfying
 \begin{equation}
 \label{bary-tilde-nu}
 \int_{L^2(\Omega; \R^d)} B \,\dd \widetilde{\nu}_t (B) =\nabla  z'(t) \qquad \text{ for almost all $t \in (0,T)$.}
 \end{equation}
Therefore, inequality \eqref{heq:gamma-liminf} with the normal integrand
 $\mathcal{H}(t, (A,B)) :=
\left(  \int_\Omega
(1+\abs{A}^2)^{\frac{q-2}{2}}\abs{B}^2 \dx \right)^{\frac{1}{2}}$
yields
\[
\begin{aligned}
&
\liminf_{j\to 0} \int_0^T\left( \int_\Omega (1+ |\nabla \hat{z}_{\tau_j}(t)|^2  
)^{\frac{q-2}{2}} |\nabla z'_{\tau_j}(t)|^2 \,\dd x  \right)^{\frac12}\,\dd t
\\ & \quad
\geq \int_0^T  \iint_{L^2(\Omega; \R^d) \times L^2(\Omega; \R^d) }
\left(  \int_\Omega
(1+\abs{A}^2)^{\frac{q-2}{2}}\abs{B}^2 \dx \right)^{\frac{1}{2}}
 \,\dd ( \delta_{\nabla z(t)} \otimes \widetilde{\nu_t})(A,B)\,\dd t
\\
 & \quad
 = \int_0^T \int_{L^2(\Omega; \R^d)}
\left( \int_\Omega (1+|\nabla z(t)|^2)^{\frac{q-2}{2}} |B|^2  \,\dd x \right)^{\frac12}
\,\dd  \widetilde{\nu}_t(B)
\,\dd t
\\
& \quad
\geq
\int_0^T \left(
 \int_\Omega (1+|\nabla z(t)|^2)^{\frac{q-2}{2}} \, \left| \int_{L^2(\Omega; \R^d)} B  \,\dd  \widetilde{\nu}_t(B)  \right|^2    \,\dd x \right)^{\frac12}
\,\dd t
\end{aligned}
\]
where the latter estimate again follows from Jensen's inequality. Then,  in view of \eqref{bary-tilde-nu},    \eqref{liminf-2} ensues.

Estimates \eqref{crucial-elle3-esti}--\eqref{crucial-elle3-mixed} are then a consequence of \eqref{liminf-1} and
\eqref{liminf-2}, combined with the bounds \eqref{high-diff-zeta-mysterious} and \eqref{mixed-uniform-epsi},  respectively.
 \end{proof}

\section{Vanishing viscosity limit}
\label{s:7}
 \noindent  Throughout this section, we shall work
with a family
 $(z_\epsilon)_\epsi \subset L^\infty (0,T; W^{1,\il} (\Omega)) \cap W^{1,2} (0,T;W^{1,2}(\Omega)) $  of
 \emph{weak} solutions  (in the sense of Definition \ref{def:wsol}),  to the
$\epsilon$-viscous Cauchy problem
\eqref{visc-eps-dne}--\eqref{Cauchy-condition}. We shall suppose that for $(z_\epsilon)_\epsi$
 the following estimates, uniform w.r.t.\ the parameter $\epsi$, are valid:
\begin{subequations}
\label{epsilon-uniform-estimates}
\begin{align}
\label{epsilon-uniform-estimates-1} &  \sup_{\epsi>0}
\| z_\epsi \|_{W^{1,1} (0,T; L^2 (\Omega))}
 \leq C,
\\
&
\label{epsilon-uniform-estimates-2}
 \sup_{\epsi>0}
 \| z_\epsi \|_{L^{2\il}(0,T, W^{1+\beta, \il}(\Omega)) \cap L^\infty (0,T; W^{1, \il}(\Omega)) } \leq C,
 \\
 &
 \label{epsilon-uniform-estimates-3}
 \sup_{\epsi> 0} \int_0^T \| z_\epsi(t) \|_{W^{1+\beta,
     \il}(\Omega)}^{\il} \| z_\epsi'(t) \|_{L^{2}(\Omega)} \,\dd t
 \leq C, \qquad \text{for every }
 \beta \in \left[0, \frac{1}{q}\big( 1-\frac{d}q\big) \right),
 \\
 &
 \label{epsilon-uniform-estimates-4}
  \sup_{\epsi>0}
  \int_0^T \left(  \int_\Omega (1+ |\nabla z_\epsi(t)|^2)^{\frac{\il-2}{2}} |\nabla z_\epsi' (t)|^2 \,\dd x \right)^{\frac12} \,\dd t \leq C.
\end{align}
\end{subequations}
The existence of \bnnc solutions $(z_\epsi)_\epsi$
fulfilling \eqref{epsilon-uniform-estimates} \ennc
is ensured by Theorem
\ref{thm:ex-viscous}, under the  condition
that  the initial datum $z_0 \in
\calZ$ also fulfills $\rmD_z \calI (0,z_0) \in L^2 (\Omega)$.

  In what follows, we shall reparameterize the curves $(z_\epsi)_\epsi$ by their $L^2(\Omega)$-arclength, and study the asymptotic behavior of the reparameterized  trajectories as $\epsi \to 0$.
This leads (cf.\ Theorem \ref{main-thm-vanvisc}  below) to the notion of \emph{weak parameterized}
 solution to the rate-independent damage system
\eqref{dndia}, which we introduce in Definition \ref{def:parametrized-solution}.

\subsection{ Weak parameterized solutions }
\label{s:7.1}
\noindent   The starting point for the passage of the vanishing viscosity limit is the energy inequality
\eqref{energy-inequality}, which lies at the core of the notion of \emph{weak} solutions to the viscous problem.
Taking into account  the definition of $\calR_\epsi$,  and the fact  that   $\calR_\epsilon^*$ is given by  (see, e.g., \cite[Theorem 3.3.4.1]{ioffe-tihomirov})  
\begin{equation}
\label{repre-conjugate}
\begin{aligned}
\calR_\epsilon^* (\sigma)  
= \frac1{\epsilon}
\min_{\mu \in
\partial\calR_1 (0)}  \wt{\calR}_{2}(\sigma-\mu),
 \quad\text{ with }  \wt{\calR}_{2} (\sigma):= \begin{cases} \tfrac{1}{2}
\norm{\sigma}^2_{L^2(\Omega)} & \text{if $\sigma \in L^2 (\Omega)$,}
\\
\infty & \text{if $\sigma \in \calZ^* {\setminus} L^2 (\Omega)$,}
\end{cases}
 \end{aligned}
 \end{equation}
inequality \eqref{energy-inequality} rephrases as
\begin{equation}
\label{energy-inequality-explicit}
\begin{aligned}
\int_s^t \calR_1 (z'_\epsi(r))  & + \frac{\epsi}{2}\|z'_\epsi(r)\|^2_{L^2(\Omega)}\,\dd r + \int_s^t \frac1{\epsilon} \min_{\mu \in
\partial\calR_1 (0)}  \wt{\calR}_{2}(-\rmD_z \calI (r,z_\epsilon(r))-\mu)
 \big) \,\dd r + \calI (t,z_\epsi(t))
 \\
& \leq  \calI (s,z_\epsi(s))+ \int_s^t \partial_t \calI (r,z_\epsi(r)) \,\dd r \qquad \text{for all $0\leq s\leq t\leq T$.}
\end{aligned}
\end{equation}

 Now, for every $\epsi>0$ we consider
 the
 $L^2(\Omega)$-arclength parameterization of the curve $z_\epsi$, viz.
\begin{equation}
\label{z-arc}
   s_\epsilon(t)=t+\int_0^t \norm{z'_\epsilon(r)}_{L^2(\Omega)}\dr.
\end{equation}
Let $S_\epsilon= s_\epsilon(T)$:
it follows from \eqref{epsilon-uniform-estimates-1} that
\begin{equation}
\label{est-Seps}
\sup_{\epsi>0} S_\epsilon<\infty.
\end{equation}
We introduce the functions
$\tilde{t}_\epsilon:[0,S_\epsilon]\to [0,T]$ and
$\tilde{z}_\epsilon:[0,S_\epsilon]\to\calZ$
\begin{equation}\label{d:hat}
   \tilde{t}_\epsilon(s):= s_\epsilon^{-1}(s),
   \quad
   \tilde{z}_\epsilon(s):=z_\epsilon(\tilde{t}_\epsilon(s))
\end{equation}
 fulfilling the \emph{normalization condition}
\begin{equation}
\label{true-normalization} \tilde{t}_\epsilon'(s) +
\norm{\tilde{z}_\epsilon'(s)}_{L^2(\Omega)} =1 \quad \foraa\, s \in
(0,S_\epsi)
\end{equation}
and study the limiting behavior as $\epsilon\to0$ of the
parameterized trajectories
$\Set{(\tilde{t}_\epsilon(s),\tilde{z}_\epsilon(s))}{s\in[0,S_\epsilon]}$.
It follows from \eqref{est-Seps} that,
  up to a subsequence,
$ S_\epsilon \to S \text{ as $\epsi \to 0$, with } S \geq T$
(the latter inequality follows from the fact that $s_\epsilon(t)
\geq t$). With no loss of generality, we may consider the
parameterized trajectories to be defined on the fixed time interval
$[0,S]$.

From the energy inequality \eqref{energy-inequality-explicit}
 we deduce that the parameterized
trajectories $(\tilde{t}_\epsilon(s),\tilde{z}_\epsilon(s))_{s \in
[0,S]}$
fulfill
\[
\begin{aligned}
 \int_{\sigma_1}^{\sigma_2} \Big(
 \calR_1 (\tilde{z}_\epsilon'(s))
&  {+} \frac{\epsi}{2\tilde{t}_\epsilon'(s)} \norm{\tilde{z}_\epsilon'(s)}_{L^2(\Omega)}^2  {+}
\frac{\tilde{t}_\epsilon'(s)}{2\epsilon} \twodis^2(-\rmD_z \calI
(\tilde{t}_\epsilon(s),\tilde{z}_\epsilon(s)), \partial\calR_1(0))
 \Big) \,\mathrm{d}s
  + \calI(\tilde{t}_\epsilon(\sigma_2),\tilde{z}_\epsilon(\sigma_2)) \\ & \leq
   \calI(\tilde{t}_\epsilon(\sigma_1),\tilde{z}_\epsilon(\sigma_1))
   +\int_{\sigma_1}^{\sigma_2} \partial_t
   \calI(\tilde{t}_\epsilon(s),\tilde{z}_\epsilon(s))\tilde{t}_\epsilon'(s)\,
   \mathrm{d}s \quad \forall\, (\sigma_1,\sigma_2) \subset [0,S],
   \end{aligned}
\]
where we have used the short-hand notation
\[
\twodis(\xi, \partial\calR_1 (0)):= \min_{\mu \in
\partial\calR_1(0)} \sqrt{2\wt{\calR}_2(\xi{-}\mu)}.
\]
Upon introducing the functional
 (cf. \cite[Sec.~3.2]{mrs2009dcds})
\begin{equation}
\label{def-me}  \mathcal{M}_\epsi: (0,\infty) \times L^2(\Omega)
\times [0,\infty) \to [0,\infty], \quad \mathcal{M}_\epsi
(\alpha,v,\zeta):= \calR_1 (v) + \frac{\epsi}{2\alpha} \norm{v}_{L^2
(\Omega)}^2 + \frac{\alpha}{2\epsi} \zeta^2,
\end{equation}
the above inequality  rephrases as
\begin{equation}
\label{parametrized-energy-identity-epsi}
\begin{aligned}
 &  \int_{\sigma_1}^{\sigma_2} \mathcal{M}_\epsi (\tilde{t}_\epsilon'(s),\tilde{z}_\epsilon'(s),   \twodis(-\rmD_z \calI
(\tilde{t}_\epsilon(s),\tilde{z}_\epsilon(s)), \partial\calR_1(0)) )
   \,\mathrm{d}s
  + \calI(\tilde{t}_\epsilon(\sigma_2),\tilde{z}_\epsilon(\sigma_2)) \\ & \quad \leq
   \calI(\tilde{t}_\epsilon(\sigma_1),\tilde{z}_\epsilon(\sigma_1))
   +\int_{\sigma_1}^{\sigma_2} \partial_t
   \calI(\tilde{t}_\epsilon(s),\tilde{z}_\epsilon(s))\tilde{t}_\epsilon'(s)\,
   \mathrm{d}s   \quad \forall\, (\sigma_1,\sigma_2) \subset [0,S].
   \end{aligned}
\end{equation}
We will pass to the limit as $\epsi \to 0$ in \eqref{parametrized-energy-identity-epsi}.
For this, we shall rely on the
following $\Gamma$-convergence/lower semicontinuity
 result,
 \cite[Lemma~3.1]{mrs2009dcds} (cf.\ also \cite[Lemma 5.1]{krz}).
\begin{lemma}
\label{lem:g-conver} Extend the functional $\mathcal{M}_\epsi$
\eqref{def-me} to $[0,+ \infty ) \times L^2 (\Omega) \times
[0,\infty) $ via
 \[ \mathcal{M}_\epsi (0,v,\zeta):=
\begin{cases} 0&\text{for $v=0$ and $\zeta \in [0,+ \infty )$}\,,\\
                \infty&\text{for $v \in L^2(\Omega) {\setminus} \{0\}$ and  $\zeta \in [0,+ \infty )$}\,.
                \end{cases}
\]
Define $\mathcal{M}_0: [0,\infty) \times L^2(\Omega) \times
[0,\infty) \to [0,\infty]$ by
\begin{equation}
\label{e:defg0}
 \mathcal{M}_0(\alpha, v,\zeta):=\begin{cases}
\calR_1 (v) + \zeta \norm{v}_{L^2(\Omega)}  &\text{if $\alpha=0$,}
\\
\calR_1 (v) + \rmI_{{0}} (\zeta)  &\text{if $\alpha>0$,}
\end{cases}
 \end{equation}
 where $\rmI_{{0}}$ denotes the indicator function of the singleton
 $\{ 0\}$.
 Then, \\
(A) $\mathcal{M}_\eps$ $\Gamma$-converges to $\mathcal{M}_0$ on
$[0,\infty) \times L^2(\Omega) \times [0,\infty) $ w.r.\ to
the strong-weak-strong topology.
(B) If $\alpha_\eps \weakto \bar\alpha $ in $L^1(a,b)$, $v_\eps
\weakto \bar v $ in $L^1(a,b;L^2(\Omega))$, and
 $\liminf_{\eps\to0} \zeta_\eps(s)\ge
\bar\zeta(s)$  for a.a.\  $ s \in (a,b)$, then
\[
\int_{a}^{b} \mathcal{M}_0(\bar\alpha(s),\bar v(s),\bar\zeta(s)) \,\dd
s \leq \liminf_{\eps \to 0} \int_{a}^{b}
\mathcal{M}_\eps(\alpha_\eps(s),v_\eps(s),\zeta_\eps(s)) \,\dd s\,.
\]
\end{lemma}

We are now in the position to introduce the notion of solution which arises from passing to the limit as $\epsi \to 0$ in \eqref{parametrized-energy-identity-epsi}.
\begin{definition}[Weak parameterized solutions]
\label{def:parametrized-solution} A pair $(\tilde{t},\tilde{z}) \in
\mathrm{C}_{\mathrm{lip}}^0 ([0,S]; [0,T]\times L^2(\Omega))$ is a
\emph{weak parameterized solution} of  the rate-independent damage system \eqref{dndia}, if
it satisfies the \emph{energy inequality} for all $0 \leq \sigma_1
\leq \sigma_2 \leq S$
\begin{equation}
\label{parametrized-energy-ineq}
\begin{aligned}
 \int_{\sigma_1}^{\sigma_2} \mathcal{M}_0 (\tilde{t}'(s),\tilde{z}'(s), &  \twodis(-\rmD_z \calI
(\tilde{t}(s),\tilde{z}(s)), \partial\calR_1(0)) )
   \,\mathrm{d}s
  + \calI(\tilde{t}(\sigma_2),\tilde{z}(\sigma_2)) \\ & \leq
   \calI(\tilde{t}(\sigma_1),\tilde{z}(\sigma_1))
   +\int_{\sigma_1}^{\sigma_2} \partial_t
   \calI(\tilde{t}(s),\tilde{z}(s))\tilde{t}'(s)
   \,\dd s.
   \end{aligned}
\end{equation}
We say that a weak parameterized solution $(\tilde{t},\tilde{z}) \in
\mathrm{C}_{\mathrm{lip}}^0 ([0,S]; [0,T]\times L^2(\Omega))$ is
\emph{non-degenerate} if it fulfills
\begin{equation}
\label{e:non-deg} \tilde{t}'(s) + \norm{\tilde{z}'(s)}_{L^2(\Omega)} >0 \quad
\foraa\, s \in (0,S). \end{equation}
\end{definition}

 Recall that the chain rule provided by Theorem \ref{thm:chain-rule} is a
key ingredient for  getting further insight into the notion of weak solution to the viscous system from Def.\ \ref{weak-def-sol}.
Indeed, it is by a chain rule argument that we can show that the pointwise variational inequality \eqref{weak-def-sol} is indeed
equivalent to the energy inequality \eqref{energy-inequality}.  Likewise, the
 following result, which is  the parameterized counterpart to the chain rule of Theorem \ref{thm:chain-rule}, shall enable us to obtain a differential characterization of
 the notion of weak parameterized solution in terms of the energy
 inequality \eqref{parametrized-energy-ineq}. Indeed, Prop.\
 \ref{param-counterpart} shall be
 exploited in the proof of Proposition \ref{prop:7.6}.
 \begin{proposition}
\label{param-counterpart}
Under Assumptions  \ref{assumption:energy}, \ref{ass:init}, and
 (A$_\Omega$1),
let
$(\tilde{t},\tilde{z}) \in
\mathrm{C}_{\mathrm{lip}}^0 ([0,S]; [0,T]\times L^2(\Omega))$ fulfill in addition
\begin{align}
&
\label{reg-param1.1}
\tilde{z} \in L^\infty (0,S;W^{1,\il}(\Omega)),
\\
&
\label{reg-param1.2}
\int_0^S \mixed{}(s) \,\dd s <\infty \qquad \text{with }   \mixed{}(s):= \left(  \int_\Omega (1+ |\nabla \tilde z(s)|^2)^{\frac{\il-2}{ 2}}
 |\nabla \tilde{z}' (s)|^2 \,\dd x \right)^{\frac 12}.
\end{align}
Then, the map $s \mapsto  \calI(\tilde{t}(s),\tilde{z}(s)) $
 is absolutely continuous on $(0,S)$, and
 the following chain rule formula is valid:
\begin{equation}\label{chain-rule-param}
\begin{aligned}
\frac{\dd}{\dd s} \calI(\tilde{t}(s),\tilde{z}(s)) - \partial_s\calI(\tilde{t}(s),\tilde{z}(s))\tilde{t}'(s)
= &
\int_\Omega (1+ |\nabla \tilde{z}(s)|^2)^{\frac{q-2}{2}} \nabla \tilde{z}(s) \cdot \nabla \tilde{z}'(s) \,\dd x
\\
& + \int_{\Omega}   \rmD_z \wt{\calI}(\tilde{t}(s),\tilde{z}(s))\tilde{z}'(s) \,\dd x\,.
\end{aligned}
\end{equation}
\end{proposition}
\begin{proof}
From \eqref{reg-param1.1} and \eqref{reg-param1.2} we deduce
with the H\"older inequality
 that
\[
\begin{aligned}
&
\int_0^S \int_\Omega   \left |  1+ |\nabla \tilde{z}(s)|^2)^{\frac{\il-2}{2}}
\nabla \tilde{z}(s) \cdot \nabla \tilde{z}' (s)  \right |    \,\dd x \,\dd s
\\
& = \int_0^S \int_\Omega \left |  (1+ |\nabla \tilde{z}(s)|^2)^{\frac{\il-2}{4}}  \nabla \tilde{z}' (s) \right |  \left | (1+ |\nabla \tilde{z}(s)|^2)^{\frac{\il-2}4} \nabla \tilde{z}(s)  \right |   \,\dd x \,\dd s
\\
& \leq \int_0^S  \mixed{}(s)  \|(1+ |\nabla \tilde{z}(s)|^2)^{\frac{\il-2}{4}}\nabla \tilde{z}(s) \|_{L^2(\Omega)}\,\dd s
\leq c \int_0^S  \mixed{}(s)
(1+\|\nabla
  \tilde{z}(s)\|_{L^q(\Omega)}^{\frac{q}{2}})\ds
 <\infty,
 \end{aligned}
\]
\bnnc where the last estimate relies on \eqref{reg-param1.1}. \ennc
Now we can argue as in the proof of Theorem \ref{thm:chain-rule}
to deduce that  $(\tilde{t},\tilde{z})$ fulfill the parameterized version of the chain rule
\eqref{chain-rule-param}.
%
\end{proof}

\subsection{ The vanishing viscosity result }
\label{ss:7.2}
\noindent
We are now in the position of stating and proving our main vanishing viscosity result.
\begin{theorem}
\label{main-thm-vanvisc}
Under Assumptions \ref{assumption:energy}, \ref{ass:init}, (A$_\Omega$1), and
(A$_\Omega$2), let
$(z_\epsilon)_\epsi$
 be a family of weak solutions (according to Definition \ref{def:wsol}),  in
 $ L^\infty (0,T; W^{1,\il} (\Omega)) \cap W^{1,2} (0,T;W^{1,2}(\Omega)) $,  to the
$\epsilon$-viscous Cauchy problem
\eqref{visc-eps-dne}--\eqref{Cauchy-condition}. Suppose that
 the  estimates \eqref{epsilon-uniform-estimates}
 are valid for  $(z_\epsilon)_\epsi$, and let
$(\tilde{t}_\epsi,\tilde{z_\epsi})_{\epsi>0} \subset
\mathrm{C}_{\mathrm{lip}}^0 ([0,S]; [0,T]\times L^2(\Omega))$ be defined
by \eqref{d:hat}.

Then, for every sequence $\epsi_n \searrow 0$ there exist a pair $(\tilde{t},\tilde{z}) \in
\mathrm{C}_{\mathrm{lip}}^0 ([0,S]; [0,T]\times L^2(\Omega))$,
  such that
$\tilde z$ has the regularity
\begin{equation}
\label{add-reg-reparam}
\tilde z \in  L^\il (0,S;W^{1+\beta, \il}(\Omega) \cap L^\infty (0,S;
W^{1, \il}(\Omega)) \qquad \text{for every }
 \beta \in \left[0,  \frac{1}{q}\big(1-\frac{d}q\big) \right),
\end{equation}
 and a
(not-relabeled) subsequence such that
\begin{equation}
\label{epsi-convergences}
\begin{aligned}
& (\tilde{t}_{\epsi_n},\tilde{z}_{\epsi_n}) \weaksto (\tilde{t},\tilde{z})
\text{ in $W^{1,\infty} (0,S; [0,T] \times L^2(\Omega))$,} \\ &
 \tilde{t}_{\epsi_n} \to \tilde{t} \text{ in $\mathrm{C}^0 ([0,S];
 [0,T])$,} \quad \tilde{z}_{\epsi_n}(s) \weakto \tilde{z}(s) \text{ in $L^2(\Omega)$ for
 all $s \in [0,S]$,}
 \end{aligned}
\end{equation}
and
$(\tilde{t},\tilde{z})$ is a weak parameterized solution of the rate-independent damage system \eqref{dndia},
 fulfilling
\begin{equation}
\label{almost-normalization} \hat{t}'(s) +
\norm{\hat{z}'(s)}_{L^2(\Omega)}
 \leq 1 \quad \foraa\, s \in (0,S).
\end{equation}
  Furthermore, $\tilde z$ fulfills \eqref{reg-param1.2}.
\end{theorem}

For the proof, we will rely on the following a priori estimates for the parameterized solutions
\begin{lemma}\label{lem:unif-est}
Under Assumptions \ref{assumption:energy}, \ref{ass:init}, (A$_\Omega$1), and
(A$_\Omega$2), let
$(z_\epsilon)_\epsi$
 be a family of weak solutions (according to Definition \ref{def:wsol}),  in
 $ L^\infty (0,T; W^{1,\il} (\Omega)) \cap W^{1,2} (0,T;W^{1,2}(\Omega)) $,  to the
$\epsilon$-viscous Cauchy problem
\eqref{visc-eps-dne}--\eqref{Cauchy-condition}. Suppose that $(z_\epsilon)_\epsi$
 satisfy \eqref{epsilon-uniform-estimates}.
  Then
\begin{subequations}
\label{resc-epsilon-uniform-estimates}
\begin{align}
\label{resc-epsilon-uniform-estimates-1} &  \sup_{\epsi>0}
\| \tilde{z}_\epsi \|_{W^{1,\infty} (0,S; L^2 (\Omega))}
 \leq C,
\\
&
\label{resc-epsilon-uniform-estimates-2}
 \sup_{\epsi>0}
 \| \tilde{z}_\epsi \|_{L^\infty (0,S; W^{1, \il}(\Omega)) } \leq C,
 \\
 &
 \label{resc-epsilon-uniform-estimates-3}
 \sup_{\epsi> 0} \| \tilde{z}_\epsi \|_{L^\il (0,S;W^{1+\beta,
     \il}(\Omega)} \leq C \qquad \text{for every }
 \beta \in \left[0, \frac{1}{q}\big(1-\frac{d}q\big) \right),
 \\
 &
 \label{resc-epsilon-uniform-estimates-4}
  \sup_{\epsi>0}
  \int_0^S \left(  \int_\Omega (1+ |\nabla \tilde{z}_\epsi(s)|^2)^{\frac{\il-2}2} |\nabla \tilde{z}_\epsi' (s)|^2 \,\dd x \right)^{\frac12} \,\dd s \leq C.
\end{align}
\end{subequations}
  Moreover, there holds
\begin{equation}
\label{dorothee-added}
\sup_{\epsi>0} \int_0^S  \twodis(-\rmD_z \calI
(\tilde{t}_\epsilon(s),\tilde{z}_\epsilon(s)), \partial\calR_1(0))
  \,\mathrm{d}s \leq C.
\end{equation}

\end{lemma}
\begin{proof}
Estimates \eqref{resc-epsilon-uniform-estimates-1}--\eqref{resc-epsilon-uniform-estimates-2}
are trivial consequences of  \eqref{epsilon-uniform-estimates-1} and \eqref{epsilon-uniform-estimates-2}.
 It can be easily checked that
\eqref{resc-epsilon-uniform-estimates-3} ensues from \eqref{epsilon-uniform-estimates-2} and \eqref{epsilon-uniform-estimates-3} via reparameterization. Moreover, since \eqref{epsilon-uniform-estimates-4}
essentially has a $L^1$-character (cf.\ \eqref{mixd-estim-later}),
it is preserved by the reparameterization in  \eqref{resc-epsilon-uniform-estimates-4}.

  In order to prove  \eqref{dorothee-added}, we observe that
\[
\begin{aligned}
C &  \geq
\int_0^S  \Big(
\frac{\epsi}{2\tilde{t}_\epsilon'(s)} \norm{\tilde{z}_\epsilon'(s)}_{L^2(\Omega)}^2  +
\frac{\tilde{t}_\epsilon'(s)}{2\epsilon} \twodis^2(-\rmD_z \calI
(\tilde{t}_\epsilon(s),\tilde{z}_\epsilon(s)), \partial\calR_1(0))
 \Big) \,\mathrm{d}s
 \\ &
 \geq
  \int_{\Set{s\in (0,S)}{\tilde
     t_\epsilon'(s)\leq \delta}}  \norm{\tilde{z}_\epsilon'(s)}_{L^2(\Omega)} \twodis(-\rmD_z \calI
(\tilde{t}_\epsilon(s),\tilde{z}_\epsilon(s)), \partial\calR_1(0))  \,\dd s
\\ & \quad+
 \int_{\Set{s \in (0,S)}{\tilde{t}_\epsilon'(s)>\delta}}
 \frac{\delta}{2\epsilon}  \twodis(-\rmD_z \calI
(\tilde{t}_\epsilon(s),\tilde{z}_\epsilon(s)), \partial\calR_1(0)) \ds
\end{aligned}
 \]
 for arbitrary $\delta \in (0,1)$. Since in the second term it holds $\norm{\tilde
  z'_\epsilon(s)}_{L^2(\Omega)} =1 -\wt t_\epsilon'(s)\geq 1-\delta$
  in view of the normalization condition \eqref{true-normalization},
   we
conclude.
\end{proof}

Relying on the above result, we now develop the
\begin{proof}[Proof of Theorem \ref{main-thm-vanvisc}.]
From the normalization condition \eqref{true-normalization}, we deduce that there exists
$(\tilde{t},\tilde{z}) \in \mathrm{C}_{\mathrm{lip}}^0 ([0,S];
[0,T]\times L^2(\Omega))$ such that convergences \eqref{epsi-convergences}
hold along some subsequence.
 Further, from estimates \eqref{resc-epsilon-uniform-estimates} it follows (possibly after extracting a further subsequence) that
\begin{gather}
\wt z_\epsilon \overset{*}{\rightharpoonup} \wt z \quad \text{ in } L^\infty(0,S;W^{1,q}(\Omega)),\\
\wt z_\epsilon(s) \rightarrow \wt z(s) \text{ uniformly in $\calX$ for all $W^{1,q}(\Omega)
\Subset \calX\subset L^2(\Omega)$ and all $s\in [0,S]$,  }
\\
\wt z_\epsilon \rightarrow \wt z \text{ strongly in }
 L^q(0,S;W^{1+\beta}(\Omega)) \text{ for all }\beta \in \left[0,
   \frac{1}{q}\big( 1-\frac{d}{q}\big)\right).
\end{gather}

 Arguing as in the proof of Theorem~\ref{thm:ex-viscous}
 and relying on 
 Corollary \ref{coro-fre}, we find that
\begin{equation}
\label{epsi-conv2}
\begin{gathered}
 \lim_{n\rightarrow\infty}
\calI(\tilde{t}_{\epsi_n}(s), \tilde{z}_{\epsi_n}(s)) =
\calI(\tilde{t}(s), \tilde{z}(s)),\quad \rmD_z
\calI(\tilde{t}_{\epsi_n}(s), \tilde{z}_{\epsi_n}(s))\to 
\rmD_z \calI(\tilde{t}(s), \tilde{z}(s))
 \text{ strongly (!)
  in }\calZ^*,\\
\partial_t\calI(\tilde{t}_{\epsi_n}(s),
\tilde{z}_{\epsi_n}(s)) \rightarrow
\partial_t\calI(\tilde{t}(s), \tilde{z}(s)) \quad\text{in $L^1(0,S)$.}
\end{gathered}
\end{equation}
for almost all $s \in (0,S)$.
 Now, \eqref{almost-normalization} follows by taking the limit as $\epsi_n \to 0$ in
 \eqref{true-normalization}, with
a trivial lower semicontinuity argument.
   We then apply Lemma \ref{lem:g-conver}.
Estimate \eqref{dorothee-added} guarantees that for a.a.\  $s \in (0,S)$
 $\liminf_{\epsi_n \to 0}  \twodis(-\rmD_z
\calI (\tilde{t}_{\epsi_n}(s),\tilde{z}_{\epsi_n}(s)), \partial\calR_1(0))  <\infty$.
  In view of     \eqref{epsi-conv2},
 we  have that, for all $0 \leq \sigma_1\leq   \sigma_2\leq S$
 \[
 \begin{aligned}
\liminf_{\epsi_n \to 0} \int_{\sigma_1}^{\sigma_2} \mathcal{M}_{\epsi_n}
(\tilde{t}_{\epsi_n}'(s),\tilde{z}_{\epsi_n}'(s),  &  \twodis(-\rmD_z
\calI (\tilde{t}_{\epsi_n}(s),\tilde{z}_{\epsi_n}(s)),
\partial\calR_1(0)) )
   \,\mathrm{d}s \\ & \geq  \int_{\sigma_1}^{\sigma_2} \mathcal{M}_0 (\tilde{t}'(s),\tilde{z}'(s),   \twodis(-\rmD_z \calI
(\tilde{t}(s),\tilde{z}(s)), \partial\calR_1(0)) )
   \,\dd s.
   \end{aligned}
 \]
Then, combining \eqref{epsi-convergences} and \eqref{epsi-conv2},
and using that $\tilde{z}_\epsi(0)=z_\epsi(0)=z_0$ for all $\epsi>0$,
we pass to the limit in \eqref{parametrized-energy-identity-epsi}
  for all $\sigma_2 \in [0,S]$, for $\sigma_1=0$, and
 for almost all $0< \sigma_1 <\sigma_2$    such that the convergences in \eqref{epsi-conv2} are valid.
We thus find that
the pair $(\tilde{t},\tilde{z})$ satisfies
\eqref{parametrized-energy-ineq}    for all $\sigma_2 \in [0,S]$, for $\sigma_1=0$, and
 for almost all $0 < \sigma_1 <\sigma_2$.

With the same Young measure argument
as in the proof of Lemma \ref{lemma:eps-indep-est}, it follows that the limit function
$\tilde{z}$  satisfies the    \emph{mixed estimate}    \eqref{reg-param1.2}.
Applying the chain rule \eqref{param-counterpart}, we  then conclude in the same way as at the end of the proof of
Theorem \ref{thm:ex-viscous} that the energy inequality in fact holds for \emph{all} $ 0 \leq \sigma_1 \leq \sigma_2 \leq S$.
\end{proof}

  \paragraph{Differential characterization of (non-degenerate) weak parameterized solutions.} Following
  the lines of \cite[Prop.~5.3,Cor.~5.4]{MRS10a} 
and of \cite[Prop.\ 5.1]{krz},
we now aim to provide a characterization of
weak parameterized solutions as solutions of a suitable
subdifferential inclusion. Loosely speaking, the latter should reflect  two evolutionary  regimes for the  damage system, namely
\begin{compactitem}
\item rate-independent evolution when $\tilde{t}'>0$ (and $\tilde z' \neq 0$)
\item (possibly) \emph{viscous} evolution when $\tilde{t}' =0$ (and $\tilde z' \neq 0$).
\end{compactitem}

We have to interpret Proposition \ref{prop:7.6} below in this spirit: for $\tilde{t}'>0$,  the variational inequality
\eqref{t'>0} is a weak formulation of the \emph{rate-independent}
 subdifferential inclusion $\partial \calR_1 (\tilde{z}'(s)) +  \rmD_z \calI (s,\tilde{z} (s)) \ni 0$ for a.a.\ $s \in (0,S)$.
 For $\tilde{t}'=0$, $\tilde z' \neq 0$ follows from the non-degeneracy condition. The system may be subject to \emph{viscous} dissipation. This \emph{viscous} regime is 
seen as a jump in the (slow) external time scale, encoded by the time
function $\tilde{t}$, which is frozen.
Indeed, the variational inequality \eqref{t'=0} is a (very) weak form of the
 \emph{viscous} $\partial \calR_1 (\tilde{z}'(s)) +\lambda(s) \tilde{z}'(s)  +  \rmD_z \calI (s,\tilde{z} (s)) \ni 0$ for a.a.\ $s \in (0,S)$ (with $\lambda : (0,S) \to [0,+\infty)$).

\begin{proposition}[Differential characterization]
\label{prop:7.6}
Under Assumptions \ref{assumption:energy}, \ref{ass:init}, (A$_\Omega$1), and
(A$_\Omega$2), let
$(\tilde{t},\tilde{z})\in \mathrm{C}_{\mathrm{lip}}^0 ([0,S]; [0,T]\times
L^2(\Omega))$ be a \emph{non-degenerate} parameterized weak solution of
\eqref{dndia}
with \eqref{reg-param1.2},
then
\begin{enumerate}
\item
If $\tilde{t}'(s)>0$, then for every $w\in\calZ$
\begin{equation}\label{t'>0}
\begin{aligned}
\calR_1(w) -  \calR_1(\tilde{z}'(s)) &  \geq - \int_\Omega (1+ |\nabla \tilde{z}(s)|^2)^{\frac{q-2}{2}} \nabla \tilde{z}(s) \cdot (\nabla w - \nabla \tilde{z}'(s)) \,\dd x
\\ & \quad
-
\langle \rmD_z \wt{\calI}(\tilde{t}(s),\tilde{z}(s)), w-\tilde{z}' (s)\rangle_{W^{1,2}(\Omega)}\,.
\end{aligned}
\end{equation}
\item
If $\tilde{t}'(s)=0$, then
 \begin{equation}\label{e:3}
 \begin{aligned}
&\calR_1(\tilde{z}'(s)) + \twodis(-\rmD_z \calI (\tilde{t}(s),\tilde{z}(s)), \partial\calR_1(0)) )\| \tilde{z}'(s)\|_{L^2(\Omega)}
\\
&\leq
- \int_\Omega (1+ |\nabla \tilde{z}(s)|^2)^{\frac{q-2}{2}} \nabla \tilde{z}(s) \cdot \nabla \tilde{z}'(s) \,\dd x
 - \int_{\Omega}   \rmD_z \wt{\calI}(\tilde{t}(s),\tilde{z}(s))\tilde{z}'(s) \,\dd x.
\end{aligned}
\end{equation}
As a consequence, for every $w\in\calZ$
\begin{equation}\label{t'=0}
\begin{aligned}
 & \calR_1(w) -  \calR_1(\tilde{z}'(s)) \geq \pairing{L^2(\Omega)}{\rmD_z \calI (\tilde{t}(s),\tilde{z}(s)) +\eta(s)}{w-\tilde{z}'(s)}
   \\ &
  -  \int_\Omega (1+ |\nabla \tilde{z}(s)|^2)^{\frac{q-2}{2}} \nabla \tilde{z}(s) \cdot (\nabla w - \nabla \tilde{z}'(s)) \,\dd x
 -
\langle \rmD_z \wt{\calI}(\tilde{t}(s),\tilde{z}(s)), w-\tilde{z}' (s)\rangle_{W^{1,2}(\Omega)} ,
 \end{aligned}
\end{equation}
where $\eta(s)\in \partial\calR_1(0)$ is such that
\begin{equation}\label{e:eta}
\twodis(-\rmD_z \calI (\tilde{t}(s),\tilde{z}(s)), \partial\calR_1(0)) )=\|-\rmD_z \calI (\tilde{t}(s),\tilde{z}(s)) -\eta(s)\|_{L^2(\Omega)}.
\end{equation}
\end{enumerate}
\end{proposition}
\noindent
Observe that, in view of Notation \ref{not:abuse} we could replace the duality pairings on the right-hand sides of
\eqref{t'>0} and \eqref{t'=0}
  by $\int_{\Omega}   \rmD_z \wt{\calI}(\tilde{t}(s),\tilde{z}(s))( w-\tilde{z}' (s)) \,\dd x$.
\begin{proof}
We differentiate \eqref{parametrized-energy-ineq} w.r.t.\ time and get
 $\foraa\ s\in(0,S)$
\begin{equation}
\label{e:start}
\begin{aligned}
&\mathcal{M}_0 (\tilde{t}'(s),\tilde{z}'(s),   \twodis(-\rmD_z \calI
(\tilde{t}(s),\tilde{z}(s)), \partial\calR_1(0)) )\leq
      -\frac{\dd}{\dd s}\calI(\tilde{t}(s),\tilde{z}(s)) +  \partial_t  \calI(\tilde{t}(s),\tilde{z}(s))\tilde{t}'(s)
   \\ & =  - \int_\Omega (1+ |\nabla \tilde{z}(s)|^2)^{\frac{q-2}{2}} \nabla \tilde{z}(s) \cdot \nabla \tilde{z}'(s) \,\dd x
 - \int_{\Omega}   \rmD_z \wt{\calI}(\tilde{t}(s),\tilde{z}(s))\tilde{z}'(s) \,\dd x,
   \end{aligned}
\end{equation}
where the second equality follows from the parameterized chain rule \eqref{chain-rule-param}.
Now, according to the definition  \eqref{e:defg0} of $\mathcal{M}_0$
we distinguish between two cases.

If $\tilde{t}'(s)>0$, then \eqref{e:start} yields
\begin{equation}\label{e:1}
\begin{aligned}
\calR_1(\tilde{z}'(s)) +& I_0( \twodis(-\rmD_z \calI (\tilde{t}(s),\tilde{z}(s)), \partial\calR_1(0)) )
\\&\leq
- \int_\Omega (1+ |\nabla \tilde{z}(s)|^2)^{\frac{q-2}{2}} \nabla \tilde{z}(s) \cdot \nabla \tilde{z}'(s) \,\dd x
 - \int_{\Omega}   \rmD_z \wt{\calI}(\tilde{t}(s),\tilde{z}(s))\tilde{z}'(s) \,\dd x.
 \end{aligned}
\end{equation}
Thus $\twodis(-\rmD_z \calI (\tilde{t}(s),\tilde{z}(s)), \partial\calR_1(0))=0$, so that
 $-\rmD_z \calI (\tilde{t}(s),\tilde{z}(s))\in\partial\calR_1(0)\subset\calZ^*$ which implies that
 for every $w\in\calZ$
 \begin{equation}\label{e:2}
 \calR_1(w)\geq \pairing{\calZ}{-\rmD_z \calI (\tilde{t}(s),\tilde{z}(s))}{w}.
 \end{equation}
Adding \eqref{e:1} and \eqref{e:2} we get \eqref{t'>0}.

 If $\tilde{t}'(s)=0$, then from \eqref{e:start} together with \eqref{e:defg0}
we deduce \eqref{e:3}.
Let now $\eta\in \partial\calR_1(0)$ as in \eqref{e:eta}.
Then, for every $w\in\calZ$ there holds $\calR_1(w)\geq \pairing{\calZ}{\eta}{w}$ which, together with \eqref{e:3}
(upon adding and subtracting $\pairing{\calZ}{\rmD_z \calI
  (\tilde{t}(s),\tilde{z}(s))}{w}$ on the right-hand side)
provides~\eqref{t'=0}.
\end{proof}

\subsubsection*{Acknowledgment}
We would like to thank Giuseppe Savar\'e for some useful suggestions
about the proof of Theorem \ref{thm:chain-rule}.
We are also grateful to Karoline Disser and Jens Griepentrog for
enlightening discussions on function spaces.

 D.K.\ and C.Z.\ are grateful for the kind hospitality of the Section of Mathematics of DICATAM, University of Brescia, where part of this work was performed. R.R.\ and C.Z.\
 acknowledge the kind hospitality of the Weierstrass Institute for Applied Analysis and Stochastics, where part of this
 research was carried out.

This project was partially supported by the
2012-PRIN project ``Calcolo delle Variazioni'', by GNAMPA (Indam) and
the Deutsche Forschungsgemeinschaft
   through the project C32 ``Modeling of Phase Separation and Damage
   Processes in Alloys'' of
  the Research Center MATHEON. 


 \appendix
\section{Young measure tools}
\label{s:a-1}
 \noindent In this section, we provide a minimal aside on
  Young measures
with values in infinite-dimensional  spaces (see   e.g.\
\cite{Balder84, Ball89, Valadier90});  in particular, we shall
focus on Young measures with values in a reflexive Banach space
$V$. The main result here, Theorem \ref{thm.balder-gamma-conv}  below,
is extension to weak topologies
of  the so-called {\em
  Fundamental Theorem of Young measures}, see e.g.\
\cite[Thm.\,1]{Balder84},
\cite{Ball89},
\cite[Thm.\,16]{Valadier90}.

Preliminarily, we establish some basic notation and definitions:
We  denote by $\mathscr{L}_{(0,T)}$
the $\sigma$-algebra of the Lebesgue measurable subsets of  the interval $(0,T)$ and,
given
   a reflexive Banach space $V$,
 by $\mathscr B(V)$ its Borel $\sigma$-algebra.
We  use the symbol $\otimes$ for product $\sigma$-algebrae.

We consider the space $V$ endowed   with the \emph{weak}
topology, and say that a $\mathscr{L}_{(0,T)} \otimes \mathscr
B(V)$--measurable functional $\mathcal{H}: (0,T) \times
V \to (-\infty,+\infty]$ is a \emph{weakly-normal
integrand}  if for a.a. $t \in (0,T)$ the map
\begin{equation}
\label{def:wws}
\begin{gathered}
 \text{$w \mapsto  \calH (t,w)$
is sequentially lower semicontinuous on $V$ w.r.t.\ the
weak topology.}
\end{gathered}
\end{equation}
We   denote by $\mathscr{M} (0,T; V)$ the set of all
$\mathscr{L}_{(0,T)}$-measurable functions $y: (0,T) \to
V$.
 \begin{definition}[\bf (Time-dependent) Young measures]
  \label{parametrized_measures}
  A \emph{Young measure} in the space $V $
  is a family
  $\mmu:=\{\mu_t\}_{t \in (0,T)} $ of Borel probability measures
  on $ V $
  such that the map on $(0,T)$
\begin{equation}
\label{cond:mea} t \mapsto \mu_{t}(B) \quad \mbox{is}\quad
{\mathscr{L}_{(0,T)}}\mbox{-measurable} \quad \text{for all } B \in
\mathscr{B}(V).
\end{equation}
We denote by $\mathscr{Y}(0,T; V)$ the set of all Young
measures in $V $.
\end{definition}

We are now  in the position of recalling the following compactness result, which was proved in
\cite[Thm.\,3.2]{RossiSavare06} (see also \cite[Thm.\,4.2]{Stef08?BEPD} and \cite[Thm.\ A.2]{mrs2013}).
\begin{theorem}
\label{thm.balder-gamma-conv}
Let $1 \leq p \leq \infty$ and let  $(w_n) \subset L^p
(0,T;V)$ be a bounded sequence.  Then,
  there exists a subsequence $(w_{n_k})$ and
  a Young  measure
  $ \mmu=\{ \mu_{t} \}_{t \in (0,T)}  \in \mathscr{Y}(0,T;V)$
  such that for a.a.\ $t\in (0,T)$
  \begin{equation}
\label{e:concentration}
\begin{gathered}
  \mbox{$ \mu_{t} $ is
      concentrated on
      the set
      $ L(t):=
      \bigcap_{l=1}^{\infty}
      \overline{\big\{w_{n_k}(t)\,: \ k\ge  l
      \big\}}^{\mathrm{weak}}$}
      \end{gathered}
  \end{equation}
of the limit points of the sequence $(w_{n_k}(t))$ with respect to
the weak topology of $V$
  and,  for every weakly-normal
integrand
  $\mathcal H
: (0,T) \times V\to (-\infty,+\infty]$ such that  the
sequence $t \mapsto \mathcal{H}^- (t,w_{n_k}(t))$
is uniformly
integrable ($\mathcal{H}^-$ denoting the negative part of $\mathcal{H}$), there holds
\begin{equation}
\label{heq:gamma-liminf} \liminf_{k \to \infty} \int_0^T
\mathcal{H}
 (t,w_{n_k}(t)) \,\dd t \geq
\int_0^T\int_{V} \mathcal{H} (t,w)\,\dd \mu_t(w)  \,\dd
t\,.
\end{equation}
As a consequence,
 setting
$\rmw(t):=\int_{V} w \,\dd \mu_t (w)$  $\foraa\, t
\in (0,T)\,$, 
there holds
\begin{equation}
  \label{eq:35}
w_{n_k} \weakto \rmw \ \ \text{ in $L^p (0,T;V)$},
\end{equation}
with $\weakto$ replaced by $\weaksto$ if $p=\infty$.
\end{theorem}


\end{document}